\documentclass[12pt,twoside]{amsart}
\usepackage{amssymb,amsmath,amsfonts,amsthm}
\usepackage{mathrsfs}
\usepackage[X2,T1]{fontenc}
\usepackage[applemac]{inputenc}
\usepackage{cite}
\usepackage{latexsym}
\def\gam{\textbf{\textrm{SG}}^{\tau, \infty}_{\mu,\nu, s}(\R^{2n})}
\def\pig{\Pi_{\mu,\nu,s}^{\tau,\infty}(\R^{3n})}
\def\Im{\mathop{\rm Im}\nolimits}
\def\Re{\mathop{\rm Re}\nolimits}

\def\Im{\mathop{\rm Im}\nolimits}
\def\Re{\mathop{\rm Re}\nolimits}

\def\R{\mathbb R}

\def\N{\mathbb N}
\def\Z{\mathbb Z}
\def\ds{\displaystyle}

\newcommand\dslash{d\llap {\raisebox{.9ex}{$\scriptstyle-\!$}}}
\usepackage{colortbl}

\newcommand{\beqsn}{\arraycolsep1.5pt\begin{eqnarray*}}
\newcommand{\eeqsn}{\end{eqnarray*}\arraycolsep5pt}
\newcommand{\beqs}{\arraycolsep1.5pt\begin{eqnarray}}
\newcommand{\eeqs}{\end{eqnarray}\arraycolsep5pt}

\newtheorem{Th}{Theorem}[section]
\newtheorem{Rem}[Th]{Remark}
\newtheorem{Ex}[Th]{Example}
\newtheorem{Lemma}[Th]{Lemma}
\newtheorem{Def}[Th]{Definition}
\newtheorem{Prop}[Th]{Proposition}

\def\pxi{\langle \xi \rangle}
\def\px{\langle x \rangle}
\def\fas{e^{i \langle x-y, \xi \rangle}}
\def\pd{\langle D \rangle}
\def\der{\partial_{\xi}^{\alpha}\partial_{x}^{\beta}}
\catcode`\@=11

\renewcommand{\section}%
   {\setcounter{equation}{0}\@startsection {section}{1}{\z@}{-3.5ex plus -1ex
  minus -.2ex}{2.3ex plus .2ex}{\Large\bf}}

\title{Schr\"odinger-type equations in Gelfand-Shilov spaces}
\author[Ascanelli]{Alessia Ascanelli}
\address{Alessia Ascanelli\\
Dipartimento di Matematica ed Informatica\\Universit\`a di Ferrara\\
Via Machiavelli 30\\
44121 Ferrara\\
Italy}
\email{alessia.ascanelli@unife.it}
\author[Cappiello]{Marco Cappiello}
\address{Marco Cappiello\\
Dipartimento di Matematica ``G. Peano" \\Universit\`a di Torino\\
Via Carlo Alberto 10\\
10123 Torino\\
Italy}
\email{marco.cappiello@unito.it}

\begin{document}

\def\thefootnote{}
\footnote{ \textit{2010 Mathematics Subject Classification}: 35Q41, 35S05, 46F05 \\
\textit{Keywords and phrases}: Schr\"odinger type equations, Gelfand-Shilov spaces, infinite order pseudodifferential operators}

\begin{abstract}
We study the initial value problem for Schr\"odinger-type equations with initial data presenting a certain Gevrey regularity and an exponential behavior at infinity. We assume the lower order terms of the Schr\"odinger operator depending on $(t,x)\in[0,T]\times\R^n$ and complex valued. Under a suitable decay condition as $|x|\to\infty$ on the imaginary part of the first order term and an algebraic growth assumption on the real part, we derive global energy estimates in suitable Sobolev spaces of infinite order and prove a well posedness result in Gelfand-Shilov type spaces. We also discuss by examples the sharpness of the result.
\end{abstract}

\maketitle

\markboth{\sc Schr\"odinger equation in Gelfand-Shilov classes}{\sc A.~Ascanelli, M. Cappiello}

\noindent
R\'ESUM\'E.
Nous \'etudions le probl\`eme de Cauchy pour des \'equations de Schr\"odinger avec des donn\'ees initiales pr\'esentent une r\'egularit\'e de type Gevrey et un comportement exponentiel a l'infini. Nous supposons les coefficients des terms d'ordre inferieur d\'ependant de $(t,x)\in[0,T]\times\R^n$ et \`a valeurs complexes. Sous une condition de d\'ecroissance convenable pour $|x| \to \infty$ pour la partie imaginaire  du terme du premier ordre et une hypot\`ese de croissance alg\'ebrique sur la partie r\'eelle, nous obtenons des estimations de l'\'energie globales dans des espaces de Sobolev d'ordre infini et prouvons que le probl\`eme de Cauchy est bien pos\'e dans des espaces de type Gelfand-Shilov. Nous discutons \'egalement par des exemples de l'optimalit\'e du r\'esultat.

\section{Introduction}
We consider for $(t,x)\in[0,T]\times\R^n$ the Cauchy problem 
\beqs
\label{CP}
\begin{cases}
P(t,x,\partial_t,\partial_x)u(t,x)=f(t,x)\cr
u(0,x)=g(x) 
\end{cases}
\eeqs
where
\begin{eqnarray}\label{P}
P(t,x,\partial_t,\partial_x)&=&\partial_t-i\triangle_x +\sum_{j=1}^{n}a_j(t,x)\partial_{x_j}+b(t,x)
\\
\nonumber
&=&\partial_t-i\triangle_x +A(t,x,\partial_x)
\end{eqnarray}
is a partial differential operator with complex valued coefficients $a_j$ and $b$, and $\triangle_x=\sum_{j=1}^n\partial_{x_j}^2$, as usual. The equation $Pu=f$ with $P$ in \eqref{P} is known in literature as a ``Schr\"odinger-type equation''.


 It is well-known that when the coefficients $a_j, b$ are real valued, smooth and uniformly bounded the Cauchy problem \eqref{CP} is $L^2$-well-posed, while if $a_j$ are complex valued suitable decay conditions for $|x|\to\infty$ are needed on the imaginary part of the coefficients in order to obtain either $L^2$ or $H^{\infty}$-well posedness (see \cite{I1} and a recent generalization \cite{ABZ3} to evolution equations with evolution degree $p\geq 2$). In  \cite{D, I2, M2, KB} convenient assumptions on the decay of the imaginary parts of the coefficients $a_j$ are given in order to obtain well-posedness of \eqref{CP} in $L^2(\R^n)$ and in $H^\infty(\R^n)$ as well as in the uniform Gevrey classes $\gamma^s(\R^n)$ of all functions $ f \in C^\infty(\R^n)$ such that
\begin{equation}\label{Gevrey}
\sup_{x \in \R^n} \sup_{\alpha \in \N^n}C^{-|\alpha|}(\alpha!)^{-s}|\partial_x^\alpha f(x)|<\infty
\end{equation}
for some $C>0$, see \cite{KB}. According to the decay in $x$ assumed for $\Im a_j(t,x),$ we may register or not a loss of derivatives of the solution with respect to the initial data. Also the more general situation where the term $\Delta_x$ in \eqref{P} is replaced by a term of the form $a(t)\Delta_x$ for a possibly vanishing real valued coefficient $a(t)$ has been recently considered, see \cite{ACic, ACC, CRe1, CReisraele} and the references therein.

A wide literature extending these kind of results to the so-called $p$-\textit{evolution equations} (i.e. anisotropic evolution equations with evolution degree $p\geq 2$) with real characteristics  has been recently developed, see for instance 
\cite{AB, ABbis, ABZ,AB2, CC} and the references therein.
Despite the precise requirements on the decay at infinity of $\textrm{Im}\,  a_j(t,x)$, all the above mentioned results do not give any information on the behavior of the solution for $|x| \to \infty$. 
This is mainly due to the fact that, apart from the precise assumption on the decay of $
\textrm{Im}\, a_j(t,x)$ for $|x| \to \infty$, the lower order terms in \eqref{P} are assumed to belong to the standard H\"ormander classes. Namely, $\textrm{Re}\, a_j$ and the derivatives of $a_j$ are simply assumed to be uniformly bounded in $(t,x)$.

 
In the recent paper \cite{scncpp4} we have considered in the case $n=1$ a class of $p$-evolution operators including \eqref{P} as a particular case, and proved that the Cauchy problem \eqref{CP}, \eqref{P} is well-posed in the Schwartz spaces $\mathscr{S}(\R^n)$, $\mathscr{S}'(\R^n)$ by deriving an energy estimate for the solution in the weighted Sobolev spaces
\begin{equation}\label{weightedSobolev}
H^{m}(\R^n)=\left\{u \in \mathscr{S}'(\R^n) \vert\ \|u \|_{H^m}=\|\px^{m_{2}}\pd^{m_{1 }}u\| <\infty\right\}, \quad m=(m_1,m_2),
\end{equation}
where $\|\cdot\|$ stands for the $L^2$-norm and we denote by $\pd^{m_1}$ the Fourier multiplier with symbol $\pxi^{m_1}:=(1+|\xi|^2)^{m_1/2}$. Notice that for $m_{2}=0$ we recapture the standard Sobolev spaces and that the following identities hold:
\begin{equation}\label{S}
\bigcap_{m \in \R^2}H^{m}(\R^n)= \mathscr{S}(\R^n), \qquad \qquad \bigcup_{m \in \R^2 }H^{m}(\R^n)= \mathscr{S}'(\R^n).
\end{equation}
In particular, for a Schr\"odinger operator of the form
$$P(t,x,\partial_t,\partial_x)=  \partial_t-i\,\partial^2_x + a(t,x)\partial_{x}+b(t,x),$$
assuming that
\begin{equation}\label{decayaj}
|\partial_x^\beta a(t,x)| \leq C_\beta \px^{-1-|\beta|}, \qquad \beta \in \N, \, x \in \R
\end{equation}
and 
$$
|\partial_x^\beta b(t,x)| \leq C_\beta \px^{-|\beta|} \qquad \beta \in \N, \, x \in \R 
$$
for some positive constant $C_\beta$, we proved that there exists $\delta>0$ such that for all $m=(m_{1},m_{2}) \in \R^2$,
$f\in C([0,T];H^{m}(\R))$ and $g\in H^{m}(\R)$ there is a unique solution
$u\in C([0,T];H^{(m_1,m_2-\delta)}(\R))$ which satisfies the following energy
estimate:
\beqs
\label{E}
\|u(t,\cdot)\|^2_{H^{(m_1,m_2-\delta)}}\leq C\left(\|g\|^2_{H^m}+
\int_0^t\|f(\tau,\cdot)\|^2_{H^m} \,d\tau\right)\qquad
\forall t\in[0,T],
\eeqs
for some $C=C(m)>0$.
\\
The main novelty in the result above is the existence of a unique solution with the same regularity as the Cauchy data but with a different behavior at infinity (either a loss of decay or an increase in growth depending on the sign of $m_2$). The result  has been proved using pseudodifferential operators in the \textbf{SG} classes, which are defined as follows. Given $m=(m_1,m_2)\in\R^2$ we define
 $$\textbf{\textrm{SG}}^{m}(\R^{2n})=\{ p(x,\xi) \in C^{\infty}(\R^{2n})\ \vert\sup_{(x,\xi) \in \R^{2n}}\pxi^{-m_{1}+|\alpha|}
\px^{-m_{2}+|\beta|}|\der p(x,\xi)| <\infty,\ \forall \alpha, \beta \in \N^n \}.$$
These classes have been defined in \cite{Co, Pa} and employed in a large number of papers involving differential operators with polynomially growing coefficients, see for instance \cite{AsCa1, AsCa2, scncpp4, CPP2, CPP3, Co}.


In the present paper we want to adapt the techniques used in \cite{scncpp4} to study the problem \eqref{CP} in arbitrary dimension $n \geq 1$ and assuming a weaker condition on the behavior at infinity of the imaginary parts of the coefficients of the lower order terms, namely

\beqsn
|\partial_x^\beta ( \Im a_j)(t,x)|\leq \ds C_\beta\px^{-\sigma-|\beta|} \qquad (t,x)\in[0,T]\times\R^n, \beta \in \N^n, \quad 1\leq j\leq n
\eeqsn
and 
\beqsn
|\partial_x^\beta (\Re a_j)(t,x)|+|\partial_x^\beta b(t,x)|\leq \ds C_\beta \px^{1-\sigma-|\beta|} , \qquad (t,x)\in[0,T]\times\R^n, \beta \in \N^n, \quad 1\leq j\leq n,
\eeqsn
for some $\sigma \in (0,1)$, with $C_\beta = C^{|\beta|+1} \beta!^{s_o}$ for some $s_o >1$ and some $C>0$ independent of $\beta$. In fact, the choice of $\sigma \in (0,1)$ leads in general to study the problem \eqref{CP} in Gevrey-type spaces, cf. \cite{CRe1, KB}, hence it is natural to assume Gevrey regularity of the $a_j$ and $b$, that is $C_\beta$ as above. We are going to prove that the Cauchy problem \eqref{CP} \eqref{P} admits a unique solution in suitable Gelfand-Shilov classes, see Theorem \ref{main} here below.
\\
More to the point, we recall that, fixed $s >1, \theta>1$ the Gelfand-Shilov space $\mathcal S_{s}^\theta(\R^n)$ is defined as the space of all functions $f \in C^{\infty}(\R^n)$ satisfying the condition
\begin{equation}\label{GSexp}
\sup_{x \in \R^n} \sup_{\alpha, \beta \in \N^n}A^{-|\alpha|}B^{-|\beta|}(\alpha!)^{-\theta}(\beta!)^{-s}|x^\beta \partial_x^\alpha f(x)|<\infty 
\end{equation}
for some positive constants $A,B$ independent of $\alpha, \beta$, or the equivalent condition
\begin{equation}\label{GSexp2}
\sup_{x \in \R^n} \sup_{\alpha \in \N^n}C^{-|\alpha|}(\alpha!)^{-\theta}e^{\epsilon |x|^{\frac1{s}}}| \partial_x^\alpha f(x)|<\infty 
\end{equation}
for some positive $C, \epsilon$. This space has been introduced in the book \cite{GS}. Later on, a projective version, denoted by $\Sigma_s^\theta(\R^n),$ has been defined by assuming \eqref{GSexp2} to hold for every $C, \epsilon >0$, cf. \cite{Pi}. Functional properties and different characterizations of these spaces and of their dual spaces were then studied in \cite{CCK, Mit, Pi, PPV}. Comparing \eqref{GSexp} with \eqref{Gevrey} we notice that $\mathcal{S}_s^\theta(\R^n)$ is a subset of $\gamma^\theta(\R^n)$ with an additional condition on the behavior at infinity. These spaces possess convenient properties for what concerns the action of Fourier transform and this makes them a suitable functional setting for pseudodifferential operators of infinite order, namely with symbols $a(x,\xi)$ admitting an exponential growth at infinity. Thanks to these properties,  Gelfand-Shilov spaces have been employed in the study of the Cauchy problem for hyperbolic equations, cf. \cite{AsCa2, CPP2, CPP3}. More recently, some papers treating Schr\"odinger equations with real valued coefficients in these spaces appeared, see \cite{CN, CW}. In the case of lower order terms with complex coefficients, the choice of Cauchy data in Gelfand-Shilov spaces intersects another field of investigation, namely the study of the smoothing effect produced by exponential decay of the data on the Gevrey regularity of the solution to the Schr\"odinger equation, see \cite{K,KT, Mizu}. The aim and the method used in the latter works is considerably different from ours since the authors approach the problem \eqref{CP} from a microlocal point of view by use of wave front sets and their results are formulated as pointwise statements with respect to $t>0$. No global energy estimates on $[0,T]$ appear in the above mentioned works. 
\\
\indent
The main goal in this paper is to derive such estimates and to identify a functional setting in which the Cauchy problem \eqref{CP} is well posed. With this purpose, we need to consider other Gelfand-Shilov type spaces, namely considering functions which are Gevrey regular but may grow exponentially at infinity, that is admitting \eqref{GSexp2} to hold for negative $\epsilon$. In fact, pointwise estimates of $u(t,x)$ given in \cite{K} for $t>0$ show that the solution may present an exponential growth as $|x|\to\infty$. In order to introduce these new spaces and to derive energy estimates in this new functional setting, it is convenient to introduce a suitable scale of related weighted Sobolev spaces. Namely, fixed $s >1, \theta >1$, $m = (m_1, m_2) \in \R^2$, $\rho= (\rho_1, \rho_2)\in \R^2 ,$  we set
$$
H^m_{\rho,s,\theta}(\R^n)=\left\{u \in \mathscr{S}'(\R^n) \vert\ \left\| u\right\|_{H^m_{\rho,s,\theta}}:=\left\| \Pi_{m,\rho,s,\theta}u \right\| <\infty\right\},
$$
where $\Pi_{m,\rho,s,\theta}$ denotes the operator 
\begin{equation}\label{pigr}
\Pi_{m,\rho,s,\theta}=\langle \cdot \rangle^{m_2} \langle D\rangle^{m_1} \exp( \rho_2 \langle \cdot \rangle^{1/s}) \exp( \rho_1 \langle D\rangle^{1/\theta}).
\end{equation}
The spaces $H^m_{\rho,s,\theta}(\R^n)$ are Hilbert spaces endowed with the inner product $(\cdot, \cdot)_{H^m_{\rho,s,\theta}}$. Notice that for $\rho=(0,0)$ they reduce to the Sobolev spaces $H^m(\R^n)$ defined by \eqref{weightedSobolev}.
Moreover, we have that $H^m_{\rho, s,\theta}(\R^n) \subset H^{m'}_{\rho', s',\theta'}(\R^n)$ if $s \leq s', \theta \leq \theta', m_j \geq m'_j, \rho_j \geq \rho'_j, j=1,2$. By Riesz theorem the dual space $((H^m_{\rho, s,\theta}(\R^n))' $ coincides with $H^{-m}_{-\rho, s,\theta}(\R^n)$ for every $m,\rho \in \R^2, s>1, \theta >1.$ 
Notice that we have
$$\mathcal{S}^\theta_s (\R^n) =  \bigcup_{\stackrel{\rho \in \R^2}{\rho_j>0,j =1,2}} H^m_{\rho,s,\theta}(\R^n) \qquad \textit{and} \qquad
\Sigma_s^\theta (\R^n) = \bigcap_{\rho_j>0,j =1,2}H^m_{\rho,s,\theta}(\R^n)$$
for all $m \in \R^2$.

Now we can also introduce the spaces
$$\widetilde{\mathcal{S}}_s^\theta(\R^n) = \bigcup_{\rho_1>0, \rho_2 \in \R}H^m_{\rho,s,\theta}(\R^n)$$ and
$$\widetilde{\Sigma}_s^\theta(\R^n) = \bigcap_{\rho_1>0, \rho_2 \in \R}H^m_{\rho,s,\theta}(\R^n).$$

Given these preliminaries, the main result of the paper reads as follows.

\begin{Th}\label{main}
Let $s_0>1, \sigma \in (0,1)$ such that $s_0 < 1/(1-\sigma)$ and let $P(t,x,\partial_t,\partial_{x})$ be an operator of the form \eqref{P} with $a_j$ and $b$ continuous with respect to $t$ and satisfying for all $(t,x)\in[0,T]\times\R^n$, $\beta \in \N^n$ and $1\leq j\leq n$ the following conditions:
\beqs
\label{imm}
|\partial_x^\beta ( \Im a_j)(t,x)|&\leq& \ds C^{|\beta|+1}\beta!^{s_0}\px^{-\sigma-|\beta|}, 
\\
\label{re}
|\partial_x^\beta ( \Re a_j)(t,x)|&\leq& \ds C^{|\beta|+1}\beta!^{s_0}\px^{1-\sigma-|\beta|}
\\
\label{reb}
|\partial_x^\beta  b(t,x)|&\leq& \ds C^{|\beta|+1}\beta!^{s_0} \px^{1-\sigma-|\beta|}, 
\eeqs
for some positive constant $C$ independent of $\beta$.
Let moreover $f\in C([0,T]; H^m_{\rho, s,\theta}(\R^n))$ and $g\in H^m_{\rho, s,\theta}(\R^n)$ for some $s \in (s_0, 1/(1-\sigma)), \theta >s_0$ and $\rho=(\rho_1,\rho_2),m=(m_1,m_2) \in \R^2$. Then there exists $\bar \delta=\bar \delta (s, \rho_2)>0$ such that the Cauchy problem \eqref{CP} admits a unique solution $u\in C([0,T];H^m_{(\rho_1,\rho_2-\bar \delta), s,\theta}(\R^n))$
which satisfies:
\begin{equation}\label{enest}
\|u(t)\|^2_{H^m_{(\rho_1, \rho_2-\bar \delta),s,\theta}}\leq C_s\left(\|g\|_{H^m_{\rho,s,\theta}}^2+\int_0^t \|f(\tau)\|_{H^m_{\rho,s,\theta}}^2 d\tau \right),
\end{equation}
for  $t\in [0,T]$ and for some $C_s>0$. In particular, the Cauchy problem \eqref{CP} is well posed in $\widetilde{\mathcal{S}}^\theta_s(\R^n)$ and in $\widetilde{\Sigma}^\theta_s(\R^n)$.
\end{Th}

\begin{Rem}
We notice that the loss $\bar \delta$ depends in general on $s$ and on $\rho_2$ whereas it is independent of $\theta$ and $\rho_1.$ This means that we obtain a solution with the same (Gevrey) regularity as the initial data but with a \emph{worse} behavior at infinity. In particular, as it will be clear from the proof in Section 4, in general $\rho_2 -\bar \delta$ may be negative, which means that the solution may admit an exponential growth for $|x| \to \infty$ even if the data decay to $0$ exponentially. This phenomenon had been already observed in \cite{K,KT}. 
We also notice that, thanks to the exponential decay of the data, we do not have any loss of Gevrey regularity as in \cite{CRe1, CReisraele, KB}. Moreover, we stress the fact that the coefficients of the lower order terms in our paper are not uniformly bounded as in \cite{CRe1, CReisraele, KB} but they can admit an algebraic growth in $x$; on the other hand, in order to use the $\textrm{\textbf{SG}}$ calculus, we need to assume also some conditions on the behavior of the derivatives of $a_j$ and $b$, which however are quite natural. 
\end{Rem}

\begin{Rem} By the technique used in the present paper we can show that the result of Theorem \ref{main} still holds in the critical case $s=1/(1-\sigma)$, but only locally in time, that is the solution $u(t,x)$ satisfiying \eqref{enest} exists in general only on a small enough subinterval $[0,T^\ast]\subseteq [0,T]$, see Remark \ref{critical} below for further details.
\end{Rem}

\begin{Rem}
The modification of the behavior at infinity of the solution stated in Theorem \ref{main} is not only a consequence of the method used; in Section 5 of the present paper we give some examples of solutions to \eqref{CP} which really present a loss of decay or an increase of growth with respect to the data $f, g.$ The same examples show that the critical threshold $s=1/(1-\sigma)$ in Theorem \ref{main} is sharp.
\end{Rem}
Summing up, Theorem \ref{main} evidentiates a new phenomenon occurring in the Cauchy problem for Schr\"odinger-type equations which is alternative to the loss of regularity found in the previous literature, cf. \cite{CRe1, CReisraele, KB}, and which appears when we fix precise assumptions on the behavior at infinity of the Cauchy data. It would be interesting to inquire if  it is possible to obtain intermediate results between Theorem \ref{main} and the result in  \cite[Theorem 1.1]{KB}, that is a solution which presents a loss in the spaces $H^m_{(\rho_1, \rho_2), s,\theta}(\R^n)$ both with respect to $\rho_1$ and $\rho_2$. At present this issue remains a challenging open problem for the authors.

The paper is organized as follows. In Section  \ref{sec2} we present some aspects of the theory of Gelfand-Shilov spaces and of pseudodifferential operators of infinite order. We stress the fact that the specific calculus for these operators is new and interesting per se. Nevertheless, in order to address the reader as soon as possible to the proof of Theorem \ref{main}, we decided to postpone and detail the calculus in an Appendix at the end of the paper. In Section \ref{sec3} we introduce the change of variable needed to reduce the Cauchy problem \eqref{CP} to a new problem for which we can derive an energy estimate in Section \ref{sec4}. From this we can go back to our original problem and prove Theorem \ref{main}. The paper is completed by a discussion of the critical case $s=1/(1-\sigma)$ in Remark \ref{critical} and by Section 5, where we give explicit examples showing the phenomena stated in Theorem \ref{main} and we discuss the optimality of the results.

\section{Function spaces and pseudodifferential operators of infinite order} \label{sec2}
In the next sections we will need some continuity properties and composition theorems for pseudodifferential operators of infinite order. In this section we only state the crucial properties and the theorems that we need in Section 4. The proofs of these technical statements are reported in the Appendix at the end of the paper, where we detail the complete calculus.\\
 In the following, given a Hilbert space $H$, we shall denote by $( \cdot, \cdot )_H$ and by $\| \cdot \|_H$ the inner product and the corresponding norm on $H$. 
 We will occasionally use the notation $e_1(1,0)$, $e_2=(0,1)$ for the vectors of the canonical basis of $\R^2$.

\subsection{Function spaces}
Before introducing pseudodifferential operators of infinite order, we recall some basic properties of Gelfand-Shilov spaces  and of the Sobolev spaces $H^m_{\rho,s,\theta}(\R^n)$ defined in the Introduction. First of all we observe that, fixed $s >0, \theta >0, A >0, B>0$ we denote by $\mathcal{S}_{s,B}^{\theta, A}(\R^n)$ the Banach space of all functions $f \in C^\infty(\R^n)$ such that 
\begin{equation}\label{GSnorm}\| f \|_{s,\theta,A, B}: = \sup_{\alpha, \beta \in \N^n} \sup_{x \in \R^n} \frac{|x^\beta \partial^\alpha f(x)|}{A^{|\alpha|}B^{|\beta|}(\alpha!)^\theta(\beta!)^s} < \infty 
\end{equation} endowed with the norm $\| \cdot \|_{s,\theta,A, B}$. We have
the following relations
$$\mathcal{S}_{s}^{\theta}(\R^n) = \bigcup_{A>0,B>0 } \mathcal{S}_{s,B}^{\theta, A}(\R^n),\qquad\Sigma^{\theta}_{s}(\R^n)= \bigcap_{A>0, B>0}\mathcal{S}_{s,B}^{\theta, A}(\R^n).$$
It is immediate to observe that the following inclusions hold:
$$\Sigma_s^\theta(\R^n) \hookrightarrow \mathcal{S}^\theta_s(\R^n) \hookrightarrow \Sigma_{s+\varepsilon}^{\theta+\varepsilon}(\R^n) $$
 $$(\Sigma_{s+\varepsilon}^{\theta+\varepsilon})'(\R^n)  \hookrightarrow (\mathcal{S}_s^\theta)'(\R^n) \hookrightarrow (\Sigma_s^\theta)'(\R^n)$$
for every $\varepsilon >0$.
In the sequel of the paper we shall need to consider two other Gelfand-Shilov spaces whose elements satisfy intermediate estimates between those of $\Sigma^\theta_s(\R^n)$ and those of $\mathcal{S}^\theta_s(\R^n)$. These spaces do not appear in the main Theorem \ref{main} but are instrumental to set up the calculus for our pseudodifferential operators.
Namely, we set
$$\mathscr{S}_{s}^{\theta}(\R^n)= \bigcup_{A>0}\bigcap_{B>0} \mathcal{S}_{s,B}^{\theta, A}(\R^n)$$
and 
$$\widetilde{\mathscr{S}}_{s}^{\theta}(\R^n)= \bigcap_{A>0}\bigcup_{B>0} \mathcal{S}_{s,B}^{\theta, A}(\R^n).$$
In terms of exponential estimates \eqref{GSexp2}, the elements of $\mathscr{S}_{s}^{\theta}(\R^n)$ (resp. $\widetilde{\mathscr{S}}_{s}^{\theta}(\R^n)$) satisfy \eqref{GSexp2} for some $C>0$ and for every $\epsilon>0$ (resp. for some $\epsilon>0$ and for every $C>0$). We also observe that 
$$ \Sigma^{\theta}_{s}(\R^n) \subset \mathscr{S}_{s}^{\theta}(\R^n) \subset \mathcal{S}_{s}^{\theta}(\R^n)$$ and 
$$ \Sigma^{\theta}_{s}(\R^n) \subset \widetilde{\mathscr{S}}_{s}^{\theta}(\R^n) \subset \mathcal{S}_{s}^{\theta}(\R^n).$$
According to their definition the spaces $\widetilde{\mathscr{S}}_{s}^{\theta}(\R^n)$ and $\mathscr{S}_{s}^{\theta}(\R^n)$ can be equipped with a natural topology starting from the Banach spaces $\mathcal{S}_{s,B}^{\theta, A}(\R^n)$ and considering inductive limit with respect to one of the constants $A$ or $B$ and the projective limit with respect to the other or viceversa.
In the sequel we shall assume $s >1, \theta>1$ since all our results hold only under this condition. 
In the following we shall denote by $(\mathcal{S}^\theta_s)'(\R^n)$, $(\Sigma^\theta_s)'(\R^n),$ $(\mathscr{S}_{s}^{\theta})'(\R^n)$, $(\widetilde{\mathscr{S}}_{s}^{\theta})'(\R^n)$ the dual spaces of $\mathcal{S}^\theta_s(\R^n)$, $\Sigma^\theta_s(\R^n)$, $\mathscr{S}_{s}^{\theta}(\R^n)$, $\widetilde{\mathscr{S}}_{s}^{\theta}(\R^n)$.  

The spaces $\mathscr{S}_{s}^{\theta}(\R^n)$ and $\widetilde{\mathscr{S}}_{s}^{\theta}(\R^n)$ can also be expressed in terms of the Sobolev spaces $H^m_{\rho,s,\theta}(\R^n)$. Namely, denoting $\rho=(\rho_1, \rho_2)$ we have:
$$\mathscr{S}_{s}^{\theta}(\R^n) = \bigcup_{\rho_1 >0}\bigcap_{\rho_2 >0}H^m_{\rho,s,\theta}(\R^n),$$
$$\widetilde{\mathscr{S}}_{s}^{\theta}(\R^n) = \bigcap_{\rho_1 >0}\bigcup_{\rho_2 >0}H^m_{\rho,s,\theta}(\R^n)$$
for every $m \in \R^2.$
Concerning the action of Fourier transform $\mathscr{F}$ we have the following isomorphisms:
$$\mathscr{F}: \mathcal{S}_s^\theta(\R^n) \rightarrow \mathcal{S}_\theta^s(\R^n),$$
$$\mathscr{F}: \Sigma_s^\theta(\R^n) \rightarrow \Sigma_\theta^s(\R^n),$$
$$\mathscr{F}: \mathscr{S}_s^\theta(\R^n) \rightarrow \widetilde{\mathscr{S}}_\theta^s(\R^n),$$
$$\mathscr{F}: \widetilde{\mathscr{S}}_s^\theta(\R^n) \rightarrow \mathscr{S}_\theta^s(\R^n)$$
and analogous mapping properties on the dual spaces. Moreover $\mathscr{F}$ maps continuously $H^m_{\rho, s,\theta}(\R^n)$ into $H^{m'}_{\rho', \theta,s}(\R^n)$ where $m'=(m_2,m_1)$ and $\rho' =(\rho_2,\rho_1).$

\subsection{Pseudodifferential operators of infinite order}

\begin{Def}
Fixed $C >0, c >0$, let $\mu, \nu, s, \tau \in \R$ such that $1< \mu \leq s$ and $\nu >1$. We shall denote by $\textbf{\textrm{SG}}_{\mu,\nu,s}^{\tau,\infty}(\R^{2n};C, c)$ the Banach space of all functions $a(x,\xi) \in C^\infty(\R^{2n})$ satisfying the following estimates:
\begin{equation}\label{doubleinfinite}
\sup_{\alpha, \beta \in \N^n} \sup_{(x,\xi) \in \R^{2n}}
C^{-|\alpha|-|\beta|} (\alpha !)^{-\mu}(\beta!)^{-\nu} \pxi^{-\tau+|\alpha|} \px^{|\beta|}  \exp\left[-c(|x|^{\frac{1}{s}}) \right] \left| \der a(x,\xi) \right| < +\infty.
\end{equation}
We set $\textbf{\textrm{SG}}_{\mu,\nu, s}^{\tau,\infty}(\R^{2n}) = \lim\limits_{\stackrel{\longrightarrow}{C,c \to \infty}}\textbf{\textrm{SG}}_{\mu,\nu, s}^{\tau,\infty}(\R^{2n};C, c)$ endowed with the inductive limit topology.
\end{Def}

In the following we shall also consider symbols with finite orders.

\begin{Def}
Fixed $C >0, m=(m_1, m_2) \in \R^2$, we shall denote by $\textbf{\textrm{SG}}_{\mu,\nu}^m(\R^{2n};C)$ the Banach space of all functions $a(x,\xi) \in C^\infty(\R^{2n})$ satisfying the following estimates:
\begin{equation}\label{doublefinite}
\sup_{\alpha, \beta \in \N^n} \sup_{(x,\xi) \in \R^{2n}}
C^{-|\alpha|-|\beta|} (\alpha !)^{-\mu}(\beta!)^{-\nu} \pxi^{-m_1+|\alpha|} \px^{-m_2+|\beta|}  \left| \der a(x,\xi) \right| < +\infty
\end{equation}
and set $\textbf{\textrm{SG}}_{\mu,\nu}^{m}(\R^{2n}) = \lim\limits_{\stackrel{\longrightarrow}{C \to \infty}}\textbf{\textrm{SG}}_{\mu,\nu}^{m}(\R^{2n};C)$.
\end{Def}
Obviously we have $\textbf{\textrm{SG}}_{\mu,\nu}^{m}(\R^{2n}) \subset \textbf{\textrm{SG}}_{\mu,\nu,s}^{m_1, \infty}(\R^{2n})$ for every $s >1.$ In the case $\mu=\nu,$ we shall denote by $\textbf{\textrm{SG}}_{\mu, s}^{\tau,\infty}(\R^{2n})$ and $\textbf{\textrm{SG}}_{\mu}^{m}(\R^{2n})$ the classes $\textbf{\textrm{SG}}_{\mu,\mu, s}^{\tau,\infty}(\R^{2n})$ and $\textbf{\textrm{SG}}_{\mu,\mu, s}^{\tau,\infty}(\R^{2n})$ respectively.
\\

For a symbol $a \in \textbf{\textrm{SG}}_{\mu, \nu, s}^{\tau,\infty}(\R^{2n})$ we can consider the pseudodifferential operator 
\begin{equation}\label{pseudo}
a(x,D)u(x)= \int_{\R^n} e^{i \langle x,\xi \rangle} a(x,\xi) \hat {u}(\xi) \, \dslash \xi, \qquad u \in \mathscr{S}(\R^n),
\end{equation}
where we denote $\dslash \xi = (2\pi)^{-n} d\xi$ and where $\langle\ ,\ \rangle$ stands for the scalar product in $\R^n.$
\\

The first step is to analyse the continuity properties of pseudodifferential operators on the infinite order Sobolev spaces defined above.
First of all, we recall that any operator of the form \eqref{pseudo} with symbol in $\textbf{\textrm{SG}}_{\mu, \nu}^{m'}(\R^{2n})$  is continuous on $\mathscr{S}(\R^n)$ and it extends to a continuous map on $\mathscr{S}'(\R^n)$ and from $H^{m}(\R^n)$ into $H^{m-m'}(\R^n)$ for every $m \in \R^2$, cf. \cite{Co, Pa}.
Moreover, if $a \in  \textbf{\textrm{SG}}_{\mu, \nu}^{m'}(\R^{2n})$, then $a(x,D)$ is continuous from $H^m_{\rho, s,\theta}(\R^n)$ to $H^{m-m'}_{\rho, s,\theta}(\R^n)$ for every $m,\rho \in \R^2$ and if $\min\{ s,\theta\} >\mu+\nu-1$, cf. Theorem \ref{doubleconj}.
Concerning operators of infinite order, we have the following result, see the Appendix for the proof.

\begin{Prop}\label{gencontinforder}
Let $\mu, \nu, \tau, s \in \R$ with $1< \mu < s, \nu >1$ and let $a \in \textbf{\textrm{SG}}_{\mu,\nu,s}^{\tau, \infty}(\R^{2n})$. Then $a(x,D)$ is linear and continuous on  $\mathscr{S}_{s}^\theta(\R^n)$ and it extends to a continuous map on $(\mathscr{S}_{s}^\theta)'(\R^n)$ for every $\theta \geq \nu$.
\end{Prop}

Apart from these general results, in the following we shall focus on operators with symbol $e^{\lambda(x,\xi)}$, that is
\begin{equation} \label{explambda}
e^{\lambda}(x,D)u(x) = \int_{\R^n} e^{i \langle x, \xi \rangle} e^{  \lambda(x,\xi)} \hat{u}(\xi) \dslash \xi
\end{equation}
for some real valued symbol $\lambda \in \textbf{\textrm{SG}}^{(0,1/s)}_{\mu}(\R^{2n}).$  By Proposition \ref{appaelambda} we have 
$e^{\lambda(x,\xi)} \in  \textbf{\textrm{SG}}^{0, \infty}_{\mu}(\R^{2n}).$ We have the following result, see the Appendix for the proof.

\begin{Prop}\label{continforder}
Let $\mu,s,\theta \in \R$ such that $\mu >1, \min\{ s,\theta\} >2\mu-1,$ and let $\lambda \in \textbf{\textrm{SG}}^{(0,1/s)}_{\mu}(\R^{2n}).$ Then, for every $\rho,m \in \R^2,$  the operator $e^{\lambda}(x,D)$ is continuous from $H^m_{\rho,s,\theta}(\R^n)$ into $H^m_{\rho-\delta e_2,s,\theta}(\R^n)$ for every $\delta> C(\lambda):= \sup\limits_{(x,\xi) \in \R^{2n}} \frac{\lambda(x,\xi)}{\px^{1/s}}$.
\end{Prop}

Under slightly stronger assumptions on $\lambda$ we shall prove in the sequel that $e^{\lambda}(x,D)$ is invertible and express the inverse in terms of operators of the form
\begin{equation}
\label{reverseop}
^R \hskip-1pt e^{ -\lambda}u(x)= \iint \fas e^{ -\lambda(y,\xi)} u(y)\, dy \dslash \xi,
\end{equation}
usually called \textit{reverse} operators, cf. \cite{KB, KN}) and defined as oscillatory integrals. It is then convenient to introduce a class of operators including both \eqref{explambda} and \eqref{reverseop}. Details on this class are again postponed to the Appendix at the end of the paper.

\begin{Def}
Let $\mu, \nu,s, \tau  \in \R$ such that $1<\mu \leq s, \nu >1$. We denote by $\Pi_{\mu,\nu,s}^{\tau,\infty}(\R^{3n})$ the space of all functions $a(x,y,\xi) \in C^\infty(\R^{3n})$ such that
$$\left| \der \partial_y^\gamma a(x,y,\xi)\right|  \leq C^{|\alpha+\beta+\gamma|+1} (\alpha!)^{-\mu}(\beta! \gamma!)^{-\nu} \pxi^{\tau-|\alpha|} \langle (x,y) \rangle^{-|\beta+\gamma|}\langle x-y \rangle^{|\beta+\gamma|} e^{c(|x|^{\frac{1}{s}}+|y|^{\frac{1}{s}})}
$$
for every $\alpha, \beta, \gamma \in \N^n, (x,y,\xi) \in \R^{3n}$ and for some positive constants $C, c$ independent of $\alpha, \beta, \gamma$, where $\langle (x,y)\rangle=\left(1+\sum_{j=1}^n(x_j^2+y_j^2)\right)^{1/2}$. We write $\Pi_{\nu,s}^{\tau,\infty}(\R^{3n})$ for $\Pi_{\nu,\nu,s}^{\tau,\infty}(\R^{3n})$.
\end{Def}

It is easy to show that if $a \in \Pi_{\mu,\nu,s}^{\tau,\infty}(\R^{3n})$, then the function $p(x,\xi)= a(x,x,\xi) \in \textbf{\textrm{SG}}_{\mu,\nu,s}^{\tau, \infty}(\R^{2n}).$ On the other hand, if $p \in \textbf{\textrm{SG}}_{\mu,\nu,s}^{\tau,\infty}(\R^{2n})$, then for every $t \in [0,1]$ the function 
$a(x,y,\xi) = p((1-t)x+t y, \xi)$ belongs to $\Pi_{\mu,\nu,s}^{\tau,\infty}(\R^{3n})$.
\\
To every $a \in \Pi_{\mu,\nu,s}^{\tau,\infty}(\R^{3n})$ it is associated an operator of the form
$$Au(x) = \iint e^{i \langle x-y,\xi \rangle} a(x,y,\xi) u(y) \, dy \dslash \xi, \qquad u \in \mathscr{S}_s^\theta(\R^n),$$
defined as standard as an oscillatory integral, namely:
\begin{equation}
\label{opampl}
Au(x) = \lim_{\delta  \to 0}\iint e^{i \langle x-y,\xi \rangle}\chi(\delta y, \delta \xi) a(x,y,\xi) u(y) \, dy \dslash \xi, \qquad u \in \mathscr{S}_s^\theta(\R^n),
\end{equation}
for some $\chi \in \mathcal{S}_\kappa^\kappa(\R^{2n}),$ with $\chi(0,0)=1, \kappa =\min\{\mu,\nu\}.$

By Theorem \ref{reduction} and Proposition \ref{continforder} $A$ acts continuously on $\mathscr{S}^\theta_s(\R^n)$ and on $(\mathscr{S}^\theta_s)'(\R^n)$ for $\mu+\nu-1 <\min\{s,\theta\}$. This fact applies in particular to the operator $^R \hskip-1pt e^{ \lambda}$ if $\lambda \in \textbf{\textrm{SG}}^{(0,1/s)}_{\mu}(\R^{2n})$. As a matter of fact, in this case, by Proposition \ref{appaelambda}, the amplitude $e^{ \lambda(y,\xi)}$ is an element of $\Pi_{\mu,s}^{0,\infty}(\R^{3n})$. Moreover, $^R \hskip-1pt e^{\lambda}$ is in fact the $L^2$-adjoint of $e^{ \lambda}(x,D).$ Hence it maps continuously $H^m_{\rho,s,\theta}(\R^n)$ into $H^m_{\rho-\delta e_2,s,\theta}(\R^n)$ if $\min\{s,\theta\} >2\mu-1$ and $\delta > C(\lambda).$ 

We observe that the classes $\textbf{\textrm{SG}}^{\tau, \infty}_{\mu,\nu,s}(\R^{2n}), \textbf{\textrm{SG}}^{m}_{\mu,\nu}(\R^{2n}), \Pi^{\tau, \infty}_{\mu,\nu,s}(\R^{3n}) $ can be equivalently defined by replacing the weight functions $\pxi$ by $\pxi_h = (h^2+|\xi|^2)^{1/2}$ for some fixed $h \geq 1$. All the previous results can be reformulated in the new notation since it does not modify the classes.
However, under this small modification, we can prove the following composition result, which in turn, choosing $h$ large enough, implies the invertibility of the operators $e^{ \lambda}(x,D)$ as a by-product. 

\begin{Prop} \label{compwithrev}
Let $\lambda \in C^\infty(\R^{2n})$ satisfy the condition 
\begin{equation}\label{questa}
|\der \lambda(x,\xi)| \leq C^{|\alpha+\beta|+1}(\alpha!\beta!)^\mu \pxi_h^{-|\alpha|}\px^{\frac{1}{s}-|\beta|}
\end{equation}
with $\mu>1, s >2\mu-1$.  Then we have 
$$e^{\lambda}(x,D) \circ  ^R \hskip-1pt e^{-\lambda} =I + r_1(x,D),$$
$$^R \hskip-1pt e^{- \lambda}  \circ e^{\lambda}(x,D)  =I + r_2(x,D),$$
where the symbols $r_1, r_2$ satisfy for every $\gamma,\delta \in \N^n, (x,\xi) \in \R^{2n}$
\begin{equation} \label{stimainvertibilita}
|\partial_\xi^\gamma \partial_x^\delta r_{k}(x,\xi) | \leq C_{\gamma \delta} \pxi_h^{-1-|\gamma|} \px^{-1+\frac{1}{s} -|\delta|}, \qquad k=1,2.
\end{equation}
In particular, $r_1,r_2 \in \textbf{\textrm{SG}}^{(0,0)}(\R^{2n})$.

\end{Prop}

From \eqref{stimainvertibilita}, we deduce in particular that
$$ |\partial_\xi^\gamma \partial_x^\delta r_{k}(x,\xi) | \leq C_{\gamma \delta}  h^{-1}\pxi_h^{-|\gamma|} \px^{ -|\delta|}, \qquad k=1,2.$$
Hence $r_k(x,D), k=1,2$ are bounded operators on $L^2(\R^n)$ with norm $\|r_k (\cdot, D) \|$ which can be taken as small as we want by enlarging $h$.
Then it turns out that if $h \geq h_o$ with $h_o$ sufficiently large, the operator
$e^{\lambda}(x,D)$ is invertible with inverse given in terms of a suitable Neumann series.

The previous result can be generalized to the following conjugation theorem which will play a crucial role in the sequel.

\begin{Th} \label{conjugationthm}
Let $p \in \textbf{\textrm{SG}}_{\mu,\nu}^m (\R^{2n})$ and $\lambda$ satisfying the condition \eqref{questa} with $s >\mu+\nu-1$. Then, there exists $h_o\geq 1$ such that if $h \geq h_o$, we have
\begin{equation}
\label{conjugformula}
e^{\lambda}(x,D) p(x,D) (e^{\lambda}(x,D))^{-1} = p(x,D) +q(x,D) +r(x,D)+r_0(x,D) ,
\end{equation}
where $r \in \textbf{\textrm{SG}}_{\mu,\nu}^{(m_1 -2, m_2- 2(1-1/s) )}(\R^{2n}), r_0 \in \mathcal{S}_{\mu+\nu-1}(\R^{2n})$ and
\begin{equation} \label{qexpression}
q(x,\xi)= \sum_{|\alpha|=1} \partial_\xi^\alpha p(x,\xi)  (i\partial_x)^\alpha \Lambda(x,\xi) + \sum_{|\beta|=1}  D_x^\beta p(x,\xi) \partial_{\xi}^\beta \Lambda(x,\xi)
\end{equation}
belongs to $\textbf{\textrm{SG}}_{\mu,\nu}^{\left(m_1 -1, m_2- 1+1/s\right) }(\R^{2n}).$ 
\end{Th}


\section{Change of variable}\label{sec3}
Theorem \ref{conjugationthm} is needed to derive energy estimates for a possible solution $u$ of \eqref{CP} which cannot be obtained directly. In fact, looking for an energy estimate in $L^2$ for $u$ and arguing as standard, we have by \eqref{P}:  
\beqsn
\ds\frac d {dt} \|u(t)\|^2&=&2\Re\langle \ds\frac d {dt} u(t), u(t)\rangle
\\
&=&2\Re\langle f(t), u(t)\rangle+2\Re\langle(i\Delta_x)u(t),u(t)\rangle
- 2\Re\langle A(t)u(t), u(t)\rangle
\\
&=&2\Re\langle f(t), u(t)\rangle- \langle (A+A^\ast)(t)u(t), u(t)\rangle.
\eeqsn
Now it is easy to notice that the operator $(A+A^\ast)(t,x,D)$ is not $L^2$-bounded since its principal symbol is $-2\sum_{j=1}^n\Im a_j(t,x){\xi_j}\in \textbf{\textrm{SG}}^{(1,-\sigma)}(\R^{2n})$. The idea is then to perform first a change of variable of the form
\beqs\label{change}
v(t,x)=e^{\Lambda}(t,x,D)u(t,x),
\eeqs 
where $e^{\Lambda}(t,x,D)$ denotes a pseudodifferential operator with symbol $e^{\Lambda(t,x,\xi)}$ and 
$e^{\Lambda}(t,x,D)$ is in\-ver\-ti\-ble with inverse $(e^\Lambda(t,x,D))^{-1}.$
In this way we are reduced to consider the auxiliary Cauchy problem 
\beqs
\label{CP2}
\begin{cases}
P_\Lambda(t,x,D_t,D_x)v(t,x)=e^\Lambda(t,x,D)f(t,x)=:f_\Lambda(t,x)\cr
v(0,x)=e^\Lambda(0,x,D)g(x)=:g_{\Lambda}(x) 
\end{cases}
\eeqs
in the unknown $v,$ where 
$$P_\Lambda:=e^{\Lambda}(t,x,D)P(t,x,\partial_t, \partial_x)(e^\Lambda(t,x,D))^{-1}.$$
We shall take $\Lambda(t,x,\xi)$ of the form
\begin{equation}\label{formofLambda}
\Lambda(t,x,\xi)= k(t) \px_h^{1-\sigma}+ \lambda(x,\xi)
\end{equation}
for a suitable symbol $\lambda \in \textrm{\textbf{SG}}^{(0,1/s)}_\mu (\R^{2n})$, for $h$ large enough, and for some positive non-increasing function $k\in C^1[0,T]$ which will be chosen later on. Our aim is to prove that, choosing properly $\lambda(x,\xi)$ and $k(t)$ we may have 
\begin{equation}
\label{PLambda}
P_\Lambda=e^{\Lambda}(t,x,D)P(t,x,\partial_t, \partial_x)(e^\Lambda(t,x,D))^{-1}(x,D)=\partial_t-i\Delta_x+A_\Lambda(t,x,D)+\tilde r_0(t,x,D)
\end{equation}
where $\tilde r_0$ is an operator of order $(0,0)$ and $A_\Lambda$ is an operator (of order $(1,1/s)$) such that $A_\Lambda+A_\Lambda^*$ is a positive operator on $L^2(\R^n)$. This would imply 
that the Cauchy problem \eqref{CP2} is $L^2-$well posed, since the corresponding energy estimate can be obtained from the inequality
\beqs\label{enworks}
\ds\frac d {dt} \|v(t)\|^2&=&2\Re\langle f_\Lambda(t), v(t)\rangle-\langle (A_\Lambda+A_\Lambda^\ast)(t)v(t), v(t)\rangle-2\Re\langle \tilde r_0(t)v(t), v(t)\rangle
\\
\nonumber
&\leq & C\left(\|f_\Lambda(t)\|^2+\|v(t)\|^2\right)
\eeqs
that now works for a positive constant $C$. 
\\
We observe that, taking $\Lambda$ of the form \eqref{formofLambda}, we have
\beqsn
\partial_t-i\Delta_x=(\partial_t-i\Delta_x)e^\Lambda(e^\Lambda)^{-1}=(k'(t) \px_h^{1-\sigma}e^\Lambda +e^\Lambda\partial_t-i\Delta_x e^\Lambda)(e^\Lambda)^{-1}
\eeqsn
but
\beqsn \label{conjLap}
\Delta_xe^{\Lambda}&=& \sum_{j=1}^n\partial_{x_j}(e^\Lambda\partial_{x_j}\Lambda+ e^\Lambda\partial_{x_j})
=\sum_{j=1}^n\left(e^\Lambda (\partial_{x_j}\Lambda)^2+e^\Lambda\partial_{x_j}^2\Lambda+2e^\Lambda(\partial_{x_j}\Lambda)\partial_{x_j}\right)+ e^\Lambda\Delta_x, 
\eeqsn
so that we have
\beqsn
\partial_t-i\Delta_x=k'(t) \px_h^{1-\sigma}+ e^\Lambda\left(\partial_t-i\Delta_x-i\sum_{j=1}^n\left((\partial_{x_j}\Lambda)^2+\partial_{x_j}^2\Lambda+2(\partial_{x_j}\Lambda)\partial_{x_j}\right)\right)(e^\Lambda)^{-1}
\eeqsn
and so
\beqs\label{conjdelta}
\nonumber e^\Lambda\left( \partial_t-i\Delta_x\right)(e^\Lambda)^{-1}&=&\partial_t-i\Delta_x -k'(t) \px_h^{1-\sigma}+ i e^\Lambda
\sum_{j=1}^n\left((\partial_{x_j}\Lambda)^2+\partial_{x_j}^2\Lambda+2(\partial_{x_j}\Lambda)\partial_{x_j}\right)(e^\Lambda)^{-1}
\\
&=&\partial_t-i\Delta_x-k'(t) \px_h^{1-\sigma}- 2\sum_{j=1}^n(\partial_{x_j}\Lambda)D_{x_j}+ r'_0(t,x,D)
\eeqs
with $r'_0(t,x,D)$ a term of order $(0,0)$. By Theorem \ref{conjugationthm} 
the operator $P_\Lambda$ is of the form \eqref{PLambda} with 
\begin{equation}\label{biscio}A_\Lambda:=i\ds\sum_{j=1}^na_j(t,x)D_{x_j}-2\ds\sum_{j=1}^n(\partial_{x_j}\Lambda)D_{x_j}-k'(t) \px_h^{1-\sigma}+ r(t,x,D)+b(t,x),
\end{equation}
where $r\in C([0,T], {\rm{\bf{SG}}}^{(0,1-\sigma)}_{\mu, s_0}(\R^{2n}))$. 
The idea is then to choose $\Lambda$ such that the symbol of the positive order part of $A_\Lambda+A_\Lambda^\ast$  is positive, that is
\begin{equation}\label{fine!}
2\sum_{j=1}^n(\partial_{x_j}\Lambda)\xi_j+\ds\sum_{j=1}^n\Im a_j(t,x)\xi_j+k'(t)\px_h^{1-\sigma}-\Re r(t,x,\xi)-\Re b(t,x) \leq 0
\end{equation}
so that \eqref{enworks} is satisfied by application of the sharp G{\aa}rding inequality.
\\

Our idea is to choose $\lambda$ and $k(t)$ such that the term $2\sum_{j=1}^n(\partial_{x_j}\Lambda)\xi_j$ compensates the term $\sum_{j=1}^n \Im a_j\xi_j$ and $k'(t)\px_h^{1-\sigma}$  controls the term $\Re r(t,x,\xi)$ and $\Re b(t,x)$ in order to get \eqref{fine!}. 
This is the reason of the choice \eqref{formofLambda}.

\medskip
The sequel of this section is devoted to the construction of the function $\lambda(x,\xi)$.
Since the behavior of $\Im a_j$ is described in \eqref{imm}, to control $-\sum_{j=1}^n \Im a_j\xi_j$ it is convenient to look for a function $\lambda_1(x,\xi)$ such that
\beqs\label{eq1}
\ds\sum_{j=1}^n(\partial_{x_j}\lambda_1)\xi_j=|\xi|g_1(x), \quad x\in\R^n,
\eeqs
where $g_1(x)=M\px^{-1+1/s}$ with $M>0, h \geq 1$ to be chosen later on.
It is easy to check that the desired function $\lambda_1$ is given by:
\beqs\label{lambda1}
\lambda_1(x,\xi)=\ds\int_0^{x\cdot\omega} g_1(x-\tau\omega)d\tau, \qquad \omega=\xi/|\xi|.
\eeqs
As a matter of fact we have
\beqsn
\ds\sum_{j=1}^n(\partial_{x_j}\lambda_1)\xi_j&=&\ds\sum_{j=1}^ng_1(x-(x\cdot\omega)\omega)\frac{\xi_j^2}{|\xi|}
+\ds\sum_{j=1}^n \ds\int_0^{x\cdot\omega} \partial_{x_j}g_1(x-\tau\omega)d\tau\ \xi_j
\\
&=&|\xi| g_1(x-(x\cdot\omega)\omega)+|\xi|\ds\int_0^{x\cdot\omega}\ds\sum_{j=1}^n\partial_{x_j}g_1(x-\tau\omega) \frac{\xi_j}{|\xi|} d\tau
\\
&=& |\xi| g_1(x-(x\cdot\omega)\omega)-|\xi|\ds\int_0^{x\cdot\omega}\frac d{d\tau}\left(g_1(x-\tau\omega)\right)d\tau
\\
&=& |\xi| g_1(x-(x\cdot\omega)\omega)- |\xi|g_1(x-(x\cdot\omega)\omega)+|\xi|g_1(x)
\\
&=&|\xi|g_1(x).
\eeqsn
The following Lemma \ref{lemma1} states the behavior of the derivatives of $\lambda_1$.
\begin{Lemma}\label{lemma1} There exists a constant $C_s$ independent of $h$ such that for every $\alpha,\beta\in\Z^n_+$ the function $\lambda_1(x,\xi)$ defined by \eqref{lambda1} satisfies the following estimates:
\beqs\label{derlambda1}
|\partial_\xi^\alpha\partial_x^\beta\lambda_1(x,\xi)|&\leq& MC_s^{|\alpha|+|\beta|+1}\alpha!\beta!\langle x\rangle^{\frac{1}{s}-|\beta|}|\xi|^{-|\alpha|},\quad (x,\xi)\in\mathcal R,
\eeqs
 where $\mathcal R=\{(x,\xi)\in\R^{2n}\vert\ |x\cdot\omega|\leq \px/2\}$.
\end{Lemma}

\begin{proof}
We divide the proof in two steps:
\\
\textbf{Step 1.}  We show that for every $\delta>0$, $\alpha,\beta\in\Z^n_+$, $k\in\N,$ the function $\langle x-t\omega\rangle^{-\delta}$ satisfies for every $0\leq |t|\leq |x\cdot\omega|$ and $(x,\xi)\in\mathcal R$ the following estimate:
\beqs\label{stimlangle}
|\partial_\xi^\alpha\partial_x^\beta\partial_t^k\langle x-t\omega\rangle^{-\delta}|\leq C_\delta^{|\alpha|+|\beta|+k+1}\alpha!\beta!k!\langle x\rangle^{-\delta-|\beta|-k}|\xi|^{-|\alpha|}.
\eeqs
We first notice that on $\mathcal{R}$ we have
\beqs\label{kab}
|t|\leq |x\cdot\omega| \Rightarrow \langle x-t\omega\rangle\geq \langle x\rangle/2.
\eeqs
Indeed, for $|t|\leq |x\cdot\omega|$ and on $\mathcal R$, the inequality $|x|\leq |x-t\omega|+ |t\omega|\leq |x-t\omega|+\langle x\rangle/2$ holds, so $\langle x\rangle^2=h^2+|x|^2\leq h^2+2|x-t\omega|^2+\langle x\rangle^2/2$, which gives $\langle x\rangle^2/2\leq h^2+2|x-t\omega|^2\leq 2\langle x-t\omega\rangle^2$, and so \eqref{kab} holds.

Now, if $|\alpha+\beta|+k=0$, the estimate \eqref{stimlangle} trivially holds:
\beqsn
|\langle x-t\omega\rangle^{-\delta}|\leq 2^\delta\langle x\rangle^{-\delta}=C_\delta\langle x\rangle^{-\delta}.
\eeqsn
Let us suppose so to have \eqref{stimlangle} for every $|\alpha+\beta|+k\leq\ell$, $\ell\in\N$. To conclude it is sufficient to check the three following items:
\beqs\label{item3}
&&|\partial_\xi^\alpha\partial_x^{\beta}\partial_t^{k+1}\langle x-t\omega\rangle^{-\delta}|\leq C_\delta^{|\alpha|+|\beta|+k+2}\alpha!\beta!(k+1)!|\xi|^{-|\alpha|}\langle x\rangle^{-\delta-|\beta|-k-1}.\\\label{item1}
&&|\partial_\xi^\alpha\partial_x^{\beta+e_j}\partial_t^k\langle x-t\omega\rangle^{-\delta}|\leq C_\delta^{|\alpha|+|\beta|+k+2}\alpha!(\beta+e_j)!k!|\xi|^{-|\alpha|}\langle x\rangle^{-\delta-|\beta|-1-k},\quad 1\leq j\leq n,
\\\label{item2}
&& |\partial_\xi^{\alpha+e_j}\partial_x^\beta\partial_t^k\langle x-t\omega\rangle^{-\delta}|\leq C_\delta^{|\alpha|+|\beta|+k+2}(\alpha+e_j)!\beta!k!|\xi|^{-|\alpha|-1}\langle x\rangle^{-\delta-|\beta|-k},\quad1\leq j\leq n.
\eeqs
We first compute
\beqsn
\partial_t\langle x-t\omega\rangle^{-\delta}&=&-\delta\langle x-t\omega\rangle^{-\delta-2}\sum_{j=1}^n(x_j-t\omega_j)(-\omega_j)=\delta\langle x-t\omega\rangle^{-\delta-2}(x\cdot\omega-t)
\\
\partial_{x_j}\langle x-t\omega\rangle^{-\delta}&=& -\delta\langle x-t\omega\rangle^{-\delta-2}(x_j-t\omega_j)
\\
\partial_{\xi_j}\langle x-t\omega\rangle^{-\delta}&=&-\frac{\delta}{2}\langle x-t\omega\rangle^{-\delta-2}\partial_{\xi_j}\left(\sum_{i=1}^n(x_i-t\omega_i)^2\right)=\delta t\langle x-t\omega\rangle^{-\delta-2}\sum_{i=1}^n(x_i-t\omega_i)\partial_{\xi_j}\omega_i
\\
&=&\delta t\langle x-t\omega\rangle^{-\delta-2}p_j(t,x,\xi),
\eeqsn
where $$p_j(t,x,\xi)  =\left((x_j-t\omega_j)|\xi|^{-1}+\sum_{i=1}^n(x_i-t\omega_i)\xi_i\partial_{\xi_j}|\xi|^{-1}\right).$$

Then:
\beqs\nonumber
\partial_\xi^\alpha\partial_x^{\beta}\partial_t^{k+1}\langle x-t\omega\rangle^{-\delta}&=&\partial_\xi^\alpha\partial_x^{\beta}\partial_t^k\left(\delta\langle x-t\omega\rangle^{-\delta-2}(x\cdot\omega-t)\right)
\\\label{0} 
&=&\delta\hskip-0.4cm\sum_{\alpha_1+\alpha_2=\alpha}\sum_{\beta_1+\beta_2=\beta}\binom\alpha{\alpha_1}\binom\beta{\beta_1}\left(\partial_\xi^{\alpha_1}\partial_x^{\beta_1}\partial_t^k\langle x-t\omega\rangle^{-\delta-2}\right)\partial_\xi^{\alpha_2}\partial_x^{\beta_2}(x\cdot\omega-t)
\\\nonumber
&&-\delta\left(\partial_\xi^\alpha\partial_x^{\beta}\partial_t^{k-1}\langle x-t\omega\rangle^{-\delta-2}\right)
\eeqs
where the second term in the right-hand side appears only if $k \neq 0.$ Moreover
\beqs
\nonumber
\partial_\xi^{\alpha+e_j}\partial_x^\beta\partial_t^k\langle x-t\omega\rangle^{-\delta}&=&\partial_\xi^{\alpha}\partial_x^\beta\partial_t^k\left(\delta t\langle x-t\omega\rangle^{-\delta-2}p_j(t,x,\xi)\right)
\\\label{1}
&=&\delta\sum_{\alpha_1+\alpha_2=\alpha}\sum_{\beta_1+\beta_2=\beta}\sum_{k_1+k_2=k-1}\binom\alpha{\alpha_1}\binom\beta{\beta_1}\binom k{k_1}\partial_\xi^{\alpha_1}\partial_x^{\beta_1}\partial_t^{k_1}\langle x-t\omega\rangle^{-\delta-2}\cdot
\\
\nonumber&&\quad\cdot\partial_\xi^{\alpha_2}\partial_x^{\beta_2}\partial_t^{k_2}p_j(t,x,\xi)
\\\nonumber
&&+\delta t\sum_{\alpha_1+\alpha_2=\alpha}\sum_{\beta_1+\beta_2=\beta}\sum_{k_1+k_2=k}
\binom\alpha{\alpha_1}\binom\beta{\beta_1}\binom k{k_1}\partial_\xi^{\alpha_1}\partial_x^{\beta_1}\partial_t^{k_1}\langle x-t\omega\rangle^{-\delta-2}\cdot
\\\nonumber&&\quad\cdot\partial_\xi^{\alpha_2}\partial_x^{\beta_2}\partial_t^{k_2}p_j(t,x,\xi),
\eeqs
where the first term in the right-hand side appears only if $k \neq 0$, and
\beqs
\nonumber
\partial_\xi^{\alpha}\partial_x^{\beta+e_j}\partial_t^k\langle x-t\omega\rangle^{-\delta}&=&\partial_\xi^{\alpha}\partial_x^\beta\partial_t^k\left(-\delta\langle x-t\omega\rangle^{-\delta-2}(x_j-t\omega_j)\right),
\\
\label{2}
&=&-\delta\sum_{\alpha_1+\alpha_2=\alpha}\sum_{\beta_1+\beta_2=\beta}\binom\alpha{\alpha_1}\binom\beta{\beta_1}\partial_\xi^{\alpha_1}\partial_x^{\beta_1}\partial_t^{k}\langle x-t\omega\rangle^{-\delta-2}\partial_\xi^{\alpha_2}\partial_x^{\beta_2}(x_j-t\omega_j)
\\\nonumber
&&+\delta\sum_{\alpha_1+\alpha_2=\alpha}\binom\alpha{\alpha_1}\partial_\xi^{\alpha_1}\partial_x^{\beta}\partial_t^{k-1}\langle x-t\omega\rangle^{-\delta-2}\partial_\xi^{\alpha_2}\omega_j,
\eeqs
where the second term in the right-hand side appears only if $k \neq 0$. 
To estimate \eqref{0}, we compute for $\alpha,\beta\in\Z^n_+$
\beqsn
\partial_\xi^\alpha\partial_x^\beta(x\cdot\omega)=
\begin{cases}
\partial_\xi^\alpha(x\cdot\omega)=(x\cdot\xi)\partial_\xi^\alpha |\xi|^{-1}+{\ds\sum_{\ell=1}^n} x_\ell\partial_\xi^{\alpha-e_\ell,\ast}|\xi|^{-1} & \beta=0
\\
\partial_\xi^\alpha\omega_j & \beta=e_j,\quad j=1,\ldots,n
\\
0 & |\beta|\geq 2
\end{cases}
\eeqsn
where the notation $\partial^{\alpha-e_\ell,\ast}$ means that the term appears only if $\alpha_\ell > 0$, and we get
\beqs\label{derxixomega}
|\partial_\xi^\alpha\partial_x^\beta(x\cdot\omega)|\leq A_0^{|\alpha|+|\beta|+1}|\alpha|!|\beta|! |x|^{1-|\beta|}|\xi|^{-|\alpha|}.
\eeqs
Choosing properly the constant $C_s$, we easily obtain \eqref{item3} by taking the modulus in \eqref{0}, making use of formula \eqref{derxixomega}, of the inductive assumption, and of the following well known estimate (cf. \cite{KB}):
\beqs\label{KB2}
|\partial^\gamma |y|^m|\leq A_0^{|\gamma|+1}|\gamma|!|y|^{m-|\gamma|}, \quad y\in\R^n,\ n\geq 1,\ m\in \R, \gamma\in\Z^n_+.
\eeqs
To estimate \eqref{1} 
we first notice that 
\beqs\label{lepij}
\partial_\xi^\alpha\partial_x^\beta\partial_t^kp_j(t,x,\xi)=\begin{cases}
\partial_\xi^\alpha\partial_x^\beta p_j(t,x,\xi) & k=0
\\
\partial_\xi^\alpha(-\omega_j|\xi|^{-1}-|\xi|\partial_{\xi_j}|\xi|^{-1}) & k=1,\beta=0
\\
0 & k=1\ {\rm{and}}\ \beta>0,\ {\rm{or}}\ k\geq 2,
\end{cases}
\eeqs
\beqs\label{ledifferenze}
\partial_\xi^{\alpha}\partial_x^{\beta}(x_j-t\omega_j)=\begin{cases}
x_j-t\omega_j & \alpha= 0, \beta=0,
\\
-t\partial_\xi^\alpha\omega_j& \alpha> 0, \beta=0,
\\
1 & \alpha=0,\beta=e_j,
\\
0& {\rm otherwise},
\end{cases}
\eeqs
and $\partial_\xi^\alpha\omega_j=\xi_j\partial_\xi^\alpha|\xi|^{-1}+\partial_\xi^{\alpha-e_j,\ast}|\xi|^{-1}$.
Then, using again \eqref{KB2} we get
\beqs\label{stima:deromegai}
|\partial_\xi^{\alpha}\omega_j|&\leq& 2A_0^{|\alpha|+1}|\alpha|!|\xi|^{-|\alpha|},\quad \alpha\in\Z^n_+,\ j=1,\cdots,n,
\\\label{pezzetti}
|\partial_\xi^{\alpha}\partial_x^{\beta}(x_j-t\omega_j)|&\leq & A_0^{|\alpha|+|\beta|+1}|\alpha|!|\beta|!\langle x\rangle^{1-|\beta|}|\xi|^{-|\alpha|},
\eeqs
where in \eqref{pezzetti} we used the fact that $|t|\leq |x\cdot\omega|\leq\langle x\rangle /2$.
Now, from \eqref{lepij}, making use of \eqref{KB2}, \eqref{stima:deromegai} and \eqref{pezzetti}, we get, for $(x,\xi)\in\mathcal R$ and $|t|\leq |x\cdot\omega|$:
\beqs\label{cevole}
\begin{cases}
|\partial_\xi^\alpha p_j(t,x,\xi)|\leq A_0^{|\alpha|+1}|\alpha|! (|x|+|t|)|\xi|^{-1-|\alpha|}\leq A_0^{|\alpha|+1}|\alpha|!\langle x\rangle|\xi|^{-1-|\alpha|}
\\
|\partial_\xi^\alpha\partial_{x_k} p_j(t,x,\xi)|\leq A_0^{|\alpha|+1}|\alpha|!|\xi|^{-1-|\alpha|},\quad 1\leq k\leq n
\\
|\partial_\xi^\alpha\partial_{t} p_j(t,x,\xi)|\leq A_0^{|\alpha|+1}|\alpha|!|\xi|^{-1-|\alpha|}
\end{cases}
\eeqs
and $\partial_\xi^\alpha\partial_{x}^\beta\partial_t^k p_j=0$ for $k=0$ and $|\beta|\geq 2$, for $k=1$ and $\beta\neq 0$, for $k\geq2$.

In view of these considerations, if $k \geq 2$ we can split \eqref{1} into
\beqsn
\partial_\xi^{\alpha+e_j}\partial_x^\beta\partial_t^k\langle x-t\omega\rangle^{-\delta}&=&\delta\left(\partial_\xi^{\alpha}\partial_x^{\beta}\partial_t^{k-1}\langle x-t\omega\rangle^{-\delta-2}\right)p_j(t,x,\xi)
\\
&+&\delta\hskip-0.7cm\sum_{\alpha_1+\alpha_2=\alpha,\alpha_2\neq 0}\binom\alpha{\alpha_1}\partial_\xi^{\alpha_1}\partial_x^{\beta}\partial_t^{k-1}\langle x-t\omega\rangle^{-\delta-2}\partial_\xi^{\alpha_2}p_j(t,x,\xi)
\\
&+&\delta\beta_\ell \sum_{\alpha_1+\alpha_2=\alpha,\alpha_2\neq 0}\binom\alpha{\alpha_1}\sum_{\ell=1}^n\partial_\xi^{\alpha_1}\partial_x^{\beta-e_\ell, *}\partial_t^{k-1}\langle x-t\omega\rangle^{-\delta-2}
\partial_\xi^{\alpha_2}\partial_{x_\ell}p_j(t,x,\xi)
\\
&+&\delta(k-1)\sum_{\alpha_1+\alpha_2=\alpha}
\binom\alpha{\alpha_1}\partial_\xi^{\alpha_1}\partial_x^{\beta}\partial_t^{k-2}\langle x-t\omega\rangle^{-\delta-2}\partial_\xi^{\alpha_2}\partial_tp_j(t,x,\xi)
\\
&+&\delta t\left(\partial_\xi^{\alpha}\partial_x^{\beta}\partial_t^k\langle x-t\omega\rangle^{-\delta-2}\right)p_j(t,x,\xi)
\\
&+&\delta t\hskip-0.7cm\sum_{\alpha_1+\alpha_2=\alpha,\alpha_2\neq 0}\binom\alpha{\alpha_1}\partial_\xi^{\alpha_1}\partial_x^{\beta}\partial_t^k\langle x-t\omega\rangle^{-\delta-2}\partial_\xi^{\alpha_2}p_j(t,x,\xi)
\\
&+&\delta t\beta_\ell \sum_{\alpha_1+\alpha_2=\alpha,\alpha_2\neq 0}\binom\alpha{\alpha_1}\sum_{\ell=1}^n\partial_\xi^{\alpha_1}\partial_x^{\beta-e_\ell,*}\partial_t^k\langle x-t\omega\rangle^{-\delta-2}
\partial_\xi^{\alpha_2}\partial_{x_\ell}p_j(t,x,\xi)
\\
&+& k \delta t\sum_{\alpha_1+\alpha_2=\alpha}
\binom\alpha{\alpha_1}\partial_\xi^{\alpha_1}\partial_x^{\beta}\partial_t^{k-1}\langle x-t\omega\rangle^{-\delta-2}\partial_\xi^{\alpha_2}\partial_t p_j(t,x,\xi);
\eeqsn
by the inductive hypothesis and \eqref{cevole}, we thus obtain:
\beqsn
|\partial_\xi^{\alpha+e_j}\partial_x^\beta\partial_t^k\langle x-t\omega\rangle^{-\delta}|&\leq&\delta C_\delta^{|\alpha|+|\beta|+k}\alpha!\beta!(k-1)!|\xi|^{-1-|\alpha|}\langle x\rangle^{-\delta-|\beta|-k}
\\
&&+\delta|t|C_\delta^{|\alpha|+|\beta|+k+1}\alpha!\beta!k!|\xi|^{-1-|\alpha|}\langle x\rangle^{-\delta-1-|\beta|-k}
\\
&\leq&C_\delta^{|\alpha|+|\beta|+k+2}(\alpha+e_j)!\beta!k!|\xi|^{-1-|\alpha|}\langle x\rangle^{-\delta-|\beta|-k},
\eeqsn
that is \eqref{item1} when $k \geq 2$. The cases $k=0$ and $k=1$ can be treated similarly.
\\
Finally, with analogue computations we have:
\beqsn
|\partial_\xi^{\alpha}\partial_x^{\beta+e_j}\partial_t^k\langle x-t\omega\rangle^{-\delta}|&\leq& \delta|\partial_\xi^\alpha\partial_x^\beta\partial_t^k \langle x-t\omega\rangle^{-\delta-2}|\cdot|x_j-t\omega_j|
\\
&&+\delta\hskip-0.7cm\sum_{\alpha_1+\alpha_2=\alpha,\alpha_2\neq 0}\binom\alpha{\alpha_1}|\partial_\xi^{\alpha_1}\partial_x^\beta\partial_t^k\langle x-t\omega\rangle^{-\delta-2}|\cdot |t|\cdot |\partial_\xi^{\alpha_2}\omega_j|
\\
&&+\delta|\partial_\xi^\alpha\partial_x^{\beta-e_j}\partial_t^{k}\langle x-t\omega\rangle^{-\delta-2}|
\\
&&+\delta k\sum_{\alpha_1+\alpha_2=\alpha}\binom\alpha{\alpha_1}|\partial_\xi^{\alpha_1}\partial_x^\beta\partial_t^{k-1} \langle x-t\omega\rangle^{-\delta-2}||\partial_\xi^{\alpha_2}\omega_j|,
\eeqsn
where the last term in the right-hand side appears only if $k \neq 0$. Then, arguing as before, we obtain \eqref{item2}. Hence \eqref{stimlangle} is proved. 

\vskip0.2cm
\textbf{Step 2.} Set 
$$H(x,\xi, t)= \int_{0}^t \langle x-\tau \omega \rangle^{1/s-1}\, d\tau.$$
We want to estimate the derivatives of the composed function $\lambda_1(x,\xi)=H(x,\xi, x\cdot \omega).$ Formal computations give
\beqs\nonumber
\partial_x^\beta (H(x,\xi, x\cdot\omega))&=&(\partial_x^\beta H)(x,\xi,x\cdot\omega)+\ds\sum_{j=1}^n\beta_j(\partial_x^{\beta-e_j}\partial_tH)(x,\xi,x\cdot\omega)\partial_{x_j}(x\cdot\omega)
\\\nonumber
&&+\cdots+(\partial_t^{|\beta|}H)(x,\xi,x\cdot\omega)(\partial_{x_1}(x\cdot\omega))^{\beta_1}\cdots(\partial_{x_n}(x\cdot\omega))^{\beta_n}
\\\label{asd}
&=&\ds\sum_{\gamma\leq \beta}\binom{\beta}{\gamma}(\partial_x^{\beta-\gamma}\partial_t^{|\gamma|}H)(x,\xi,x\cdot\omega)\omega_1^{\gamma_1}\cdots\omega_n^{\gamma_n} 
\eeqs
from which it follows by Leibniz formula that
\beqs\nonumber
\partial_\xi^\alpha \partial_x^\beta ( H(x,\xi, x\cdot \omega)) & =& 
\sum_{\alpha_0+\ldots + \alpha_n=\alpha} \frac{\alpha!}{\alpha_0! \ldots \alpha_n!}\sum_{\delta \leq \alpha_0}\binom{\alpha_0}{\delta}\sum_{\gamma\leq \beta}\binom{\beta}{\gamma}(\partial_x^{\beta-\gamma}\partial_\xi^{\alpha_0-\delta}\partial_t^{|\gamma|+|\delta|}H)(x,\xi,x\cdot\omega) 
\\ &&\times (\partial_{\xi_n}(x\cdot \omega))^{\delta_1} \cdots (\partial_{\xi_n}(x\cdot \omega))^{\delta_n} \partial_\xi^{\alpha_1}(\omega_1^{\gamma_1})\cdots \partial_\xi^{\alpha_n}(\omega_n^{\gamma_n}) \\ \nonumber
&=& \sum_{\alpha_0+\ldots + \alpha_n=\alpha} \frac{\alpha!}{\alpha_0! \ldots \alpha_n!}\sum_{\stackrel{\gamma\leq \beta, \delta \leq \alpha_0}{\gamma+\delta \neq 0}}\binom{\alpha_0}{\delta}\binom{\beta}{\gamma}(\partial_x^{\beta-\gamma}\partial_\xi^{\alpha_0-\delta}\partial_t^{|\gamma|+|\delta|-1} \langle x-t\omega \rangle^{-1+1/s})_{|_{t=x\cdot \omega}} 
\\ \nonumber &&\times (\partial_{\xi_1}(x\cdot \omega))^{\delta_1} \cdots (\partial_{\xi_n}(x\cdot \omega))^{\delta_n} \partial_\xi^{\alpha_1}(\omega_1^{\gamma_1})\cdots \partial_\xi^{\alpha_n}(\omega_n^{\gamma_n}) 
\\ \nonumber &&+
\int_0^{x \cdot \omega}\partial_x^{\beta}\partial_\xi^{\alpha} \langle x-\tau\omega \rangle^{-1+1/s}\, d\tau.
\eeqs
Hence the estimate \eqref{stimlangle} follows from the previous estimates. We leave the details to the reader.
\end{proof}

Lemma \ref{lemma1} states that $\lambda_1$ behaves like a $\textbf{\textrm{SG}}$ symbol of order $(0, 1/s)$ only in the region $\mathcal{R}.$ 
This leads to introduce a partition of the phase space and to consider in the complementary region $|x\cdot\omega|\geq \px/2$ the function 
\beqs \label{g2}
g_2(x,\xi)=M\langle x\cdot\omega\rangle^{-1+1/s}
\eeqs
and the corresponding $\lambda_2(x,\xi)$ satisfying the condition
\beqs\label{eq2}
\ds\sum_{j=1}^n(\partial_{x_j}\lambda_2)\xi_j=|\xi|g_2(x), \quad x\in\R^n,
\eeqs
compare with \eqref{eq1}. As before we can take
\beqs\label{lambda2}
\lambda_2(x,\xi) :=\ds\int_0^{x\cdot\omega} g_2(x-\tau\omega,\xi)d\tau.
\eeqs 
Notice that 
\beqs \label{equivlambda2}
\lambda_2(x,\xi) &=& \int_0^{x\cdot\omega} \hskip-0.2cmM\langle (x-\tau\omega)\cdot\omega\rangle^{-1+1/s} d\tau
\\
\nonumber&=&\int_0^{x\cdot\omega} \hskip-0.2cmM\langle x\cdot\omega-\tau\rangle^{-1+1/s} d\tau=\int_0^{x\cdot\omega} \hskip-0.2cmM\langle z\rangle^{-1+1/s}\, dz.
\eeqs

In the next Lemma we derive suitable estimates for the function $\lambda_2$ on the complementary set of $\mathcal{R}$.

\begin{Lemma}\label{lemma2} There exists a positive constant $C_s$ independent of $h$ such that the function $\lambda_2(x,\xi)$ defined by \eqref{g2}, \eqref{lambda2} satisfies for every $\alpha,\beta\in\Z^n_+$ the following estimates  
\beqs\label{derlambda2}
|\partial_\xi^\alpha\partial_x^\beta\lambda_2(x,\xi)|&\leq& MC_s^{|\alpha|+|\beta|+1}\alpha!\beta!\langle x\rangle^{1/s-|\beta|}|\xi|^{-|\alpha|},\quad (x,\xi)\in\R^{2n}\setminus\mathcal R.
\eeqs
\end{Lemma}
\begin{proof}
We work again by induction on $|\alpha+\beta|$. If $|\alpha+\beta|=0$ the assertion is a direct consequence of \eqref{equivlambda2}: 
$|\lambda_2(x,\xi)|\leq MC_s \langle x\cdot\omega\rangle^{1/s}\leq MC_s \langle x\rangle^{1/s}$
since $1/s>0.$
Now, suppose that \eqref{derlambda2} holds for every $|\alpha|+|\beta|\leq\ell$. We have to check for every $1\leq j\leq n$:
\beqs\label{one}
|\partial_\xi^{\alpha+e_j}\partial_x^\beta\lambda_2(x,\xi)|&\leq& MC_s^{|\alpha|+|\beta|+2}(\alpha+e_j)!\beta!\langle x\rangle^{1/s-|\beta|}|\xi|^{-|\alpha|-1}
\\\label{two}
|\partial_\xi^{\alpha}\partial_x^{\beta+e_j}\lambda_2(x,\xi)|&\leq& MC_s^{|\alpha|+|\beta|+2}\alpha!(\beta+e_j)!\langle x\rangle^{-1+1/s-|\beta|}|\xi|^{-|\alpha|}.
\eeqs
We have:
\beqs\label{dxi+1}
\partial_\xi^{\alpha+e_j}\partial_x^\beta\lambda_2(x,\xi)&=&\partial_\xi^{\alpha}\partial_x^\beta\partial_{\xi_j}\int_0^{x\cdot\omega} \!\!\!M\langle z\rangle^{-1+1/s}\, dz=\partial_\xi^{\alpha}\partial_x^\beta\left(M\langle x\cdot\omega\rangle^{-1+1/s}\cdot \partial_{\xi_j}(x\cdot\omega)\right)
\\\nonumber
&=&M\sum_{\alpha_1+\alpha_2=\alpha}\sum_{\beta_1+\beta_2=\beta}\binom\alpha{\alpha_1}\binom\beta{\beta_1}
\partial_\xi^{\alpha_1}\partial_x^{\beta_1} \langle x\cdot\omega\rangle^{-1+1/s}\partial_\xi^{\alpha_2+e_j}\partial_x^{\beta_2}(x\cdot\omega)
\eeqs
and
\beqs\label{dx+1}
\hskip+0.2cm\partial_\xi^\alpha\partial_x^{\beta+e_j}\lambda_2(x,\xi)&=&\partial_\xi^{\alpha}\partial_x^\beta\partial_{x_j}\int_0^{x\cdot\omega} M\langle w\rangle^{-1+1/s}dw=\partial_\xi^{\alpha}\partial_x^\beta\left(M\langle x\cdot\omega\rangle^{-1+1/s}\cdot \partial_{x_j}(x\cdot\omega)\right)
\\
\nonumber&=&\partial_\xi^{\alpha}\partial_x^\beta\left(M\langle x\cdot\omega\rangle^{-1+1/s}\cdot \omega_j\right)
\\\nonumber
&=&M\sum_{\alpha_1+\alpha_2=\alpha}\binom\alpha{\alpha_1}\partial_\xi^{\alpha_1}\partial_x^{\beta} \langle x\cdot\omega\rangle^{-1+1/s}\partial_\xi^{\alpha_2}(\omega_j).
\eeqs
To give estimates of the derivatives here above, we need to show that for every $\alpha,\beta\in\Z^n_+$ the following formula holds:
\beqs\label{derbracketdot}
|\partial_\xi^{\alpha}\partial_x^{\beta} \langle x\cdot\omega\rangle^{-1+1/s}|\leq C_s A_0^{|\alpha|+|\beta|}\alpha!\beta! \langle x\cdot\omega\rangle^{-1+1/s-|\beta|}|\xi|^{-|\alpha|}, \qquad (x,\xi)\in\R^{2n}\setminus\mathcal R.
\eeqs
We do that by induction: if $|\alpha+\beta|=0$, \eqref{derbracketdot} is true. If it holds for every $(\alpha,\beta)$ with $|\alpha|+|\beta|\leq\ell$, we have to show that it holds also for the pairs $(\alpha+e_j,\beta)$ and $(\alpha,\beta+e_j)$, $1\leq j\leq n.$
Applying Leibniz formula and taking into account the fact that $\partial_x^\beta (x \cdot \omega)=0$ for $|\beta| >2$, we obtain:
\beqs\nonumber
\partial_\xi^{\alpha+e_j}\partial_x^\beta \langle x\cdot\omega\rangle^{-1+1/s}&=&\partial_\xi^{\alpha}\partial_x^\beta\left(\left(-1+\frac{1}{s}\right)\langle x\cdot\omega\rangle^{-3+1/s}(x\cdot\omega)\partial_{\xi_j}(x\cdot\omega)\right)
\\\nonumber
&=&\left(-1+\frac{1}{s}\right)\hskip-0.3cm\sum_{\alpha_1+\alpha_2+\alpha_3=\alpha}\sum_{\stackrel{\beta_1+\beta_2+\beta_3=\beta}{|\beta_2+\beta_3|\leq 2 }}\frac{\alpha!\beta!}{\alpha_1!\alpha_2!\alpha_3!\beta_1!\beta_2!\beta_3!}\cdot
\\\label{s}
&&\hskip+3.7cm\cdot
\partial_\xi^{\alpha_1}\partial_x^{\beta_1} \langle x\cdot\omega\rangle^{-3-1/s}\partial_\xi^{\alpha_2}\partial_x^{\beta_2}(x\cdot\omega)\partial_\xi^{\alpha_3+e_j}\partial_x^{\beta_3}(x\cdot\omega) 
\\\nonumber
\partial_\xi^{\alpha}\partial_x^{\beta+e_j} \langle x\cdot\omega\rangle^{-1+1/s}&=&\partial_\xi^{\alpha}\partial_x^\beta\left(\left(-1+\frac{1}{s}\right)\langle x\cdot\omega\rangle^{1/s-3}(x\cdot\omega)\omega_j\right)
\\\label{ss}
&=&-\left(-1+\frac{1}{s}\right)\sum_{\alpha_1+\alpha_2+\alpha_3=\alpha}\sum_{\stackrel{\beta_1+\beta_2=\beta}{|\beta_2| \leq 1}}\frac{\alpha!\beta!}{\alpha_1!\alpha_2!\alpha_3!\beta_1!\beta_2!}\cdot
\\
&&\nonumber\hskip+5.5cm\cdot
\partial_\xi^{\alpha_1}\partial_x^{\beta_1} \langle x\cdot\omega\rangle^{1/s-3}\partial_\xi^{\alpha_2}\partial_x^{\beta_2}(x\cdot\omega)\partial_\xi^{\alpha_3}\omega_j.
\eeqs
From the inductive assumption \eqref{derbracketdot}, using \eqref{derxixomega} and \eqref{stima:deromegai} we get respectively
\beqs\label{sss}
\hskip+0.5cm|\partial_\xi^{\alpha+e_j}\partial_x^\beta \langle x\cdot\omega\rangle^{-1+1/s}|
&\leq&C_s A_0^{|\alpha|+|\beta|+3}(\alpha+e_j)!\beta! |\xi|^{-|\alpha|-1}
\hskip-0.3cm\sum_{\stackrel{\beta_1+\beta_2+\beta_3=\beta}{|\beta_2+\beta_3|\leq 2}}\hskip-0.5cm \langle x\cdot\omega\rangle^{1/s-3-|\beta_1|}|x|^{2-|\beta_2+\beta_3|}
\\\label{ssss}
|\partial_\xi^{\alpha}\partial_x^{\beta+e_j} \langle x\cdot\omega\rangle^{-1+1/s}|&\leq&C_s A_0^{|\alpha|+|\beta|+3}\alpha!(\beta+e_j)! |\xi|^{-|\alpha|}
\hskip-0.3cm\sum_{\stackrel{\beta_1+\beta_2=\beta}{|\beta_2|\leq 1}}\hskip-0.2cm \langle x\cdot\omega\rangle^{1/s-3-|\beta_1|}|x|^{1-|\beta_2|}.
\eeqs
Now, in \eqref{sss}, since $|\beta_2+\beta_3|\leq 2$, then $|x|^{2-|\beta_2+\beta_3|}\leq
4 |x\cdot\omega|^{2-|\beta_2+\beta_3|}\leq 4\langle x\cdot\omega\rangle^{2-|\beta_2+\beta_3|}$, on $\R^{2n}\setminus\mathcal R$.
From \eqref{sss} we get so
 \beqsn
 |\partial_\xi^{\alpha+e_j}\partial_x^\beta \langle x\cdot\omega\rangle^{-1+1/s}|&\leq &
C_s A_0^{|\alpha|+|\beta|+1}(\alpha+e_j)!\beta! |\xi|^{-|\alpha|-1} \langle x\cdot\omega\rangle^{-1+1/s-|\beta|}.  
\eeqsn
Similarly, in \eqref{ssss} since $|\beta_2|\leq 1$, then $|x|^{1-|\beta_2|}\leq
2 |x\cdot\omega|^{1-|\beta_2|}\leq2 \langle x\cdot\omega\rangle^{1-|\beta_2|}$ on $\R^n\setminus\mathcal R$.
From \eqref{ssss} we then get 
 \beqsn
 |\partial_\xi^{\alpha}\partial_x^{\beta+e_j} \langle x\cdot\omega\rangle^{-1+1/s}|&\leq &
C_s A_0^{|\alpha|+|\beta|+1}\alpha!(\beta+e_j)!  |\xi|^{-|\alpha|} \langle x\cdot\omega\rangle^{1/s-2-|\beta|}.
\eeqsn
Formula \eqref{derbracketdot} is completely proved by induction.
\\
Coming now back to \eqref{dxi+1}, using \eqref{derbracketdot} and \eqref{derxixomega}, noting that whenever $|\beta_2|\geq 2$ the corresponding term in \eqref{dxi+1} is zero, we get
\beqsn
|\partial_\xi^{\alpha+e_j}\partial_x^\beta\lambda_2(x,\xi)|&\leq &
M\sum_{\alpha_1+\alpha_2=\alpha}\sum_{\beta_1+\beta_2=\beta}(\alpha+e_j)!\beta!C_s A_0^{|\alpha|+|\beta|+2}\langle x\cdot\omega\rangle^{1/s-1-|\beta_1|}|x|^{1-|\beta_2|}|\xi|^{-|\alpha|-1}
\\
&\leq & M(\alpha+e_j)!\beta!C_s A_0^{|\alpha|+|\beta|+2}\langle x\cdot\omega\rangle^{1/s-|\beta|}|\xi|^{-|\alpha|-1},
\eeqsn
that is \eqref{one}. Using \eqref{derbracketdot} and \eqref{stima:deromegai}, from \eqref{dx+1} we get:
\beqsn
|\partial_\xi^{\alpha}\partial_x^{\beta+e_j}\lambda_2(x,\xi)|&\leq & M\alpha!(\beta+e_j)!C_s A_0^{|\alpha|+|\beta|+2}\langle x\cdot\omega\rangle^{1/s-1-|\beta|}|\xi|^{-|\alpha|},
\eeqsn
and then \eqref{two} by definition of $\R^{2n}\setminus\mathcal R.$

\end{proof}

Taking into account the results of Lemmas \ref{lemma1} and \ref{lemma2} we can now define
\beqs\label{la}
\tilde{\lambda}(x,\xi)=-\lambda_1(x,\xi)\tilde\chi(x,\xi)-\lambda_2(x,\xi)(1-\tilde\chi(x,\xi)),\quad \tilde\chi(x,\xi)=\chi\left(\frac{2x\cdot\omega}{\langle x\rangle}\right),
\eeqs
where $\chi\in C^\infty_c(\R)$ is such that $0\leq \chi(t)\leq 1$, $t\chi'(t)\leq 0$, $\chi(t)=1$ for $|t|\leq 1/2$, $\chi(t)=0$ for $|t|\geq 1$, and
$|\chi^{(k)}(t)|\leq A_0^{k+1}k!^\mu$ for some $\mu >1$ to be chosen later on.

\begin{Lemma}\label{lemma3}
The function $\tilde\chi$ defined in \eqref{la} satisfies the following estimate:
\beqs
|\partial_\xi^\alpha\partial_x^\beta\tilde\chi(x,\xi)|\leq C^{|\alpha|+|\beta|+1}|\alpha+\beta|!^\nu\langle x\rangle^{-|\beta|}|\xi|^{-|\alpha|}, \quad (x,\xi)\in \R^{2n}.
\eeqs
\end{Lemma}
\begin{proof}
The proof follows by application of the tha Fa\`a di Bruno formula to the composed function $\chi(\eta(x,\xi))$ with $\eta(x,\xi)=2(x\cdot\omega)\langle x\rangle^{-1}$, and with the help of  \eqref{derxixomega}, \eqref{KB2} and of the estimate
\beqs \label{KB1}
|\partial^\gamma \langle y\rangle^m|\leq A_0^{|\gamma|+1}|\gamma|!\langle y\rangle^{m-|\gamma|}, \quad y\in\R^n, \, m \in \R, \, \gamma\in\Z^n_+, 
\eeqs
for some $A_0$ independent of $\gamma$ and $h$.  (cf. \cite{KB}). We leave the details to the reader.
\end{proof}

As an immediate consequence of Lemmas \ref{lemma1}, \ref{lemma2}, \ref{lemma3} we obtain the following result. The details of the proof are left to the reader.

\begin{Lemma}\label{lemma4} There exists a constant $C_s$ independent of $h$ such that the function $\tilde{\lambda}(x,\xi)$ defined in \eqref{la} satisfies the following estimate for every $\alpha,\beta\in\Z^n_+$: 
\beqs\label{derlambda}
|\partial_\xi^\alpha\partial_x^\beta\tilde{\lambda}(x,\xi)|&\leq& MC_s^{|\alpha|+|\beta|+1}|\alpha+\beta|!^\mu\langle x\rangle^{1/s-|\beta|}|\xi|^{-|\alpha|}.
\eeqs
for every $(x,\xi) \in \R^{2n}, |\xi| > 1$.
\end{Lemma}


In order to apply the results of Section 2 we need a last step, namely to cut-off the function $\tilde{\lambda}(x,\xi)$ near $\xi =0$ in order to deal with a symbol in the  class $\textbf{\textrm{SG}}^{(0,1/s)}_{\mu}(\R^{2n})$. Let us then define
\beqs \label{la2}
\lambda(x,\xi) =\left(1-\chi (h^{-1}|\xi|)\right)\tilde{\lambda}(x,\xi)
\eeqs
where $\chi$ is a compactly supported function of class $\gamma^\mu(\R^n)$ with $\chi(t)=1$ for $|t| \leq 1$.
It is easy to verify that the symbol $\lambda$ still satisfies the estimates in Lemma \ref{lemma4} on all $\R^{2n}$ with $|\xi|^{-|\alpha|}$ replaced by $\langle \xi\rangle_h^{-|\alpha|}$.
This happens because the new symbol we are considering vanishes on the set $\{|\xi|\leq h\}$, and on the complementary set we have $|\xi|^{-|\alpha|}\leq \sqrt2^{|\alpha|}\langle\xi\rangle_h^{-|\alpha|}$. This implies in particular that $\lambda \in \textbf{\textrm{SG}}^{(0,1/s)}_{\mu}(\R^{2n})$.
More precisely, we have
\beqs\label{base}
|\partial_\xi^\alpha\partial_x^\beta\lambda(x,\xi)|\leq MC_s^{|\alpha|+|\beta|+1}|\alpha+\beta|!^\mu\langle x\rangle^{1/s-|\beta|}\langle\xi\rangle_h^{-|\alpha|},\quad (x,\xi)\in\R^{2n}
\eeqs
for every $h\geq 1,$ where the constants $M$ and $C_s$ are independent of $h$. \\

The following simple but crucial result will be the key of the proof of Theorem \ref{main}: 
\begin{Lemma}\label{lemmastar}
The symbol $\lambda$ in \eqref{la2} is such that $\ds\sum_{j=1}^n(\partial_{x_j}\lambda)(x,\xi){\xi_j}\leq -M\langle x\rangle^{1/s-1}|\xi|$.
\end{Lemma}
\begin{proof}
Since the assertion involves only the derivatives with respect to $x$, it is sufficient to prove it for the function $\tilde{\lambda}$ defined by \eqref{la}.
By definition of $g_1,g_2$, and since $\langle x\cdot\omega\rangle\leq\langle x\rangle$, we have 
\beqs\label{utile}
(g_1-g_2)(x,\xi)=M(\langle x\rangle^{-1+1/s}-\langle x\cdot\omega\rangle^{-1+1/s})\leq 0,
\eeqs
and then
\beqs\label{serve}
(x\cdot\omega)(\lambda_1-\lambda_2)(x,\xi)=(x\cdot\omega)\int_0^{x\cdot\omega}(g_1-g_2)(x-\tau\omega,\xi)d\tau\leq 0.
\eeqs
Now by definition \eqref{la} of $\tilde \lambda$ and making use of \eqref{eq1}, \eqref{eq2} we have, omitting the dependence on $(x,\xi)$:
\beqsn
\sum_{j=1}^n(\partial_{x_j}\tilde\lambda){\xi_j}&=&\sum_{j=1}^n\left(-(\partial_{x_j}\lambda_1)\tilde\chi-\lambda_1\partial_{x_j}\tilde\chi-(\partial_{x_j}\lambda_2)(1-\tilde\chi)+\lambda_2\partial_{x_j}\tilde\chi\right){\xi_j}
\\
&=&-(\lambda_1-\lambda_2)\sum_{j=1}^n(\partial_{x_j}\tilde\chi)\xi_j-\tilde\chi\sum_{j=1}^n(\partial_{x_j}\lambda_1)\xi_j-(1-\tilde\chi)\sum_{j=1}^n(\partial_{x_j}\lambda_2)\xi_j
\\
&=&-(\lambda_1-\lambda_2)\chi'\left(\frac{2(x\cdot\omega)}{\langle x\rangle}\right)\sum_{j=1}^n\partial_{x_j}\frac{2(x\cdot\omega)}{\langle x\rangle}\xi_j-\tilde\chi|\xi|g_1-(1-\tilde\chi)|\xi|g_2.
\eeqsn
On the other hand
\beqsn
\sum_{j=1}^n\partial_{x_j}\frac{2(x\cdot\omega)}{\langle x\rangle}\xi_j&=&2\sum_{j=1}^n\left(\frac{\omega_j}{\langle x\rangle}-\frac{(x\cdot\omega)x_j}{\langle x\rangle^3}\right)\xi_j
\\
&=&2\left(\frac{|\xi|}{\langle x\rangle}-\frac{(x\cdot\xi)^2}{|\xi|\langle x\rangle^3}\right)=\frac{2|\xi|}{\langle x\rangle}\left(1-\left(\frac{x\cdot\omega}{\langle x\rangle}\right)^2\right),
\eeqsn
so we get
\beqsn
\sum_{j=1}^n(\partial_{x_j}\tilde\lambda){\xi_j}&=&-(\lambda_1-\lambda_2)\chi'\left(\frac{2(x\cdot\omega)}{\langle x\rangle}\right)
\frac{2|\xi|}{\langle x\rangle}\left(1-\left(\frac{x\cdot\omega}{\langle x\rangle}\right)^2\right)-\tilde\chi|\xi|g_1-(1-\tilde\chi)|\xi|g_2
\\
&\leq & 0-\tilde\chi|\xi|g_1-(1-\tilde\chi)|\xi|g_1=-|\xi|g_1,
\eeqsn
because $g_1\leq g_2$ by \eqref{utile} and
\beqsn
&&(\lambda_1-\lambda_2)\chi'\left(\frac{2(x\cdot\omega)}{\langle x\rangle}\right)
\frac{2|\xi|}{\langle x\rangle}\left(1-\left(\frac{x\cdot\omega}{\langle x\rangle}\right)^2\right)\geq 0\iff
\\
&&\iff(x\cdot\omega)^2(\lambda_1-\lambda_2)\chi'\left(\frac{2(x\cdot\omega)}{\langle x\rangle}\right)
\frac{2|\xi|}{\langle x\rangle}\left(1-\left(\frac{x\cdot\omega}{\langle x\rangle}\right)^2\right)\geq 0\iff
\\
&&\iff \left((x\cdot\omega)(\lambda_1-\lambda_2)\right) \left((x\cdot\omega)\chi'\left(\frac{2(x\cdot\omega)}{\langle x\rangle}\right)\right)
\frac{2|\xi|}{\langle x\rangle}\left(1-\left(\frac{x\cdot\omega}{\langle x\rangle}\right)^2\right)\geq 0,
\eeqsn
where this last inequality is true thanks to \eqref{serve}, the fact that $t\chi'(t)\leq 0\ \forall t$, and the fact that $|x\cdot\omega|\leq\langle x\rangle$. Lemma \ref{lemmastar} is proved.
\end{proof}

\section{The proof of Theorem \ref{main}}\label{sec4}
The proof of Theorem \ref{main} is based on the change of variable described in the previous section and on the application of Theorem \ref{conjugationthm} in the case when 
the operator $p(x,D)$ is replaced by each of the terms appearing in the expression of $P(t,x,\partial_t,\partial_x)$ in \eqref{P}. We observe that these terms can be regarded as operators with symbols which are continuous in $t$, analytic in $\xi$ and Gevrey of order $s_0$ with respect to $x$, where $s_0>1$ is the same appearing in the statement of Theorem \ref{main}. Namely they are \textbf{\textrm{SG}} operators with symbols which are analytic in $\xi$, hence in particular they belong to $C([0,T], \textrm{\textbf{SG}}^m_{\mu,s_0}(\R^{2n}))$ for every $\mu >1$ with $m=(m_1,m_2) \in \R^2, m_j \leq 2, j=1,2.$   
We observe moreover that since $\lambda_1$ and $\lambda_2$ are analytic, the function $\lambda(x,\xi)$ in \eqref{la2} inherits the same regularity of the cut-off function $\chi$ which can be chosen in $\gamma^{\mu}(\R^n)$ where $\mu$ is chosen so small that $s_0+\mu-1 < 1/(1-\sigma)$. Then the function $\Lambda$ in \eqref{formofLambda} is such that $\Lambda \in C^1([0,T],\textrm{\textbf{SG}}^{(0, 1/s)}_{\mu}(\R^{2n})$. Moreover, taking $h \geq h_o$ for some suitable $h_o \geq 1$ we have that $e^\lambda(t,x,D)$ is invertible and satisfies the assumptions of Theorem \ref{conjugationthm}. Consequently, $e^{\Lambda}(t,x,D)$ is invertible and by Theorem \ref{conjugationthm} we can state  the following:

\begin{Th}
\label{conjugation}
Let $s_0 >1,\sigma \in (0,1)$ such that $s_0  \in (1, 1/(1-\sigma))$ and let $\mu>1$ such that $\mu+s_0-1< 1/(1-\sigma)$. Let $p \in C([0,T],\textbf{\textrm{SG}}_{\mu,s_0}^m (\R^{2n}))$ for some $m=(m_1,m_2) \in \R^2,$ with $m_j\leq 1,j=1,2,$ and let $\Lambda$ be the symbol defined in \eqref{formofLambda} with $\lambda$ satisfying the condition \eqref{base} for some $s\in (s_0,1/{1-\sigma})$, $h\geq 1$ and for some constants $C_s$ and $M$ independent of $h$. Then there exists $h_o \geq 1$ such that if $h \geq h_o$, then
$$e^{\Lambda}(t,x,D_x)  p(t,x, D_x) (e^\Lambda(t,x,D_x))^{-1} = p(t,x,D_x) +q(t,x,D_x)+ r(t,x,D_x)+r_0(t,x,D_x)$$
where $$
q(t,x,\xi)= \sum_{|\alpha|=1} \partial_\xi^\alpha p(t,x,\xi)  (i\partial_x)^\alpha \Lambda(t,x,\xi) + \sum_{|\beta|=1}  D_x^\beta p(t,x,\xi) \partial_{\xi}^\beta \Lambda(t,x,\xi) $$
and $r \in C([0,T],\textbf{\textrm{SG}}_{\mu,s_0}^{(m_1-2, m_2-2(1-1/s))}(\R^{2n}))$ and $r_0 \in C([0,T], \mathcal{S}_{s_0+\mu-1}(\R^{2n})).$
\end{Th}

\textit{Proof of Theorem \ref{main}.}

First of all, by 
\eqref{conjdelta} and applying Theorem \ref{conjugation}, using the assumptions on $a_j$ and $b$,
we get
\beqs\nonumber
P_\Lambda(t,x,\partial_t, \partial_x)&=&e^{\Lambda}(t,x,D_x) P(t,x,\partial_t,\partial_x) (e^\Lambda)^{-1}(t,x,D_x)
\\
\nonumber
&=&\partial_t-i\Delta_x -k'(t)\px_h^{1-\sigma}-\sum_{j=1}^n2(\partial_{x_j}\Lambda)D_{x_j}+r'_0(t,x,D)
\\
\nonumber
& &+e^\Lambda\left(\sum_{j=1}^na_j(t,x)\partial_{x_j}+ b(t,x)\right)(e^\Lambda)^{-1}
\\
\label{ht}
&=&\partial_t-i\triangle_x+A_{\Lambda}(t,x,D_x)+\tilde r_0(t,x,D_x)
\eeqs
with $\tilde r_0(t,x,\xi) \in C[0,T], \textbf{\textrm{SG}}^{(0,0)}_{\mu,s_0}(\R^{2n}))$
and 
$$A_\Lambda(t,x,D)=i\ds\sum_{j=1}^na_j(t,x)D_{x_j}+b(t,x)-2\ds\sum_{j=1}^n(\partial_{x_j}\Lambda)(t,x,D)D_{x_j}-k'(t)\px_h^{1-\sigma}+r(t,x,D),$$
according to \eqref{biscio}, where $r(t,x,\xi)\in C([0,T], \textbf{\textrm{SG}}^{(0,1-\sigma)}_{\mu,s_0}(\R^{2n}))$. 
Let us now derive an energy estimate for the operator $P_\Lambda$. We have
\beqsn
\ds\frac d {dt} \|v\|^2&=&2\Re\langle P_\Lambda v, v\rangle-\langle (A_\Lambda+A_\Lambda^\ast)v, v\rangle-2\Re\langle \tilde r_0v, v\rangle.
\eeqsn
Let us now compute the symbol of the operator $A_\Lambda+A_\Lambda^\ast$. 
First of all we observe that
\beqsn
&&ia_j(t,x)D_{x_j}+(ia_j(t,x)D_{x_j})^\ast=
\\
&&=\textrm{op}\left(i\Re a_j(t,x)\xi_j-\Im a_j(t,x)\xi_j-i\Re a_j(t,x)\xi_j-\Im a_j(t,x)\xi_j+\tilde r_1(t,x,\xi) \right)
\\
&&=\textrm{op}\left(-2\Im a_j(t,x)\xi_j+\tilde r_1(t,x,\xi)+\tilde{\tilde r}_2(t,x,\xi)\right),
\eeqsn
with $\tilde r_1 \in C([0,T], \textbf{\textrm{SG}}^{(0,0)}_{\mu,s_0}(\R^{2n})$ and $\tilde{\tilde r}_2(t,x,D)$ a regularizing operator; similarly $$\left(b+b^\ast\right)(t,x)=2\Re b(t,x);$$ moreover the real valued operators $(\partial_{x_j}\Lambda)(t,x,D)D_{x_j}$ are such that $$\left(\partial_{x_j}\Lambda)(t,x,D)D_{x_j}\right)^\ast=(\partial_{x_j}\Lambda)(t,x,D)D_{x_j}+\tilde r_2(t,x,D),$$ with $\tilde r_2$ of order $(0,0)$. This gives
\beqsn
A_\Lambda+A_\Lambda^\ast&=&\textrm{op}\left(-2\sum_{j=1}^n\left[\Im a_j(t,x)+2(\partial_{x_j}\Lambda)(t,x,\xi)\right]{\xi_j}\right)
\\
&&+2\left(\Re b(t,x)-k'(t)\px_h^{1-\sigma}+\Re r(t,x,D_x)\right)+r'(t,x,D)
\eeqsn
with $r'(t,x,\xi) \in C([0,T], \textbf{\textrm{SG}}^{(0,0)}_{\mu,s_0}(\R^{2n}))$, and so
\beqs\label{quasi}
\ds\frac d {dt} \|v\|^2&=&2\Re\langle P_\Lambda v, v\rangle
-2\left\langle \textrm{op}\left(-\sum_{j=1}^n\left[\Im a_j(t,x)+2(\partial_{x_j}\Lambda)(t,x,\xi)\right]{\xi_j}\right)v, v\right\rangle
\\\nonumber
&&-2\langle\left(\Re b(t,x)-k'(t)\px_h^{1-\sigma}+\Re r(t,x,D_x)\right)v,v \rangle-\langle r'(t,x,D_x)v, v\rangle.
\eeqs
On one hand, by assumption \eqref{imm} and Lemma \ref{lemmastar}, recalling that we are going to choose a non-negative function $k(t)$, we have 
\beqs\nonumber
\hskip-0.2cm\sum_{j=1}^n\left[\Im a_j(t,x)\xi_j+2(\partial_{x_j}\Lambda)(t,x,\xi){\xi_j}\right]&=&\sum_{j=1}^n\left[\Im a_j(t,x)\xi_j+2(\partial_{x_j}\lambda)(x,\xi){\xi_j}+2k(t)\partial_{x_j}\px_h^{1-\sigma}\xi_j\right]
\\
\nonumber
&\leq& C\langle x\rangle^{-\sigma}|\xi|-2M\langle x\rangle^{1/s-1}|\xi|+2(1-\sigma)k(0)\px_h^{-\sigma}|\xi|
\\
\nonumber
&\leq & \left(C\langle x\rangle^{1/s-1}-2M\langle x\rangle^{1/s-1}+2(1-\sigma)k(0)h^{-\sigma+1-\frac{1}{s}}\px_h^{1/s-1}\right)|\xi|
\\
 \label{estimateremainders}
&\leq & \left(C-2M+2(1-\sigma)k(0)h^{-\sigma+1-\frac{1}{s}}\right)\langle x\rangle^{1/s-1}|\xi|
\eeqs
where $C>0$ is the constant in Theorem \ref{main}.

%
On the other hand, we have
\beqs \label{bruttiresti}
k'(t)\px_h^{1-\sigma}-\Re b(t,x)-\Re r(t,x,\xi)\leq \left(k'(t)+N+N(|k(t)|+ M)\right) \px_h^{1-\sigma}
\eeqs
where $N$ is a positive constant independent of $k(t)$ and $M$. To let the symbol on the right hand side of \eqref{bruttiresti} be non-positive it is suffcient to find $k(t)$ which solves the equation
$$k'(t)+N|k(t)|+ N(M+1) =0,$$
that is
$$k(t)= e^{-Nt}k(0) - (M+1)(1-e^{-Nt}),$$
where we choose $k(0)$ so large that $k(t) \geq 0$ for $t \in [0,T].$

Let us now fix $M>C$ and then $k(0)\geq(M+1)(e^{NT}-1)$. With these choices, we have that there exists an $h_o\geq1$ such that for $h\geq h_o$ the transformation $e^\Lambda$ is invertible and $P_\Lambda$ has the form \eqref{ht}, and moreover $k(t)\geq 0\ \forall t\in [0,T]$. Let us fix $h\geq h_o$ large enough to have $(1-\sigma)k(0)h^{-\sigma+1-\frac{1}{s}}\leq C/2$. With this choice of $h$, on one hand we have from \eqref{estimateremainders}
\begin{equation}
\label{minorazione1}\left(C-2M+2(1-\sigma)k(0)h^{-\sigma+1-\frac{1}{s}}\right)\langle x\rangle^{1/s-1}|\xi|\leq \left(2C-2M\right)\langle x\rangle^{1/s-1}|\xi|\leq 0
\end{equation}
thanks to the choice $M>C;$ on the other hand, by the choice of $k(t)$, from  \eqref{bruttiresti} we get
\begin{equation} \label{minorazione2}
k'(t)\px_h^{1-\sigma}-\Re b(t,x)-\Re r(t,x,\xi)\leq 0.
\end{equation}

Then by applying the sharp G{\aa}rding inequality we obtain
\beqsn
\left\langle \textrm{op}\left(-\sum_{j=1}^n\left[\Im a_j(t,x)+2(\partial_{x_j}\Lambda)(x,\xi)\right]{\xi_j}\right)v, v\right\rangle\geq 0\quad\forall v\in \mathscr{S}(\R^n),
\\
\left\langle  \Re b(t,x)-k'(t)\px_h^{1-\sigma}+\Re r(t,x,D_x))v,v \right\rangle \geq 0 \qquad \forall v \in \mathscr{S}(\R^n).
\eeqsn
 Hence we get
$$\ds\frac d {dt} \|v\|^2\leq C_0\left(\|P_\Lambda v \|^2+\|v\|^2\right)$$
for a positive constant $C_0$ and Gronwall's lemma gives the $L^2$-energy estimate: 
 \beqs\label{enL2}
 \|v(t)\|^2\leq c\left(\|v(0)\|^2+\int_0^t\|P_\Lambda v(\tau)\|^2d\tau\right),\quad\forall t\in[0,T]
 \eeqs
for a suitable constant $c>0$ and for all $v\in C([0,T]; \mathscr{S}(\R^n))$. Similarly, fixed $m =(m_1,m_2) \in \R^2$ and differentiating $\|v(t)\|^2_{H^m}$, we obtain the same estimates with the $L^2$-norms replaced by the $H^m$-norms. This implies in particular that the Cauchy problem \eqref{CP2} is well posed in $\mathscr{S}(\R^n)$. 
We want to show now that the same holds replacing $P_\Lambda$ by the operator 
$$\widetilde{P}_\Lambda = \Pi_{m,\rho,s, \theta} P_\Lambda  \Pi_{m,\rho,s,\theta}^{-1},$$
where $ \Pi_{m,\rho,s, \theta}$ is defined by \eqref{pigr}.
Now, it is easy to verify that by Propositions \ref{conjxi} and \ref{congx} we have, modulo terms of order $(0,0)$:
$$\Pi_{m,\rho,s, \theta} (-i\Delta)  \Pi_{m,\rho,s,\theta}^{-1} =-i\Delta + q(x,D)$$ with
\begin{equation}
\label{termineq}q(x,\xi)= \textrm{op} \left( 2 \rho_2 \sum_{k=1}^n \partial_{x_k} \px^{\frac{1}{s}} \xi_k \right)
\end{equation} satisfying the estimate
$$|q(x,\xi)| \leq \frac{2\rho_2}{s} \px^{1/s-1}|\xi|,$$
 whereas 
$$\Pi_{m,\rho,s, \theta} A_\Lambda \Pi_{m,\rho,s,\theta}^{-1} = A_\Lambda + 
 r_1(t,x,D) + r_2(t,x,D), $$
where $r_1 \in C([0,T], \textrm{\textbf{SG}}^{(1/\theta, -\sigma)}(\R^{2n}))$ and $r_2 \in C([0,T], \textrm{\textbf{SG}}^{(0, 1/s-\sigma)}(\R^{2n}))$.
More precisely, $r_1$ is the remainder of the conjugation of $A_\Lambda$ with $e^{\rho_1\pxi^{1/\theta}}$ and $e^{-\rho_1\pxi^{1/\theta}}$, then it is easy  to verify by \eqref{sviluppoconiugazionexi} that it satisfies the estimate:
$$|r_1(t,x,\xi)|\leq C' \px^{-\sigma} |\xi| \pxi^{1/\theta-1} \leq C''\px^{-\sigma}|\xi| h^{1/\theta-1}.$$
The term $r_2$ is the remainder of the conjugation of $A_\Lambda$ with $e^{\rho_2\px^{1/s}}$ and $e^{-\rho_2\px^{1/s}}$, hence it satisfies the estimate:
$$|r_2(t,x,\xi)|\leq N'\px^{1-\sigma}$$
for some positive constant $N'$ independent of $M$ and $k(t)$. Hence the additional terms $q$ and $r_1$ can be treated as in \eqref{minorazione1} possibly enlarging $M$ and taking $M>C+\rho_2/s$ but independent of $\rho_1$, whereas the term $r_2$ can be treated as the other terms in \eqref{minorazione2}.
We finally obtain the energy estimate:
 \beqs\label{enHmtilde}
 \|v(t)\|_{H^m}^2\leq c\left(\|v(0)\|_{H^m}^2+\int_0^t\|\widetilde{P}_\Lambda v(\tau)\|_{H^m}^2d\tau\right),\quad\forall t\in[0,T].
 \eeqs
for every $v \in C^1([0,T], \mathscr{S}(\R^n)).$ This implies that the Cauchy problem \eqref{CP2} is well posed in $\mathcal{S}_s^\theta(\R^n)$ and in $\Sigma_s^\theta(\R^n)$ and the solution $v$ satisfies the energy estimate:
 \beqs \label{eninfinite}
 \|v(t)\|_{H^m_{\varrho,s,\theta}}^2\leq c\left(\|v(0)\|_{H^m_{\varrho,s,\theta}}^2+\int_0^t\|f_\Lambda (\tau,\cdot)\|_{H^m_{\varrho,s,\theta}}^2d\tau\right)\quad\forall t\in[0,T].
 \eeqs
Indeed, if $f_\Lambda \in C([0,T], H^m_{\rho,s,\theta})$ and $g \in H^m_{\rho,s,\theta}$ for every $m \in \R^2$ and for some $\rho =(\rho_1, \rho_2) \in \R^2$ with $\rho_j >0, j =1,2,$ then $\Pi_{m,\rho,s}f \in C([0,T], \mathscr{S}(\R^n))$ and $ \Pi_{m,\rho,s} g \in \mathscr{S}(\R^n)$, hence the Cauchy problem for the operator $\widetilde{P}_\Lambda$ and data $\Pi_{m,\rho,s}f$ and $\Pi_{m,\rho,s}g$ has a unique solution $v \in C^1([0,T], \mathscr{S}(\R^n))$ satisfying \eqref{enHmtilde}. But this implies that the function $w = \Pi_{m,\rho,s}^{-1}v \in C^1([0,T], H^m_{\rho,s,\theta}(\R^n))$ is a solution of \eqref{CP2} and satisfies \eqref{eninfinite}.

Let us finally come back to our Cauchy problem \eqref{CP}, which is equivalent to \eqref{CP2} by the change of variable \eqref{change}, with $\Lambda$ of order $(0,1/s)$ and $s \in (s_0, 1/(1-\sigma)).$ 

For all $g\in H^m_{\rho,s,\theta}$ and $f\in C([0,T]; H^m_{\rho, s,\theta})$, by Proposition \ref{continforder} we have that  for every $\delta >C (\Lambda)$, $ f_\Lambda \in C([0,T]; H^m_{\rho-\delta e_2, s,\theta})$ and $g_{\Lambda} \in H^m_{\rho-\delta e_2, s, \theta}$ . Then, if $v$ is the unique solution of the Cauchy problem \eqref{CP2}, then the function $u= (e^\Lambda)^{-1} v$ solves the problem \eqref{CP} and satisfies the following energy estimate:
\beqsn
\|u\|_{H^m_{(\rho_1,\rho_2-2\delta),s,\theta}}^2&=&\|(e^\Lambda)^{-1} v\|_{H^m_{(\rho_1,\rho_2-2\delta),s,\theta}}^2
\leq c\| v\|_{H^m_{(\rho_1,\rho_2-\delta),s, \theta}}^2
\\
&\leq& c\left(\|e^\Lambda g \|_{H^m_{(\rho_1,\rho_2-\delta),s, \theta}}^2+\int_0^t\|e^\Lambda f(\tau)\|_{H^m_{(\rho_1,\rho_2-\delta),s, \theta}}^2d\tau\right)
\\
&\leq& c\left(\|g\|_{H^m_{\rho,s, \theta}}^2+\int_0^t\| f(\tau)\|_{H^m_{\rho,s, \theta}}^2d\tau\right), \qquad t\in[0,T]
\eeqsn
for some positive $\delta$.  Theorem \ref{main} is then proved.

\qed

\begin{Rem}\label{critical} By small changes in the proof, in the critical case $s=1/(1-\sigma)$, under the same assumptions we can show local in time well-posedness for the Cauchy problem \eqref{CP}, that is we can show that the solution $u$ is in $C([0,T^\ast], H^m_{(\rho_1, \rho_2-\bar \delta),s,\theta})$ for some $T^\ast\leq T$. More precisely, taking $s=1/(1-\sigma)$,
formula \eqref{estimateremainders} turns into
$$\sum_{j=1}^n\left[\Im a_j(t,x)\xi_j+2(\partial_{x_j}\Lambda)(t,x,\xi){\xi_j}\right]\leq\left(C-2M+2(1-\sigma)k(0)\right)\langle x\rangle^{1/s-1}|\xi|.$$
Taking into account also the term $q(x,D)$ in \eqref{termineq}, we have to choose $M\geq C/2+(1-\sigma)k(0)+\rho_2/s$. Then since we want 
$k(0)\geq (M+1)(e^{NT}-1)$, then we have to take
$$k(0)\geq \ds\frac{(C/2+1+\rho_2/s)(e^{NT}-1)}{1-(1-\sigma)(e^{NT}-1)}.$$
This can be done only if $T<\frac1N\ln\left(1+\frac1{1-\sigma}\right)$, that is only locally in time.
\end{Rem}

\section{Examples and concluding remarks}
In this section we give some examples showing that the phenomenon of the loss of decay really appears in the problem \eqref{CP}. Moreover, we show by a counterexample that the bound $s=1/(1-\sigma)$ is sharp. Finally, we discuss the possibility to obtain solutions of \eqref{CP} with loss of regularity but no loss of decay, stating a result which can be easily proved combining the argument of the proof of Theorem \ref{main} with the techniques used in \cite{KB}; since the proof of this result is a mere repetition of the argument of the proof of Theorem \ref{main} for a different choice of $\Lambda$, it is omitted. We conclude the paper leaving to the reader an open question.
\\ 

\begin{Ex}\label{e1}
Let $T>0$, $s>1$, $\sigma\in (0,1)$ such that $s< \ds\frac{1}{1-\sigma}$, and consider the Cauchy problem
\begin{equation}\label{CPex}\begin{cases}
\partial_t u -i \partial^2_x u +i t(1-\sigma)x\px^{-\sigma-1} \partial_x u + b(t,x)u=0  \qquad (t,x) \in [0,T] \times \R, \\
u(0,x)= g(x):= e^{-\px^{\frac{1}{s}}} \qquad \qquad \qquad \qquad \qquad \qquad \qquad x \in \R, 
\end{cases} \end{equation}
where $$b(t,x)= -\px^{1-\sigma}+i c(t,x)$$ and 
\begin{eqnarray*}c(t,x)&=& \frac{1}{s}x \px^{1/s-2}\left( t(1-\sigma)x \px^{-\sigma-1}-\frac{1}{s}x \px^{1/s-2} \right)
\\
&&+t(1-\sigma) \px^{-\sigma-1}-\frac{1}{s} \px^{1/s-2}- t(1-\sigma^2)x^2 \px^{-\sigma-3}- \frac{1}{s}\left(\frac{1}{s}-2\right)x^2 \px^{1/s-4} . 
\end{eqnarray*}
We notice that the function $b$ satisfies the condition \eqref{reb}. Moreover the coefficient $a(t,x)= i t(1-\sigma)x\px^{-\sigma-1} $ is purely imaginary and satisfies the condition \eqref{imm}.
Finally, $g \in H^{(0,-\tau)}_{(0,1), s, \theta}(\R)$ for every $\tau >1/2$ and $\theta >1$ since
$$\px^{-\tau} e^{\px^{1/s}}g(x) = \px^{-\tau} \in L^2(\R).$$
Then the assumptions of Theorem \ref{main} are all satisfied. It is easy to verify that the problem \eqref{CPex} admits the solution
\begin{equation}
\label{solutionex}
u(t,x)= e^{t \px^{1-\sigma}-\px^{1/s}} \notin H^{(0,-\tau)}_{(0,1), s, \theta}(\R)
\end{equation}
for any $t\in (0,T]$ since
 $$\px^{-\tau} e^{\px^{1/s}} u(x) = \px^{-\tau}e^{t \px^{1-\sigma}} \notin L^2(\R).$$
However, we have that $u \in C([0,T],  H^{(0,-\tau)}_{(0,1-\bar \delta), s, \theta}(\R))$ for every $\bar \delta>0$. Hence we obtain that the solution does not present a weaker regularity with respect to the initial datum $g$ but a weaker decay at infinity. In this case the loss $\bar \delta$ is arbitrarily small.
\end{Ex}
\begin{Ex}
We observe that Example \ref{e1} is valid also in the critical case $s=1/(1-\sigma)$. In this case, rephrasing it in terms of $\sigma$ we obtain that the Cauchy problem
\begin{equation}\label{CPex2}\begin{cases}
\partial_t u -i \partial^2_x u +i (t-1)(1-\sigma)x\px^{-\sigma-1} \partial_x u + b(t,x)u=0  \qquad (t,x) \in [0,T] \times \R, \\
u(0,x)= g(x):= e^{-\px^{1-\sigma}} \qquad \qquad \qquad \qquad \qquad \qquad \qquad x \in \R, 
\end{cases} \end{equation}
with
$b(t,x)= -\px^{1-\sigma}+i\left[(1-\sigma)(t-1)\px^{-\sigma-1}-(1-\sigma^2)(t-1)x^2 \px^{-\sigma-3}\right]$ admits the solution
$$u(t,x)= e^{(t-1)\px^{1-\sigma}} \in C([0,T], H^{(0,-\tau)}_{(0,1-T),1/(1-\sigma),\theta}(\R)).$$
In this case Theorem \ref{main} holds with the loss of decay $\bar \delta=T$.
\end{Ex}
\begin{Ex}
With minor changes in Example \ref{e1} the reader can easily verify that the function $$u(t,x)=e^{t\px^{1-\sigma}+\px^{1/s}}$$ with $s\leq \ds\frac1{1-\sigma}$ solves the Cauchy problem with exponentially growing initial datum $g(x)=e^{\px^{1/s}}$ for the equation
$$\partial_t u -i \partial^2_x u +i t(1-\sigma)x\px^{-\sigma-1} \partial_x u + b(t,x)u=0$$
where $b(t,x)=-\px^{1-\sigma}+ic(t,x)$ with
\begin{eqnarray*}c(t,x)&=& \frac{1}{s}x \px^{1/s-2}\left( t(1-\sigma)x \px^{-\sigma-1}+\frac{1}{s}x \px^{1/s-2} \right)
\\
&&+t(1-\sigma) \px^{-\sigma-1}+\frac{1}{s} \px^{1/s-2}- t(1-\sigma^2)x^2 \px^{-\sigma-3}+ \frac{1}{s}\left(\frac{1}{s}-2\right)x^2 \px^{1/s-4} . 
\end{eqnarray*}
Also in this situation the assumptions of Theorem \ref{main} are satisfied and  we have  a solution whose growth at infinity is stronger than the growth of the Cauchy datum.
\end{Ex}
\begin{Rem}
Example \ref{e1} shows also that the value $s=1/(1-\sigma)$ is a sharp threshold for the well posedness. In fact, if we assume $s>1/(1-\sigma)$, then we still have
$g \in H^{(0,-\tau)}_{(0,1),s,\theta}(\R)$ and the solution of \eqref{CPex} is still expressed by the function $u$ in \eqref{solutionex} but in this case we cannot find any $\bar \delta>0$ such that
$u \in H^{(0,-\tau)}_{(0,1-\bar \delta),s,\theta}(\R)$
since for every $\bar \delta >0$ we have
$$\px^{-\tau} e^{(1-\bar \delta) \px^{1/s}} u(x) = \px^{-\tau}e^{t \px^{1-\sigma}-\bar \delta \px^{1/s}} \notin L^2(\R)$$
if $s >1/(1-\sigma).$
\end{Rem}

The choice of the function $\Lambda(t,x,\xi)$ in Section 2 has a key role in the proof of Theorem \ref{main} and the main effect is to concentrate all the loss in the spaces $H^m_{(\rho_1, \rho_2), s, \theta}$ in the second index $\rho_2$. On the other hand, if we assume the coefficients $a_j$ and $b$ bounded in $(t,x)$ and use the function $\Lambda$ in \cite[Formula (2.15)]{KB} we can recapture the main results of \cite{KB} in the more general functional setting of our paper repeating readily the argument in the proof of Theorem \ref{main}, that is we obtain a solution which presents a loss of regularity with respect to the initial data but with the same behavior at infinity. More precisely, the following result holds. The proof is omitted for the sake of brevity.

\begin{Th} \label{main2}
Let $s_0>1, \sigma \in (0,1)$ such that $s_0 < 1/(1-\sigma)$ and let $P(t,x,\partial_t,\partial_{x})$ be an operator of the form \eqref{P} with $a_j$ and $b$ continuous with respect to $t$ and satisfying for all $(t,x)\in[0,T]\times\R^n$, $\beta \in \N^n$ and $1\leq j\leq n$ the following conditions:
\beqs
\label{immreg}
|\partial_x^\beta ( \Im a_j)(t,x)|&\leq& \ds C^{|\beta|+1}\beta!^{s_0}\px^{-\sigma-|\beta|}, 
\\
\label{rereg}
|\partial_x^\beta ( \Re a_j)(t,x)|&\leq& \ds C^{|\beta|+1}\beta!^{s_0}\px^{-|\beta|}
\\
\label{rebreg}
|\partial_x^\beta  b(t,x)|&\leq& \ds C^{|\beta|+1}\beta!^{s_0} \px^{-|\beta|}, 
\eeqs
for some positive constant $C$ independent of $\beta$.
Let moreover $f\in C([0,T]; H^m_{\rho, s,\theta}(\R^n))$ and $g\in H^m_{\rho, s,\theta}(\R^n)$ for some $s >s_0$, $\theta \in (s_0, 1/(1-\sigma))$ and $\rho=(\rho_1,\rho_2),m=(m_1,m_2) \in \R^2$. Then there exists $\bar \delta=\bar \delta (\theta, \rho_1)>0$ such that the Cauchy problem \eqref{CP} admits a unique solution $u\in C([0,T];H^m_{(\rho_1-\bar \delta,\rho_2), s,\theta}(\R^n))$
which satisfies:
\begin{equation}\label{enestreg}
\|u(t)\|^2_{H^m_{(\rho_1-\bar \delta,\rho_2),s,\theta}}\leq C_s\left(\|g\|_{H^m_{\rho,s,\theta}}^2+\int_0^t \|f(\tau)\|_{H^m_{\rho,s,\theta}}^2 d\tau \right),
\end{equation}
for  $t\in [0,T]$ and for some $C_s>0$. 
\end{Th}
\appendix

\section{Calculus for pseudodifferential operators of infinite order}\label{sec:appendix}
Here we develop the calculus for pseudodifferential operators with symbols in $\textbf{\textrm{SG}}^{\tau, \infty}_{\mu, \nu, s}(\R^{2n})$ and prove the results stated in Subsection 2.2. Some proofs will be just sketched or omitted since they follow readily the arguments used for other similar calculi, cf. \cite{CPP1, CPP2, CR,CT2, Pr}. For completeness and to achieve our results, we introduce two auxiliary classes of symbols which have infinite order in $\xi$, respectively in both $x$ and $\xi$.
\begin{Def}
Fixed $C >0, c >0$ and $\mu,\nu,\tau,\theta \in \R$ with $ \mu >1$ and $1 < \nu \leq \theta$, we shall denote by $\textbf{\textrm{SG}}_{\mu,\nu,\theta}^{\infty,\tau}(\R^{2n};C, c)$ the Banach space of all functions $a(x,\xi) \in C^\infty(\R^{2n})$ satisfying the following estimates:
\begin{equation}
\sup_{\alpha, \beta \in \N^n} \sup_{(x,\xi) \in \R^{2n}}
C^{-|\alpha|-|\beta|} (\alpha !)^{-\mu}(\beta!)^{-\nu} \pxi^{|\alpha|} \px^{-\tau+|\beta|}  \exp\left[-c|\xi|^{\frac{1}{\theta}} \right] \left| \der a(x,\xi) \right| < +\infty.
\end{equation}
We set $\textbf{\textrm{SG}}_{\mu,\nu, \theta}^{\infty,\tau}(\R^{2n}) = \lim\limits_{\stackrel{\longrightarrow}{C,c \to \infty}}\textbf{\textrm{SG}}_{\mu,\nu, \theta}^{\infty,\tau}(\R^{2n};C, c)$ endowed with the inductive limit topology.
\end{Def}

\begin{Def}
Fixed $C >0, c >0$ and $\mu,\nu,s, \theta \in \R$ with $1<\mu \leq s$ and $1< \nu \leq \theta$, we shall denote by $\textbf{\textrm{SG}}_{\mu,\nu,s,\theta}^{\infty}(\R^{2n};C, c)$ the Banach space of all functions $a(x,\xi) \in C^\infty(\R^{2n})$ satisfying the following estimates:
\begin{equation}
\sup_{\alpha, \beta \in \N^n} \sup_{(x,\xi) \in \R^{2n}}
C^{-|\alpha|-|\beta|} (\alpha !)^{-\mu}(\beta!)^{-\nu} \pxi^{|\alpha|} \px^{|\beta|}  \exp\left[-c(|x|^{\frac{1}{s}}+|\xi|^{\frac{1}{\theta}}) \right] \left| \der a(x,\xi) \right| < +\infty.
\end{equation}
We set $\textbf{\textrm{SG}}_{\mu,\nu, s, \theta}^{\infty}(\R^{2n}) = \lim\limits_{\stackrel{\longrightarrow}{C,c \to \infty}}\textbf{\textrm{SG}}_{\mu,\nu, s, \theta}^{\infty}(\R^{2n};C, c)$ endowed with the inductive limit topology.
\end{Def}
For simplicity, in the case $\mu=\nu$ we shall use the notation $\textbf{\textrm{SG}}_{\mu, \theta}^{\infty,\tau}(\R^{2n}) $ and $\textbf{\textrm{SG}}_{\mu, s, \theta}^{\infty}(\R^{2n}) $ for the classes $\textbf{\textrm{SG}}_{\mu,\mu, \theta}^{\infty,\tau}(\R^{2n})$ and $\textbf{\textrm{SG}}_{\mu,\mu, s, \theta}^{\infty}(\R^{2n}) $.\\
 
The following evident inclusions hold for every $m =(m_1, m_2) \in \R^2$ and for every $s >1, \theta>1$:
$$\textbf{\textrm{SG}}_{\mu,\nu}^{m}(\R^{2n}) \subset \textbf{\textrm{SG}}_{\mu,\nu, \theta}^{\infty,m_2}(\R^{2n}) \subset \textbf{\textrm{SG}}_{\mu,\nu, s, \theta}^{\infty}(\R^{2n})$$ 
and
$$\textbf{\textrm{SG}}_{\mu,\nu}^{m}(\R^{2n}) \subset \textbf{\textrm{SG}}_{\mu,\nu, s}^{m_1,\infty}(\R^{2n})\subset \textbf{\textrm{SG}}_{\mu,\nu, s, \theta}^{\infty}(\R^{2n}).$$

\begin{Prop}\label{appaelambda}
Let $\lambda \in C^\infty (\R^{2n}).$ Then the following conditions holds:\\
i) If $\lambda \in \textbf{\textrm{SG}}_{\mu}^{(0,1/s)}(\R^{2n})$, then $e^{ \lambda (x,\xi)} \in \textbf{\textrm{SG}}_{\mu, s}^{0,\infty}(\R^{2n});$\\
ii) If $\lambda \in \textbf{\textrm{SG}}_{\mu}^{(1/\theta,0)}(\R^{2n})$, then $e^{\lambda (x,\xi)} \in \textbf{\textrm{SG}}_{\mu, \theta}^{\infty,0}(\R^{2n}).$
\end{Prop}

\begin{proof} We prove only i), the proof of ii) being similar.
Let $\lambda \in \textbf{\textrm{SG}}_{\mu}^{(0,1/s)}(\R^{2n};C)$. 
Repeating readily the argument in the proof of \cite[Lemma 6.2]{KN} we obtain that 
$$\left| \der e^{ \lambda(x,\xi)} \right| \leq B^{|\alpha+\beta|}\pxi^{-|\alpha|}\px^{-|\beta|}  e^{\lambda(x,\xi)} \sum_{j=0}^{|\alpha+\beta|}(c_o \px^{1/s})^{|\alpha+\beta|-j} j!^\mu$$
for every $\alpha, \beta \in \N^n,$ with $B=6C$ and $c_o $ is the norm of $\lambda$ in $\textbf{\textrm{SG}}_{ \mu}^{(0,1/s)}(\R^{2n};C)$.
To conclude the proof it is sufficient to observe that by standard factorial inequalities we have:
$$\sum_{j=0}^{|\alpha+\beta|}(c_o \px^{1/s})^{|\alpha+\beta|-j} j!^\mu \leq |\alpha+\beta|!^\mu e^{c_o \px^{1/s}}.$$
\end{proof}

We start by proving Proposition \ref{gencontinforder}. For this we need a preliminary result, cf. \cite{Ivrii} for the proof.

\begin{Lemma}
\label{paola}
Given $\kappa>1, \zeta >0$,
 let $$m_{\kappa, \zeta}(x)= \sum_{j=0}^\infty \frac{\zeta^j \px^{2j}}{j!^{2\kappa}}, \qquad x \in \R^n.$$
Then, for every $\epsilon >0$ there exists a constant $C=C(\kappa, \epsilon)>0$ such that
\begin {equation}
C^{-1}e^{(2\kappa-\epsilon) \zeta^{\frac{1}{2\kappa}}\px^{\frac{1}{\kappa}}} \leq  m_{\nu,\zeta}(x) \leq C e^{(2\kappa+\epsilon) \zeta^{\frac{1}{2\kappa}}\px^{\frac{1}{\kappa}}} 
\end{equation}
for every $x \in \R^n$.
\end{Lemma}

\textit{Proof of Proposition \ref{gencontinforder}.} Observe that
$$
\frac{1}{m_{\kappa , \zeta}(x)} \sum_{j=0}^{\infty} \frac{\zeta ^j}
{(j!)^{2\kappa }}(1-\Delta_{\xi})^j
e^{i\langle x,\xi \rangle} = e^{i\langle x,\xi \rangle}.
$$
Let now $\kappa=s$ and let $1<\mu < s$ and $1<\nu \leq \theta$. 
Let moreover $\Omega$ be a bounded subset of $\mathscr{S}_{s}^\theta(\R^n).$ This implies that there exists $r_0>0$ such that for every $h >0$ and for every $f \in \Omega$:
$$  \sup_{\alpha \in \N^n}\sup_{\xi \in \R^n } h^{-|\alpha|}(\alpha!)^{-s} e^{r_0|\xi|^{1/\theta}}|D_\xi^\beta \hat{f}(\xi)| <\infty$$
For fixed $\alpha, \beta \in \N^n$ and for $f \in \Omega$ we have
\begin{eqnarray*}
 x^{\alpha}D_x^{\beta} (a(x,D)f)(x) 
&=& x^{\alpha} \sum_{\gamma \leq \beta}\binom{\beta}{\gamma} \int_{\R^n}e^{i \langle x,\xi \rangle}
 \xi^{\gamma} D_x^{\beta-\gamma}a(x,\xi)
 \widehat{f}(\xi )  \,  \dslash \xi 
\\
&=& \frac{x^{\alpha}}{m_{s ,\zeta}(x)} \sum_{\gamma \leq \beta}
\binom{\beta}{\gamma} 
g_{\zeta ,\beta ,\gamma}(x),
\end{eqnarray*}
with
$$
g_{\zeta,\beta ,\gamma}(x)
=
\sum_{j=0}^{\infty} \frac{\zeta ^j}{(j!)^{2s}}
\int_{\R^n}  e^{i \langle x,\xi \rangle} (1-\Delta_{\xi} )^j
\left (  \xi^{\gamma} D_x^{\beta-\gamma}a(x,\xi)
\widehat{f}(\xi ) \right )  
\, \dslash \xi.
$$

\par

Since $\mu <s$ and $\nu \leq \theta$ then, by standard factorial inequalities, there exist $A>0, c>0$ and for every $h>0$ there exists $C_h >0$ such that 
$$
|(1-\Delta_{\xi} )^j (\xi^{\gamma} D_x^{\beta-\gamma}a(x,\xi)
\widehat{f}(\xi ))|
\leq
C_h A^{|\beta|} h^{j}j!^{2s}\gamma !^\theta
(\beta-\gamma) !^\theta e^{c|x|^{\frac 1s}-r_0|\xi |^{\frac 1\theta}}.
$$
 Hence, we get
$$
|g_{\zeta ,\beta ,\gamma}(x)|
\leq C_h
\sum_{j=0}^{\infty}
(h\zeta^{\frac{1}{2s}})^j
A^{|\beta |}
(\beta )!^\theta e^{c|x|^{\frac 1s}}
\int_{\R^n} e^{-r_0|\xi |^{\frac 1\sigma}}\, d\xi
$$
Moreover for every $\varepsilon \in (0,r_0)$ we have
$$e^{-\varepsilon|x|^{1/s}}|x^\alpha| \leq A_\varepsilon^{|\alpha|+1} (\alpha!)^s.$$ 
Now, taking $\zeta$ such that $(2s-\epsilon)\zeta^{\frac{1}{2s}}-c>0$ and choosing $h$ so small that $h\zeta <1$, we obtain that the set
$$\{a(\cdot,D) f : f \in \Omega\}$$ is bounded in $\mathscr{S}_s^\theta(\R^n).$ This proves the continuity of $a(x,D)$ on $\mathscr{S}_s^\theta(\R^n).$
The continuity on
$(\mathscr{S} _{s}^\sigma )'(\R^d)$ now follows from the preceding
continuity and duality.
\qed
\\

A similar continuity result can be proved for operators with symbols in $\textrm{\textbf{SG}}^{\infty, \tau}_{\mu, \nu,\theta}(\R^{2n})$. We omit the proof since it can be obtained using the same type of argument as in the previous proof.

\begin{Prop}
Let $\mu, \nu, \tau, \theta \in \R$ with $1< \nu < \theta, \nu >1$ and let $a \in \textbf{\textrm{SG}}_{\mu,\nu,\theta}^{\infty,\tau}(\R^{2n})$. Then $a(x,D)$ is linear and continuous on  $\widetilde{\mathscr{S}}_{s}^\theta(\R^n)$ and it extends to a continuous map on $(\widetilde{\mathscr{S}}_{s}^\theta)'(\R^n)$ for every $s \geq \mu$.
\end{Prop}

We now define asymptotic expansions for symbols in $\textrm{\textbf{SG}}^{\tau, \infty}_{\mu,\nu,s}(\R^{2n})$. For $r> 0$ denote
$$Q_r = \{(x,\xi) \in \R^{n}: \pxi <r \quad \textit{and} \quad \px < r \}$$
and 
$$Q_r^e = \R^{2n} \setminus Q_r.$$

\begin{Def}
\label{sum}
Let $B,C,c>0.$ We shall denote by $\mathcal{FS}^{\tau,\infty}_{\mu,\nu,s}(\R^{2n}; B,C,c)$ the space of all formal sums $\sum\limits_{j \geq 0}a_j(x,\xi)$ such that $a_j(x,\xi) \in C^{\infty}(\R^{2n})$ for all $j \geq 0$ and
$$\sup_{j \geq 0} \sup_{\alpha, \beta \in \N^n} \sup_{(x,\xi) \in Q_{Bj^{\mu+\nu-1}}^e} C^{-|\alpha|-|\beta|-2j} (\alpha!)^{-\mu}(\beta!)^{-\nu} (j!)^{-\mu-\nu +1} \pxi^{-\tau+|\alpha|+ j} \px^{|\beta |+j} \cdot $$
\begin{equation}
\label{dev}
\cdot \exp \left[-c |x|^{\frac{1}{s}}\right] \left|\der a_j(x,\xi) \right| < +\infty.
\end{equation}
\end{Def}
Consider now the space $FS^{\tau,\infty}_{\mu,\nu, s}(\R^{2n}; B,C,c)$ obtained from $\mathcal{FS}^{\tau,\infty}_{ \mu,\nu, s}(\R^{2n}; B,C,c)$ by quotienting by the subspace $$E = \left \{ \sum\limits_{j \geq 0}a_j(x,\xi)  \in \mathcal{FS}^{\tau,\infty}_{\mu,\nu, s}(\R^{2n}; B,C): \textit{supp}(a_j) \subset Q_{Bj^{\mu+\nu-1}} \quad \forall j \geq 0 \right\}.$$
By abuse, we shall denote the elements of $FS^{\tau,\infty}_{\mu, \nu, s}(\R^{2n}; B,C,c)$ by formal sums of the form $\sum\limits_{j \geq 0}a_j(x,\xi).$ The arguments in the following are independent of the choice of representative.
We observe that $FS^{\tau,\infty}_{\mu,\nu, s}( \R^{2n}; B,C,c)$ is a Banach space endowed with the norm given by  the left-hand side of \eqref{dev}. We set
$$FS^{\tau,\infty}_{\mu,\nu, s}(\R^{2n}) = \lim_{\stackrel{\longrightarrow}{B,C,c \rightarrow +\infty}} FS^{\tau,\infty}_{\mu,\nu, s}( \R^{2n}; B,C,c).$$
Every symbol $a  \hskip-1pt \in \hskip-2pt \gam$ can be identified with an element $\sum\limits_{j \geq 0}a_j$ of $FS^{\tau,\infty}_{\mu, \nu, s}(\R^{2n}) \hskip-1pt,$ by setting $a_0 =a$ and $a_j=0$ for all $j \geq 1.$

\begin{Def}
\label{equiv}
We say that two sums $\sum\limits_{j \geq 0}a_j, \sum\limits_{j \geq 0} a^{\prime}_j$ from $FS^{\tau,\infty}_{\mu,\nu, s}(\R^{2n})$ are equivalent if there exist constants $B,C,c >0$ such that 
\begin{equation}\label{eqsums} \sup_{N \in \mathbb{Z}_+} \sup_{\alpha, \beta \in \N^n} \sup_{(x,\xi) \in Q^e_{BN^{\mu+\nu-1}}} \hskip-3pt C^{-|\alpha|-|\beta|-2N} (\alpha !)^{-\mu}(\beta !)^{-\nu} (N!)^{-\mu-\nu +1} \pxi^{-\tau+|\alpha|+N} \px^{|\beta|+N} \cdot $$ $$\cdot \exp \left[ -c (|x|^{\frac{1}{s}} \right] \left| \der \sum_{j<N} (a_j -a_j^{\prime} ) \right| < +\infty. 
\end{equation} In this case we write
$\sum\limits_{j \geq 0}a_j \sim \sum\limits_{j \geq 0}a^{\prime}_j $.
\end{Def}

In a similar way, by simply exchanging the roles of $x$ and $\xi$ in \eqref{dev},\eqref{eqsums}, we can define the space $FS^{\infty, \tau}_{\mu,\nu,\theta}(\R^{2n})$ encoding the asymptotic expansions of symbols from $\Gamma^{\infty, \tau}_{\mu,\nu,\theta}(\R^{2n})$. Moreover, we obtain a corresponding definition for the space $FS^m_{\mu,\nu}(\R^{2n})$ of formal sums of symbols of finite order and the related notion of equivalence for every $m=(m_1,m_2) \in \R^2$ by simply replacing $\tau$ by $m_1$ and $e^{-c|x|^{\frac1{s}}}$ by $ \px^{-m_2}$  in \eqref{dev}, \eqref{eqsums}, cf. \cite{CR} where the complete calculus for this class is developed. An analogous argument allows to define the class $FS^\infty_{\mu,\nu,s, \theta}(\R^{2n})$ and the notion of asymptotic expansions for symbols from $\textrm{\textbf{SG}}^{\infty}_{\mu,\nu,s, \theta}(\R^{2n})$. We omit the details for the sake of brevity.

\begin{Prop}
\label{sviluppo}
Given a sum $\sum\limits_{j\geq 0}a_j \in FS^{\tau,\infty}_{\mu, \nu, s}(\R^{2n}), $ (resp. $ \sum\limits_{j\geq 0}a_j \in FS^{m}_{ \mu,\nu}(\R^{2n})$), we can find a symbol $a \in \textbf{\textrm{SG}}^{\tau,\infty}_{\mu,\nu, s}(\R^{2n})$ (resp. $a \in  \textbf{\textrm{SG}}^{m}_{\mu,\nu}(\R^{2n})$)  such that
$$a \sim \sum\limits_{j\geq 0}a_j \quad \textit{in} \quad FS^{\tau,\infty}_{\mu,\nu, s}(\R^{2n}) \qquad (\textrm{resp. in} \quad FS^{m}_{\mu,\nu}(\R^{2n})).$$
\end{Prop}

\begin{proof}
Let $\varphi \in C^{\infty}(\R^{2n}), 0 \leq \varphi \leq 1$ such that $\varphi(x,\xi)=0$ if $(x, \xi) \in Q_1, \varphi(x, \xi)=1$ if $(x, \xi) \in Q^e_2$ and
\begin{equation}
\label{magg}
\left| D_x^{\delta} D_{\xi}^{\gamma} \varphi(x, \xi) \right| \leq C^{|\gamma| + |\delta|+1} (\gamma !)^\mu(\delta !)^{\nu} \quad \forall (x, \xi) \in \R^{2n}.
\end{equation}
We define: 
$$\varphi_0 (x,\xi) = \varphi \left(\frac{2}{R}x, \frac{2}{R} \xi \right)$$ and
$$\varphi_j (x, \xi) = \varphi \left( \frac{1}{Rj^{\mu+\nu -1}}x, \frac{1}{Rj^{\mu+\nu -1}} \xi \right), \quad j \geq 1.$$
We want to prove that if $R$ is sufficiently large,
\begin{equation}
\label{bocci}
a(x, \xi) = \sum_{j \geq 0} \varphi_j (x, \xi) a_j (x, \xi)
\end{equation}
is well defined as an element of $\textbf{\textrm{SG}}^{\tau,\infty}_{\mu,\nu, s}(\R^{2n})$ and $a \sim  \sum\limits_{j \geq 0} a_j$ in
$ FS_{\mu, \nu, s}^{\tau, \infty} (\R^{2n}).$

\noindent
We define $$\der a(x, \xi) = \sum_{j \geq 0} \sum_{\stackrel{\gamma \leq \alpha}{\delta \leq \beta}} \binom{\alpha}{\gamma} \binom{\beta}{\delta}
D_x^{\beta - \delta} D_{\xi}^{\alpha - \gamma} a_j (x,\xi) \cdot D_x^{\delta}D_{\xi}^{\gamma}
\varphi_j (x,\xi).$$
Choosing $R \geq B$ where $B$ is the constant in Definition \ref{sum}, we can apply the estimates (\ref{dev}) and obtain
$$\left| \der a(x,\xi) \right| \leq
C^{|\alpha| + |\beta| +1} \alpha ! \beta ! \langle x \rangle ^{- |\beta|} \langle \xi \rangle ^{\tau-|\alpha|} e^{c|x|^{\frac{1}{s}}} \sum_{j \geq 0} H_{j \alpha \beta} (x, \xi)$$
where  $$H_{j \alpha \beta} (x,\xi) = \sum_{\stackrel{\gamma \leq \alpha}{\delta \leq \beta}}
\frac{(\alpha - \gamma) !^{\mu-1}(\beta - \delta) !^{\nu -1}}{\gamma ! \delta !} \cdot C^{2j - |\gamma| -|\delta|}(j!)^{\mu+\nu -1} \langle x \rangle ^{|\delta| -j}
\langle \xi \rangle ^{|\gamma| -j} \left| D_x^{\delta}D_{\xi}^{\gamma} \varphi_j (x,\xi) \right|.$$
Now the condition (\ref{magg}) and the fact that $D_x^{\delta}D_{\xi}^{\gamma} \varphi_j (x,\xi)=0 $ in $Q^e_{2Rj^{2\nu-1}}$ for $(\delta, \gamma) \neq (0,0)$ imply that
$$H_{j \alpha \beta} (x,\xi) \leq C_1^{|\alpha| + |\beta| +1} 
(\alpha !)^{\mu-1}(\beta !)^{\nu -1} \left( \frac{C_2}{R} \right)^j$$ where $C_2$ is independent of $R.$
Enlarging $R,$ we obtain that
$$\sum_{j \geq 0} H_{j \alpha \beta} (x,\xi) \leq C_3^{ |\alpha| + |\beta| +1}
(\alpha !)^{\mu-1}(\beta !)^{\nu -1} \quad \forall (x, \xi) \in \R^{2n}$$
from which we deduce that $a \in \textbf{\textrm{SG}}^{\tau, \infty}_{\mu,\nu,s}(\R^{2n}).$\\
It remains to prove that $a \sim \sum\limits_{j \geq 0} a_j.$
Let us fix $N \in \N \setminus \{ 0 \}.$
We observe that if $(x, \xi) \in Q^e_{2RN^{\mu+\nu -1}},$ then
$$a(x, \xi) - \sum_{j < N} a_j (x,\xi) = \sum_{j \geq N} \varphi_j(x,\xi) a_j(x,\xi).$$
Thus we have
$$\left| \sum_{j \geq N} \der   \left[ \varphi_j(x,\xi) a_j(x,\xi) \right] \right| \leq 
C^{ |\alpha| + |\beta| +1}
\alpha ! \beta ! \px ^{- |\beta| -N} \pxi ^{\tau-|\alpha| -N} e^{c|x|^{\frac{1}{s}}}\sum_{j \geq N}H_{jN \alpha \beta} (x, \xi)$$
where
$$H_{jN \alpha \beta} (x, \xi)=
\sum_{\stackrel{\gamma \leq \alpha}{\delta \leq \beta}}
\frac{(\alpha - \gamma )!^{\mu-1}(\beta - \delta )!^{\nu -1}}{\gamma ! \delta !} \cdot C^{2j - |\gamma|-|\delta|}(j!)^{\mu+\nu -1} \langle x \rangle ^{|\delta| +N-j}
\langle \xi \rangle ^{|\gamma|+N -j} |D_x^{\delta}D_{\xi}^{\gamma}
\varphi_j (x,\xi)|.$$
Arguing as above we can estimate
$$H_{jN \alpha \beta} (x, \xi) \leq C_4^{2N + |\alpha| + |\beta|+1} (N!)^{\mu+\nu -1}
(\alpha !)^{\mu-1}(\beta!)^{\nu -1}$$
and this concludes the proof.
\end{proof}

Using the same argument it is easy to prove that if we start from a sum $\sum\limits_{j \geq 0} a_j \in FS^{m}_{\mu, \nu}(\R^{2n}$) (respectively in $FS^{\infty,\tau}_{\mu, \nu,s}(\R^{2n})$ or in $FS^{\infty}_{\mu, \nu,s}(\R^{2n}$)) we can find a symbol $a $ in $\Gamma^{m}_{\mu, \nu}(\R^{2n})$ (respectively in $\Gamma^{\infty, \tau}_{\mu, \nu,s}(\R^{2n})$ or in $\Gamma^{\infty}_{\mu, \nu,s}(\R^{2n})$) equivalent to $\sum\limits_{j \geq 0} a_j$ in the corresponding class.

\begin{Prop} \label{regas}
Let $\mu,\nu,s$ be real numbers such that $1<\mu \leq \nu$ and $s >\mu+\nu-1$ and let $a \in \textbf{\textrm{SG}}^{0,\infty}_{\mu, \nu, s}(\R^{2n})$ such that $ a \sim 0$ in $FS^{0,\infty}_{ \mu,\nu, s}(\R^{2n})$, then $a \in \mathcal{S}_{\mu+\nu-1}(\R^{2n})$. In particular, the operator $a(x,D)$ is $\mathcal{S}_{\mu+\nu-1}$-regularizing, that is it maps continuously $(\mathcal{S}_{\mu+\nu-1})'(\R^n)$ into $\mathcal{S}_{\mu+\nu-1}(\R^n)$.
\end{Prop}

\begin{proof}
 It is sufficient to prove that if $a \sim 0,$ then $a \in \mathcal{S}_{\mu+\nu-1} (\R^{2n}).$ This would imply that the Schwartz kernel of $a(x,D)$ belongs to $\mathcal{S}_{\mu+\nu-1}(\R^{2n})$ which gives the assertion. By Definition \ref{equiv}, there exist $B,C,\tau>0$ such that, for every $(x,\xi) \in \R^{2n}$ we have:
\begin{eqnarray*}\left| \der a(x,\xi) \right| &\leq & CA^{|\alpha|+|\beta|}(\alpha!)^\mu(\beta!)^{\nu} \pxi^{-|\alpha|}
\px^{-|\beta|} e^{c|x|^{1/s}} \cdot \inf_{0 \leq N \leq (B^{-1}\pxi \px)^{\frac{1}{\mu+\nu -1}}} \frac{A^{2N} (N!)^{\mu+\nu-1}}{\pxi^N \px^N} \\ 
&\leq& C A^{|\alpha|+|\beta|}(\alpha!)^\mu (\beta!)^{\nu} e^{c|x|^{\frac{1}{s}}} \exp{\left[-\tau (\px \pxi)^{\frac1{\mu+\nu-1}} \right]}
\\
&\leq& C A^{|\alpha|+|\beta|}(\alpha!)^\mu (\beta!)^{\nu} e^{c|x|^{\frac{1}{s}}} \exp{ \left[-\tau (|x|^{\frac{1}{\mu+\nu-1}}+|\xi|^{\frac{1}{\mu+\nu-1}}) \right]}.
\end{eqnarray*}
cf. \cite[Lemma 3.2.4]{Rodino}.
Since $\max\{\mu,\nu\} <\mu+\nu-1<s$, we get $$
\left| \der a(x,\xi) \right| \leq C'A^{|\alpha|+|\beta|}(\alpha!)^\mu (\beta!)^{\nu} \exp{ \left[-\frac{\tau}{2} (|x|^{\frac{1}{\mu+\nu-1}}+|\xi|^{\frac{1}{\mu+\nu-1}}) \right]}.$$
Hence $a \in \mathcal{S}_{\mu+\nu-1}(\R^{2n}).$
This concludes the proof.
\end{proof}

The next result concerns the regularity and decay properties of the Schwartz kernel of our operators far from the diagonal. 

\begin{Prop}
\label{ker}
Let $a \in \textbf{\textrm{SG}}_{\mu, \nu,s}^{\tau,\infty}(\R^{2n})$ with $s > \mu +\nu-1$. For $k \in (0,1)$ define: $$\Omega_k = \{(x,y) \in \R^{2n}:|x-y| > k\px \}.$$ Then the kernel $K$ of $a$, defined by 
$$K(x,y) =\int_{\R^n} e^{i \langle x-y,\xi \rangle} a(x,\xi) \dslash \xi,$$ is in $C^{\infty}(\R^{2n} \setminus \Delta)$, where $\Delta$ denotes the diagonal in $\R^{2n}$ and there exist positive constants $C,\tilde c$ depending on $k$ such that
\begin{equation}
\label{rego}
\left| D_x^{\beta} D_y^{\gamma} K(x,y) \right| \leq C^{|\beta|+|\gamma|+1}(\beta! \gamma!)^{\nu} \exp{ \left[-\tilde c (|x|^{\frac{1}{\mu+\nu-1}} + |y|^{\frac{1}{\mu+\nu-1}}) \right]}
\end{equation}
for every $(x,y) \in \overline{\Omega_k}$ and for every $\beta, \gamma \in \N^n.$ 
\end{Prop}

\textit{Proof of Proposition \ref{ker}.}
First we observe that for any fixed $R>0$ we can find a partition of unity $\psi_N(\xi)$ such that $$ supp \psi_0 \subset \{ \xi: \pxi \leq 3R \}$$
$$ supp \psi_N \subset \{\xi: 2RN^{\mu} \leq \pxi \leq 3R(N+1)^{
\mu} \}, N=1,2,...$$ and
$$\left| D_{\xi}^{\alpha} \psi_N (\xi) \right| \leq C^{|\alpha| +1}(\alpha !)^{\mu} \left[ R \sup(N^{\mu} ,1) \right]^{- |\alpha|}$$
for every $\alpha \in \N^n$ and for every $\xi \in \R^n.$ 
For every fixed $\theta \geq \nu$ we can write, for $u,v \in \mathscr{S}_s^\theta(\R^n):$
$$\langle K, v \otimes u \rangle = \sum_{N=0}^{\infty} \langle K_N, v \otimes u \rangle$$
with $$K_N(x,y) = \int_{\R^n} \fas a(x,\xi) \psi_N(\xi) \dslash \xi$$
so we may decompose
$$K = \sum_{N=0}^{\infty} K_N.$$
Let $k \in (0,1)$ and $(x,y) \in \overline{\Omega}_k.$ Let $h \in \{1,...,n \}$ such that $|x_h -y_h| \geq \frac{k}{n} \px.$
Then, integrating by parts infinitely many times, we get, for every $\alpha, \gamma \in \N^n:$
\beqsn 
D_x^{\alpha} D_y^{\gamma} K_N(x,y) &=& (-1)^{|\gamma|} \sum_{\beta \leq \alpha} \binom{\alpha}{\beta} \int_{\R^n}
\fas \xi^{\beta +\gamma} \psi_N (\xi) D_x^{\alpha - \beta} a(x,\xi) \dslash \xi 
\\
&=& (-1)^{|\gamma|+N}
\sum_{\beta \leq \alpha} \binom{\alpha}{\beta} (x_h -y_h)^{-N} \int_{\R^n}
\fas D_{\xi_h}^N \left[ \xi^{\beta +\gamma}  \psi_N (\xi) D_x^{\alpha - \beta} a(x,\xi) \right] \dslash \xi
\\
&=& (-1)^{|\gamma|+N}\cdot \frac{(x_h -y_h)^{-N}}{m_{2\mu, \zeta}(x-y)} 
\sum_{\beta \leq \alpha} \binom{\alpha}{\beta} \sum_{j=0}^{\infty} \frac{\zeta^j}{(j!)^{2\nu}} \int_{\R^n}
 \fas  \lambda_{hjN\alpha \beta \gamma}(x,\xi)\dslash \xi
 \eeqsn
with 
\begin{equation}
\label{lambda}
\lambda_{hjN\alpha \beta \gamma}(x,\xi) = (1-\Delta_{\xi})^j D_{\xi_h}^N \left[\xi^{\beta +\gamma}  \psi_N (\xi) D_x^{\alpha - \beta} a(x,\xi) \right].
\end{equation}
Let $e_h$ be the h-th vector of the canonical basis of $\R^n$ and $ \beta_h = \langle \beta, e_h \rangle, \gamma_h = \langle \gamma, e_h \rangle.$ By Leibniz formula  we obtain 
\beqsn
\lambda_{hjN\alpha \beta \gamma}(x,\xi) &=& \sum_{\stackrel{N_1+N_2+N_3 =N}{N_1 \leq \beta_h + \gamma_h}}(-i)^{N_1}
\frac{N!}{N_1!N_2!N_3!} \cdot \frac{(\beta_h +\gamma_h)!}{(\beta_h+\gamma_h-N_1)!} \cdot 
\\
&&\cdot (1-\Delta_{\xi})^j \left[\xi^{\beta +\gamma-N_1e_h}  D_{\xi_h}^{N_2}\psi_N (\xi) D_{\xi_h}^{N_3} D_x^{\alpha - \beta} a(x,\xi) \right].
\eeqsn
Hence
\beqsn
\left| \lambda_{hjN\alpha \beta \gamma}(x,\xi) \right|& \leq &C \sum_{\stackrel{N_1+N_2+N_3 =N}{N_1 \leq \beta_h + \gamma_h}} \frac{N!}{N_1!N_2!N_3!} \cdot \frac{(\beta_h +\gamma_h)!}{(\beta_h+\gamma_h-N_1)!}
C_1^{|\alpha-\beta| +N_2 +N_3} \cdot
\\
&&\cdot (N_2!N_3!)^{\mu} \left[(\alpha-\beta)!\right]^{\nu} C_2^j (j!)^{2\mu} \left(\frac{1}{RN^{\mu}} \right)^{N_2} \pxi^{\tau+|\beta|+|\gamma|-N_1- N_3} e^{c|x|^{\frac{1}{s}}}. 
\eeqsn
We observe that on the support of $\psi_N$ we have $2RN^{\mu} \leq \pxi \leq 3R(N+1)^{\mu}.$ Thus from standard factorial inequalities it follows that
$$\left| \lambda_{hjN\alpha \beta \gamma}(x,\xi) \right| \leq  C_1^{|\alpha|+ |\gamma|+1}(\alpha!\gamma!)^{\nu}
C_2^j (j!)^{2\mu} \left( \frac{C_3}{R} \right)^N e^{c|x|^{\frac{1}{s}}} $$
with $C_3$ independent of $R.$ From these estimates, choosing $\zeta < \frac{1}{C_2},$ we deduce that
$$\left| D_x^{\alpha} D_y^{\gamma}K_N(x,y) \right| \leq C_4^{|\alpha|+ |\gamma|+1}(\alpha!\gamma!)^{\nu}
\left( \frac{C_5}{R} \right)^N \exp{ \left[ c |x|^{\frac{1}{s}} - c\zeta^{\frac{1}{\mu}} |x-y|^{\frac{1}{\mu}} \right]}$$
with $C_5=C_5(k)$ independent of $R.$
Finally, we observe that since $\mu<\mu+\nu-1<s$, then there exists $c_k>0$ such that
$$\sup_{(x,y) \in \Omega_k} \exp{\left[ c_k(|x|^{\frac{1}{\mu+\nu-1}}+|y|^{\frac{1}{\mu+\nu-1}})- c\zeta^{\frac{1}{\nu}}|x-y|^{\frac{1}{\mu}} +c|x|^{\frac{1}{s}} \right]} \leq 1.$$ 
Then, choosing $R$ sufficiently large, we obtain the estimates \eqref{rego}.
\qed
\\

\begin{Th}\label{reduction}
Let $\mu,\nu,s, \tau \in \R$ such that $\mu >1, \nu >1, s > \mu+\nu-1$ and let $A$ be defined by \eqref{opampl} for some $a \in \pig$. Then there exists $b\in \gam$ such that $A= b(x,D) +R$ for some $\mathcal{S}_{\mu+\nu-1}$-regularizing operator $R$. Moreover, we have $b \sim \sum\limits_{j \geq 0} b_j$ in $FS^{\tau,\infty}_{\mu,\nu,s}(\R^{2n})$, where
\begin{equation} \label{asympt}
b_j(x,\xi) = \sum_{|\alpha|=j} (\alpha!)^{-1} \partial_\xi^\alpha D_y^\alpha a(x,y,\xi)_{\left|_{y=x} \right.}.
\end{equation}
\end{Th}

\begin{proof}

Let $\chi \in C^{\infty}(\R^{2n})$ such that
\begin{equation}\label{ipoh}
\chi(x,y) = 
\left\{
\begin{array}{lll}
1 &\text{if} & |x-y| \leq \frac{1}{4}\px \\
0 &\text{if} & |x-y| \geq \frac{1}{2}\px
\end{array}
\right.
\end{equation}

\noindent
and 
$$ \left|D_x^{\beta} D_y^{\gamma} \chi (x,y) \right|
\leq C^{|\beta| +|\gamma| +1} ( \beta ! \gamma !)^{\nu}$$
for all $\beta, \gamma \in \N^n$ and $(x,y) \in \R^{2n}.$
We may decompose $a$ as the sum of two elements of $\pig$ writing$$
a(x,y, \xi) = \chi(x,y)a(x,y,\xi)+ (1-\chi(x,y))a(x,y,\xi).$$
Furthermore, it follows from Theorem \ref{ker} that $(1-\chi(x,y))a(x,y,\xi)$ defines a $\nu$-regularizing operator.
Hence, eventually perturbing $A$ with a $\nu$-regularizing operator, we can assume that $a(x,y,\xi) $ is supported on $\left( \R^{2n} \setminus \Omega_{\frac{1}{2}}\right) \times \R^n,$ where $\Omega_{\frac{1}{2}}$ is defined as in Theorem \ref{ker}. \\
It is trivial to verify that $\sum\limits_{j \geq 0}a_j$ defined by \eqref{asympt} belongs to $FS^{\tau,\infty}_{\mu,\nu,s}(\R^{2n}).$ By Proposition \ref{sviluppo} we can find a sequence $\varphi_j \in C^{\infty}(\R^{2n})$ depending on a parameter $R$ such that $$p(x,\xi) = \sum_{j \geq 0} \varphi_j(x,\xi)a_j(x,\xi)$$ defines an element of $\gam$ for $R$  large and $p \sim \sum\limits_{j \geq 0}a_j$ in $FS^{\tau,\infty}_{\mu,\nu,s}(\R^{2n}).$
Let $P= p(x,D).$ To prove the Theorem it is sufficient to show that the kernel $K(x,y)$ of $A-P$ is in $\mathcal{S}_{\mu+\nu-1}(\R^{2n}).$ \\
We can write
\beqsn a(x, y, \xi)-p(x,\xi) &=& (1- \varphi_0(x,\xi))a(x,y,\xi) 
\\
&&+ \sum_{N=0}^{\infty}(\varphi_N - \varphi_{N+1})(x,\xi) \left( a(x,y,\xi)- \sum_{j \leq N}a_j(x,\xi) \right).
\eeqsn
Consequently,
\begin{equation}
\label{nuclei}
K(x, y) = \overline{K}(x,y) + \sum_{N=0}^{\infty} K_N(x,y)
\end{equation}
where $$\overline{K}(x,y) = \int_{\R^n} \fas (1- \varphi_0(x,\xi))a(x,y,\xi)\dslash \xi,$$
$$ K_N(x,y)= \int_{\R^n} \fas (\varphi_N - \varphi_{N+1})(x,\xi) \left( a(x,y,\xi)- \sum_{j \leq N}a_j(x,\xi) \right)\dslash \xi.$$
A power expansion in the second argument gives for $N=1,2,...$
$$a(x,y,\xi)= \sum_{|\alpha| \leq N} (\alpha !)^{-1} (y-x)^{\alpha} \partial_y^{\alpha} a(x,x,\xi) + \sum_{|\alpha| =N+1} (\alpha !)^{-1} (y-x)^{\alpha} w_{\alpha}(x,y,\xi)$$
with $$w_{\alpha}(x,y, \xi)= (N+1) \int_0^1 \partial_y^{\alpha}a (x, x+t(y-x), \xi)(1-t)^{N}dt.$$
In view of our definition of the $a_j(x,\xi),$ integrating by parts, we obtain that
\beqsn 
K_N(x,y)&=&W_N(x,y)+ \sum_{1 \leq |\alpha| \leq N} \sum_{0 \neq \beta \leq \alpha} \frac{1}{\beta ! (\alpha - \beta) !} \cdot
\\
&&\cdot \int_{\R^n} \fas D_{\xi}^{\beta}
(\varphi_N - \varphi_{N+1})(x,\xi) (D_{\xi}^{\alpha -\beta}\partial_y^{\alpha}a)(x,x,\xi)\dslash \xi,
\eeqsn where for all $N=1,2,...$
\beqsn
W_N(x,y)&=&  \sum_{|\alpha| =N+1} \sum_{\beta \leq \alpha}
\frac{1}{\beta ! (\alpha - \beta )!} \cdot
\\
&&\cdot \int_{\R^n} \fas D_{\xi}^{\beta}
(\varphi_N - \varphi_{N+1})(x,\xi) D_{\xi}^{\alpha -\beta}w_{\alpha}(x,y,\xi) \dslash \xi.
\eeqsn  \\
Using an absolute convergence argument, we may re-arrange the sums under the integral sign. We also observe that
$$\sum_{N \geq |\alpha|}D_{\xi}^{\beta}(\varphi_N - \varphi_{N+1})(x,\xi)= D_{\xi}^{\beta} \varphi_{|\alpha|}(x,\xi).$$
Then we have
$$K=\overline{K} + \sum_{\alpha \neq 0}I_{\alpha} + \sum_{N=0}^{\infty} W_N$$
where $$I_{\alpha}(x,y)=\sum_{0 \neq \beta \leq \alpha} \frac{1}{\beta ! (\alpha - \beta )!}  \int_{\R^n} \fas D_{\xi}^{\beta}
\varphi_{|\alpha|}(x,\xi)D_{\xi}^{\alpha - \beta} \partial_y^{\alpha} a(x,x,\xi) \dslash \xi$$
and we may write $W_0(x,y) $ for $K_0(x,y).$ To conclude the proof, it is sufficient to prove that $\overline{K}$,
$\sum\limits_{\alpha \neq 0}I_{\alpha}$, $\sum\limits_{N=0}^{\infty} W_N \in \mathcal{S}_{\mu+\nu-1}(\R^{2n}).$ 
First of all, we have to estimate the derivatives of $\overline{K}$ for $(x,\xi) \in \textit{supp}(
1-\varphi_0(x,\xi)),$ i.e. for $\px \leq R, \pxi \leq R.$ We have
\beqsn
\left|x^k y^h D_x^{\delta}D_y^{\gamma} \overline{K}(x,y) \right| &=&  \left| x^k y^h \sum_{\stackrel{\gamma_1 +\gamma_2 = \gamma}{\delta_1 +\delta_2 +\delta_3 = \delta}} \frac{\gamma! \delta!}{\gamma_1!\gamma_2!\delta_1! \delta_2! \delta_3!}
\right. \cdot 
\\
&&\cdot \left.
(-1)^{|\gamma_1|} \int_{\R^n} \fas \xi^{\gamma_1 +\delta_1}D_x^{\delta_2}D_y^{\gamma_2}a(x,y,\xi) D_x^{\delta_3} (1-\varphi_0(x,\xi))d\xi \right| 
\\
&\leq&|x|^{|k|}|y|^{|h|} \sum_{\stackrel{\gamma_1 +\gamma_2 = \gamma}{\delta_1 +\delta_2 +\delta_3 = \delta}} \frac{\gamma! \delta!}{\gamma_1!\gamma_2!\delta_1! \delta_2! \delta_3!} C^{|\gamma_2|+|\delta_2|+|\delta_3|} (\gamma_2!\delta_2! \delta_3!)^{\nu}
\langle x-y \rangle^{|\gamma_2+\delta_2|}
\cdot 
\\
&&\cdot \exp{ \left[ a (|x|^{\frac{1}{s}} +|y|^{\frac{1}{s}}) \right]} \int_{\pxi \leq R}
\pxi^{\tau+|\gamma_1 +\delta_1|} \dslash \xi.
\eeqsn
Now, $a(x,y,\xi) $ is supported on $\left( \R^{2n} \setminus \Omega_{\frac{1}{2}}\right) \times \R^n$ and in this region $|y| \leq \frac{3}{2}\px$ so, there exist constants $C_1, C_2 >0$ depending on $R$ such that
$$\sup_{(x,y) \in \R^{2n}} \left |x^k y^h D_x^{\delta}D_y^{\gamma} \overline{K}(x,y) \right| \leq C_1 R^{|k|+|h|}
C_2^{|\gamma|+|\delta|}(\gamma! \delta!)^{\nu},$$ so $ \overline{K} \in \mathcal{S}_{\nu}(\R^{2n}) \subset \mathcal{S}_{\mu+\nu-1}(\R^{2n}).$
Consider now 
\beqsn
x^k y^h D_x^{\delta}D_y^{\gamma} I_{\alpha}(x,y)&=&  \sum_{0 \neq \beta \leq \alpha} \frac{1}{\beta! (\alpha -\beta)!} \sum_{\delta_1+\delta_2+\delta_3 = \delta} \frac{\delta !}{\delta_1!\delta_2! \delta_3!}(-1)^{|\gamma|}
x^k y^h \cdot
\\
&&\cdot \int_{\R^n} \fas \xi^{\gamma +\delta_1} D_x^{\delta_2}D_{\xi}^{\beta} \varphi_{|\alpha|}(x,\xi) 
D_x^{\delta_3} [(D_{\xi}^{\alpha -\beta}\partial_y^{\alpha}) a)(x,x,\xi)] d\xi
\\
&=& \sum_{0 \neq \beta \leq \alpha} \frac{1}{\beta! (\alpha -\beta)!} \sum_{\delta_1+\delta_2+\delta_3 = \delta} \frac{\delta !}{\delta_1!\delta_2! \delta_3!}(-1)^{|\gamma|} (-i)^h x^k 
\cdot 
\\
&&\cdot \int_{\R^n} e^{-i \langle y,\xi \rangle} \partial_{\xi}^h \left[ e^{i\langle x,\xi \rangle} \xi^{\gamma +\delta_1} D_x^{\delta_2}D_{\xi}^{\beta} \varphi_{|\alpha|}(x,\xi)D_x^{\delta_3}[( D_{\xi}^{\alpha -\beta}
 \partial_y^{\alpha} a)(x,x,\xi)] \right]\dslash \xi.
 \eeqsn
We need the estimates for $(x,\xi) \in \textit{supp} D_{\xi}^{\beta} \varphi_{|\alpha|}(x,\xi) \subset \overline{Q}_{2R|\alpha|^{\mu+\nu-1}} \setminus Q_{R|\alpha|^{\mu+\nu-1}}.$ Then, there exist $C_1,C_2,C_3>0$ such that
\beqsn
\left|x^k y^h D_x^{\delta}D_y^{\gamma} I_{\alpha}(x,y) \right| &\leq& C_1^{|h|+|k|+1}C_2^{|\alpha|}C_3^{|\gamma|+|\delta|} (k! h! \gamma!\delta!)^{s}(\alpha!)^{\nu} \px^{-|\alpha|} \cdot
\\
&&\cdot \sum_{0 \neq \beta \leq \alpha} (\beta !)^{\mu -1} \left[(\alpha-\beta)! \right]^{\mu - 1} \left(\frac{1}{R|\alpha|^{\mu+\nu -1}} \right)^{|\beta|} e^{c|x|^{\frac{1}{s}}}\int_{\pxi \leq 2R|\alpha|^{\mu+\nu-1}} \pxi^{m-|\alpha- \beta|}\dslash \xi
\eeqsn
with $C_2$ independent of $R.$ Now, since $ \mu+\nu-1 <s$ and $|x| \leq R|\alpha|^{\mu+\nu-1}$, we have that 
$$C_2^{|\alpha|}(\alpha !)^{\nu} \px^{-|\alpha|}\hskip-0.2cm\sum_{0 \neq \beta \leq \alpha} (\beta !)^{\nu -1} \left[(\alpha-\beta)! \right]^{\nu - 1} \left(\frac{1}{R|\alpha|^{\mu+\nu -1}} \right)^{|\beta|}\hskip-0.2cme^{c|x|^{\frac{1}{s}}} \int_{\pxi \leq 2R|\alpha|^{\mu+\nu-1}}\hskip-0.2cm \pxi^{-|\alpha- \beta|}d\xi \leq \left( \frac{C_4}{R} \right)^{|\alpha|}$$
with $C_4$ independent of $R.$ Finally, we conclude that
$$\sup_{(x,y) \in \R^{2n}} \left|x^k y^h D_x^{\delta}D_y^{\gamma} I_{\alpha}(x,y) \right| \leq C^{|h|+|k|+1}C_2^{|\gamma|+|\delta|} (k!h!\gamma! \delta!)^{\mu+\nu-1} \left( \frac{C_4}{R} \right)^{|\alpha|}.$$
Choosing $R> C_4,$ we obtain that $\sum\limits_{\alpha \neq 0}I_{\alpha} \in \mathcal{S}_{\mu+\nu-1}(\R^{2n}).$ \\ Arguing as for $I_{\alpha},$ we can prove that also$$
\sup_{(x,y) \in \R^{2n}} \left|x^k y^h D_x^{\delta}D_y^{\gamma} W_N(x,y) \right| \leq C_1^{|h|+|k|+1}C_2^{|\gamma|+|\delta|} (h!k!\gamma! \delta!)^{2\nu-1} \left( \frac{C}{R} \right)^N$$
 with $C$ independent of $R,$ which gives, for $R$ sufficiently large, that $\sum\limits_{N=0}^{\infty} W_N $ is in $  
\mathcal{S}_{\mu+\nu-1}(\R^{2n}).$ This concludes the proof.
\end{proof}

As a consequence of the previous theorem we obtain the two following results.

\begin{Prop}
\label{transpose}
Let $a \in \textbf{\textrm{SG}}^{\tau,\infty}_{\mu,\nu, s}(\R^n)$ with $s > \mu+\nu-1$ and let $^t A$ and $A^\ast$ be respectively the transpose and the $L^2$-adjoint of $A=a(x,D)$ defined by 
\begin{equation}
\label{rel}
\langle ^t Au,v \rangle = \langle u, Av \rangle, \quad u \in (\mathscr{S}_s^\theta)^\prime(\R^n), v \in \mathscr{S}_s^\theta(\R^n).
\end{equation}
and 
\begin{equation}\label{rel2}
\langle A^\ast u,v \rangle_{L^2} = \langle u, Av \rangle_{L^2}, \quad u,v \in \mathscr{S}_s^\theta(\R^n).
\end{equation}

Then, $^t A =B_1+R_1$ and $A^\ast =B_2+R_2$ where $R_j, j=1,2,$ are $\mathcal{S}_{\mu+\nu-1}$-regularizing operators and $B_j=b_j(x,D), j=1,2,$ with $b_j \in \textbf{\textrm{SG}}^{\infty}_{\mu,\nu ,s}(\R^n)$
with $$b_1(x,\xi) \sim \sum_{j \geq 0} \sum_{|\alpha|=j} (\alpha !)^{-1} \partial_{\xi}^{\alpha}D_x^{\alpha} a(x,-\xi)$$
and $$b_2(x,\xi) \sim \sum_{j \geq 0} \sum_{|\alpha|=j} (\alpha !)^{-1} \partial_{\xi}^{\alpha}D_x^{\alpha} \overline{a(x,\xi)}$$
in $FS^{\tau, \infty}_{\mu,\nu ,s}(\R^{2n}).$
\end{Prop}

\begin{proof}
By the formula \eqref{rel}, $^tP$ is defined by
$$^tAu(x) =  \int_{\R^{2n}} e^{i\langle x-y,\xi \rangle} a(y,-\xi)u(y)dy \dslash \xi, \quad u \in \mathscr{S}_s^\theta(\R^n). $$
Thus, $^tA$ is an operator with amplitude $a(y,-\xi) \in \Pi^{\tau, \infty}_{ \mu,\nu ,s}(\R^{3n})$. By Theorem \ref{reduction},
$^tA =b_1(x,D)+R_1$ where $R_1$ is $\mathcal{S}_{\mu+\nu-1}$-regularizing  and $b_1\in \textbf{\textrm{SG}}^{\tau, \infty}_{\mu,\nu ,s}(\R^n),$ with $$b_1(x,\xi) \sim \sum_{j \geq 0} \sum_{|\alpha| =j} (\alpha !)^{-1} \partial_{\xi}^{\alpha} D_x^{\alpha}a(x,-\xi). $$
The proof is similar for the adjoint.
\end{proof}

\begin{Th} \label{composition}
Let $a \in \textbf{\textrm{SG}}_{ \mu, \nu,s}^{\tau,\infty}(\R^{2n}), b \in \textbf{\textrm{SG}}_{\mu,\nu,s}^{\tau',\infty}(\R^{2n})$, with $s >\mu+\nu-1$. Then there exists $c \in \textbf{\textrm{SG}}_{ \mu, \nu,s}^{\tau+\tau',\infty}(\R^{2n})$  such that $a(x,D)b(x,D) =c(x,D) + R$, where $R$ is a $\mathcal{S}_{\mu+\nu-1}$-regularizing operator and
$$c(x,\xi) \sim \sum_{j\geq 0}\sum_{|\alpha|=j}\alpha!^{-1}\partial_\xi^\alpha a(x,\xi) D_x^\alpha b(x,\xi) \qquad \textrm{in} \quad FS^{\tau+\tau',\infty}_{\mu,\nu,s}(\R^{2n}).$$
\end{Th}

\begin{proof}
We can write $B= \, ^t \hskip-1pt (^tB).$ Then, by Theorem \ref{reduction} and Proposition \ref{transpose}, $B=B_1+ R_1,$ where $R_1$ is $\mathcal{S}_{\mu+\nu-1}$-regularizing and 
\begin{equation}
\label{passo}
B_1u(x) =  \int_{\R^{2n}} e^{i \langle x-y, \xi \rangle} b_1(y,\xi)u(y)dy \dslash \xi
\end{equation}
with $b_1(y,\xi) \in \gam, b_1(y,\xi) \sim \sum\limits_{\alpha} (\alpha!)^{-1} \partial_{\xi}^{\alpha}D_y^{\alpha}b(y, -\xi).$
From \eqref{passo} it follows that
$$\widehat{B_1u}(\xi) = \int_{\R^n} e^{-i \langle y, \xi \rangle} b_1(y,\xi) u(y)dy,$$
from which we deduce that
$$AB u(x) = \int_{\R^{2n}} e^{i\langle x-y,\xi \rangle} a(x,\xi)b_1(y,\xi)u(y)dy \dslash \xi +AR_1u(x).$$
We observe that $a(x,\xi)b_1(y,\xi) \in \Pi^{\tau+\tau',\infty}_{\mu,\nu,s}(\R^{3n}),$ then we may apply Theorem \ref{reduction} and obtain that
$$ABu(x)= c(x,D)u(x)+Ru(x)$$ wher $R$ is $\mathcal{S}_{\mu+\nu-1}$-regularizing and $c \in \textbf{\textrm{SG}}^{\tau+\tau',\infty}_{\mu,\nu ,s}(\R^n)$ with 
$$c(x,\xi) \sim \sum_{\alpha}(\alpha !)^{-1} \partial_{\xi}^{\alpha}a(x,\xi) D_x^{\alpha} b(x,\xi).$$
\end{proof}

\begin{Rem} \label{refcomp}
It is easy to prove that if $a\in \textbf{\textrm{SG}}^{\tau,\infty}_{\mu,\nu,s}(\R^{2n}), b \in \textbf{\textrm{SG}}_{\mu,\nu,s}^{\tau',\infty}(\R^{2n})$ with $a \sim \sum\limits_{j \geq 0}a_j $ in $FS^{\tau,\infty}_{\mu,\nu,s}(\R^{2n})$ and $b \sim \sum\limits_{j \geq 0}b_j $ in $FS^{\tau',\infty}_{\mu,\nu,s}(\R^{2n})$,
then $a(x,D)b(x,D) = c(x,D) +R(x,D)$ where $R$ is $\mathcal{S}_{\mu+\nu-1}$-regularizing and $$c(x,\xi) \sim \sum_{j \geq 0} \sum_{|\alpha|+h+k=j} (\alpha!)^{-1}
\partial_\xi^\alpha a_h(x,\xi) D_x^\alpha b_k(x,\xi) \qquad \textrm{in} \qquad  FS^{\tau+\tau', \infty}_{\mu,\nu,s}(\R^{2n}).$$
\end{Rem}

\begin{Rem}
We remark that replacing the condition $s >\mu+\nu-1$ by $\theta >\mu+\nu-1$ (resp. by $\min \{s, \theta \} >\mu+\nu-1$), analogous versions of Theorems \ref{reduction}, \ref{transpose},\ref{composition} can be formulated and proved for the class  $\textrm{\textbf{SG}}_{\mu,\nu,\theta}^{\infty, \tau}(\R^{2n})$ (resp. $\textrm{\textbf{SG}}_{\mu,\nu,s, \theta}^{\infty}(\R^{2n})$). Notice that if  $\min \{s, \theta \} >\mu+\nu-1$, then the remainder terms in Proposition \ref{regas} and in Theorems \ref{reduction}, \ref{transpose},\ref{composition} are in turn $(\mathscr{S}_s^\theta)$-regularizing, that is they map $(\mathscr{S}_s^\theta)'(\R^n)$ into $\mathscr{S}_s^\theta (\R^n)$, since $\mathcal{S}_{\mu+\nu-1} \subset \mathscr{S}_s^\theta (\R^n)$.  We do not give the proof of these parallel results since they follow the same arguments as in the proofs of Theorems \ref{reduction}, \ref{transpose},\ref{composition}. On the other hand we can use them to prove the following conjugation theorems. 
\end{Rem}

\begin{Prop} \label{conjxi}
Let $p \in \textrm{\textbf{SG}}^m_{\mu,\nu} (\R^{2n}),$ $m=(m_1,m_2)\in\R^2.$ Then for every $\rho \in \R$ and for every $\theta >\mu+\nu-1$ the operator 
$$p_{1,\rho}(x,D) = e^{\rho \pd^{\frac{1}{\theta}}}\circ p(x,D) \circ  e^{-\rho \pd^{\frac{1}{\theta}}} = p(x,D) + q_\rho (x,D)+ \tilde{r}_1(x,D),$$
for some $q_\rho \in \textrm{\textbf {SG}}_\nu^{(m_1-1+1/\theta, m_2-1)} (\R^{2n})$ and $ \tilde{r}_1 \in \mathcal{S}_{\mu+\nu-1}(\R^{2n}).$
\end{Prop}

\begin{proof}
First of all we observe that by Proposition \ref{appaelambda} we have 
that $ e^{\rho \pxi^{\frac{1}{\theta}}}$ and $p(x,\xi) e^{-\rho \pxi^{\frac{1}{\theta}}}$ both belong to $\textrm{\textbf{SG}}^{\infty,m_2}_{\mu,\nu,\theta} (\R^{2n})$. Then the symbol of the operator $p_{1,\rho}(x,D)$
is such that 
$$p_{1,\rho}(x,\xi)=p(x,\xi)+ q_\rho(x,\xi)+ \tilde{r}_1(x,D)$$ with $\tilde{r}_1 \in \mathcal{S}_{\mu+\nu-1}(\R^{2n})$ and  $q_\rho \sim \sum\limits_{j \geq 1} q_{\rho, j}$ in $FS^{\infty, m_2}_{\mu,\nu,\theta}(\R^{2n})$, where
\begin{eqnarray} \label{sviluppoconiugazionexi}
q_{\rho,j}(x,\xi) &= &  \sum_{|\alpha|=j} (\alpha!)^{-1} \partial_\xi^\alpha e^{\rho \pxi^{\frac{1}{\theta}}}D_x^\alpha \left(p(x,\xi) e^{-\rho \pxi^{\frac{1}{\theta}}}\right) \\
 \nonumber &=&  \sum_{|\alpha|=j} (\alpha!)^{-1} e^{-\rho \pxi^{\frac{1}{\theta}}} \partial_\xi^\alpha e^{\rho \pxi^{\frac{1}{\theta}}}D_x^\alpha p(x,\xi). 
\end{eqnarray}
Estimating by Fa\`a di Bruno formula  we obtain 
that 
$$e^{-\rho \pxi^{\frac{1}{\theta}}} \partial_\xi^\alpha e^{\rho \pxi^{\frac{1}{\theta}}} = \sum_{h=1}^{|\alpha|} \frac{\rho^h}{h!} \sum_{\stackrel{\alpha_1+\ldots +\alpha_h =\alpha}{\alpha_i \neq 0}} \frac{\alpha!}{\alpha_1! \ldots \alpha_h!} \prod_{\mu=1}^h \partial_\xi^{\alpha_\mu} \pxi^{\frac{1}{\theta}}$$
from which it follows that
\begin{eqnarray*}
|q_{\rho,j}(x,\xi)|& \leq &   \sum_{|\alpha|=j} \sum_{h=1}^{|\alpha|} \frac{\rho^h}{h!} \sum_{\stackrel{\alpha_1+\ldots +\alpha_h =\alpha}{\alpha_i \neq 0}} C^{|\alpha|+1} \alpha!^\nu
(\alpha_1! \ldots \alpha_h! )^{\mu-1} \pxi^{ m_1+\frac{h}{\theta}-|\alpha|} \px^{m_2-|\alpha|}
\\
& \leq &
 C_1^{2j+1}(j!)^{\mu+\nu-1} \pxi^{m_1-\left(1-\frac{1}{\theta} \right)j} \px^{m_2-j}.
\end{eqnarray*}
Similarly we can estimate the derivatives of $q_{\rho,j}$ and we obtain that 
$$|\partial_\xi^\gamma \partial_x^\delta q_{\rho,j}(x,\xi) | \leq C^{|\gamma|+|\delta|+2j+1} (\gamma!)^\mu(\delta!)^{\nu}(j!)^{\mu+\nu-1} \pxi^{m_1-\left(1-\frac{1}{\theta} \right)j-|\gamma|} \px^{m_2-j-|\delta|}, \qquad j \geq 1,$$
for some positive constant $C$ independent of $j, \gamma, \delta.$
\\
 Then we can repeat readily the argument of the proof of Proposition \ref{sviluppo} by replacing $Q_{Bj^{\mu+\nu-1}}$ by the set
$$Q^\theta_{Bj^{\mu+\nu-1}}= \{ (x,\xi) \in \R^{2n}: \pxi^{1-\frac{1}{\theta}} < Bj^{\mu+\nu-1} \quad \textit{and} \quad \px < Bj^{\mu+\nu-1} \}$$ and $Q_{Bj^{\mu+\nu-1}}^e$ by 
$Q_{Bj^{\mu+\nu-1}}^{\theta,e}= \R^{2n} \setminus Q^\theta_{Bj^{\mu+\nu-1}}.$
Then we obtain that $q_\rho \in\textbf{\textrm{SG}}_{\mu,\nu}^{(m_1+1/\theta-1,m_2-1)}(\R^{2n})$. Details are left to the reader. 
\end{proof}

\begin{Prop} \label{congx}
Let $p \in \textrm{\textbf{SG}}^m_{\mu,\nu} (\R^{2n})$ for some $m \in \R^2$. Then for every $\rho \in \R$ and for every $s >\mu+\nu-1$ the operator 
$$p_{2,\rho}(x,D) = e^{\rho \px^{\frac{1}{s}}} \circ p(x,D) \circ e^{-\rho \px^{\frac{1}{s}}} = p(x,D) + r_\rho (x,D)+ \tilde{r}_2(x,D),$$
for some $r_\rho \in \textrm{\textbf {SG}}_{\nu}^{(m_1-1, m_2-1+1/s)} (\R^{2n})$ and $\tilde{r}_{2} \in \mathcal{S}_{\mu+\nu-1}(\R^{2n})$.
\end{Prop}

\begin{proof}
We observe that the symbols $e^{\rho \px^{\frac{1}{s}}} p(x,\xi)$ and $e^{-\rho \px^{\frac{1}{s}}} $ are both in $\textrm{\textbf{SG}}^{m_1,\infty}_{\mu,\nu,s}(\R^{2n})$ and that, by Theorem \ref{composition} the symbol of the composed operator is, modulo terms in $\mathcal{S}_{\mu+\nu-1}(\R^{2n}),$ the sum of $p(x,\xi)$ and of a symbol $r_\rho(x,\xi)$ such that
\begin{eqnarray*} r_\rho(x,\xi) & \sim & \sum_{j \geq 1} \sum_{|\alpha|=j} (\alpha)^{-1}\partial_\xi^\alpha \left( e^{\rho  \px^{\frac{1}{s}}}p(x,\xi) \right) D_x^\alpha  e^{-\rho  \px^{\frac{1}{s}}} \\
&=& \sum_{j \geq 1} \sum_{|\alpha|=j} (\alpha)^{-1} \partial_\xi^\alpha p(x,\xi)  e^{\rho  \px^{\frac{1}{s}}} D_x^\alpha  e^{-\rho  \px^{\frac{1}{s}}} \qquad \textrm{in} \quad FS^{m_1,\infty}_{\mu,\nu,s}(\R^{2n}).
\end{eqnarray*}
Hence we can proceed as in the proof of Proposition \ref{conjxi} by simply interchanging the roles of $x$ and $\xi$. We leave the details to the reader.
\end{proof}

By combination of the previous two propositions we obtain the following continuity result.
 
\begin{Th} \label{doubleconj}
Let $p \in \textrm{\textbf{SG}}^{m'}_{\mu,\nu}(\R^{2n})$ for some $m' =(m'_1,m'_2) \in \R^2$. Then, for every $m, \rho \in \R^2$ and $s, \theta$ such that $\min\{s,\theta\}>\mu+\nu-1$ the operator $p(x,D)$ extends to a continuous map from $H^m_{\rho,s,\theta}(\R^n)$ into $H^{m-m'}_{\rho,s, \theta}(\R^n)$.
\end{Th}

\begin{proof}
Applying first Proposition \ref{conjxi}
$$  e^{\rho_1 \pd^{\frac{1}{\theta}}} \circ p(x,D)\circ e^{-\rho_1 \pd^{\frac{1}{\theta}}} =  p(x,D) + q_{\rho_1}(x,D) + \tilde{r_1}(x,D),$$
where $q_{\rho_1} \in  \textrm{\textbf {SG}}_{\mu,\nu}^{(m'_1-1+1/\theta, m'_2-1)} (\R^{2n})$, and $\tilde{r}_1 \in \mathcal{S}_{\mu+\nu-1}(\R^{2n}).$
We now apply Proposition \ref{congx} and we get
$$ e^{\rho_2 \px^{\frac{1}{s}}} \circ ( p(x,D) + q_{\rho_1}(x,D) ) \circ e^{-\rho_2 \px^{\frac{1}{s}}} = p(x,D) +  q_{\rho_1}(x,D) + r_{\rho_2}(x,D)+ q_{\rho_1,\rho_2}(x,D) + \tilde{r}_2(x,D),$$
with $ r_{\rho_2} \in \textrm{\textbf {SG}}_{\mu, \nu}^{(m'_1-1, m'_2-1+1/s)} (\R^{2n})$, $ q_{\rho_1,\rho_2} \in \textrm{\textbf {SG}}_{\mu, \nu}^{(m'_1-2+ 1/\theta, m'_2-2+1/s)} (\R^{2n})$  and $\tilde{r}_2 \in \mathcal{S}_{\mu+\nu-1}(\R^{2n})$.
In particular, we have that the operator $$q_\rho(x,D)= p(x,D) + q_{\rho_1}(x,D) + r_{\rho_2}(x,D)+q_{\rho_1, \rho_2}(x,D)$$ is continuous from $H^m$ into $H^{m-m'}$ for every $m \in \R^2.$ Moreover, since $\min\{s, \theta\} >\mu+\nu-1$ then 
$ \mathcal{S}_{\mu+\nu-1}(\R^{2n}) \subset \mathscr{S}_s^\theta(\R^{2n})$, hence the operators $\tilde{r}_j, j=1,2$ map $(\mathscr{S}_s^\theta)'  (\R^n)$ into $\mathscr{S}_s^\theta (\R^n)$. Hence, the same holds for the operators
$ e^{\rho_2 \px^{\frac{1}{s}}} \circ \tilde{r}_j(x,D) \circ e^{-\rho_2 \px^{\frac{1}{s}}}, j=1,2.$ Hence we have
\begin{eqnarray*}
\|p(x,D) u\|_{H^{m-m'}_{\rho,s, \theta}} &=& \|  e^{\rho_2 \px^{\frac{1}{s}}} e^{\rho_1 \pd^{\frac{1}{\theta}}} \circ p(x,D) \circ e^{-\rho_1 \pd^{\frac{1}{\theta}}}  e^{-\rho_2 \px^{\frac{1}{s}}}(e^{\rho_2 \px^{\frac{1}{s}}} e^{\rho_1 \pd^{\frac{1}{\theta}}} u) \|_{H^{m-m'}}\\ &\leq&
\| q_\rho(x,D) (e^{\rho_2 \px^{\frac{1}{s}}} e^{\rho_1 \pd^{\frac{1}{\theta}}} u)\|_{H^{m-m'}} \\ &&+ \|  e^{\rho_2 \px^{\frac{1}{s}}} \circ \tilde{r}_1(x,D) \circ e^{-\rho_2 \px^{\frac{1}{s}}}(e^{\rho_2 \px^{\frac{1}{s}}} e^{\rho_1 \pd^{\frac{1}{\theta}}} u) \|_{H^{m-m'}}
\\
&&+ \|  \tilde{r}_2(x,D) (e^{\rho_2 \px^{\frac{1}{s}}} e^{\rho_1 \pd^{\frac{1}{\theta}}} u) \|_{H^{m-m'}}
\\
&\leq & C \| e^{\rho_2 \px^{\frac{1}{s}}} e^{\rho_1 \pd^{\frac{1}{\theta}}} u \|_{H^{m}}  = C\| u\|_{H^{m}_{\rho,s, \theta}}.
\end{eqnarray*}

\end{proof}

\begin{Rem}
From Theorem \ref{doubleconj} it follows that any operator with symbol $p \in \textrm{\textbf{SG}}^{m'}_{1,\nu}(\R^{2n})$ with $\nu >1$ extends to a linear continuous map from $H^m_{\rho,s, \theta}$ into $H^{m-m'}_{\rho,s, \theta}$ for every $\rho, m \in \R^2$ and $\min\{s, \theta\}>\nu$. This can be obtained regarding $p$ as a symbol in $\textrm{\textbf{SG}}^{m'}_{\mu,\nu}(\R^{2n})$ for every $\mu>1$ and applying Theorem \ref{doubleconj}. The condition $\min\{s, \theta\} >\mu+\nu-1$ in the theorem then reduces to $\min\{s, \theta\}>\nu$ by choosing $\mu$ arbitrarily close to $1$.
\end{Rem}

We report now the proofs of Propositions \ref{continforder}, \ref{compwithrev} and Theorem \ref{conjugationthm}.
\\

\textit{Proof of Proposition \ref{continforder}.}
We can write, for any $\delta \in \R:$
$$e^\lambda(x,D) u(x) = e^{\delta\px^{\frac{1}{s}}}a(x,D)u(x),$$
where  $a(x,\xi)= e^{-\delta \px^{\frac{1}{s}}+\lambda(x,\xi)}.$ It is easy to verify that if $\delta >C(\lambda)=\sup_{(x,\xi)\in\R^{2n}}\lambda(x,\xi)/\px^{1/s},$ then $a \in \textbf{\textrm{SG}}^{(0,0)}_{\nu}(\R^{2n}),$ and then by Theorem \ref{doubleconj} $a(x,D)$ is bounded on $H^m_{\rho, s, \theta}$ for every $m,\rho \in \R^2$ and for every $\min \{s,\theta\} >2\nu-1$.
Hence we have
$$\| e^\lambda(x,D)u\|_{H^m_{\rho-\delta e_2,s, \theta}} = \| a(x,D) u\|_{H^m_{\rho, s, \theta}} \leq C \| u\|_{H^m_{\rho, s, \theta}}$$
for every $u \in H^m_{\rho, s, \theta}.$ The assertion is then proved.
\qed \\

\textit{Proof of Proposition \ref{compwithrev}.} 
We shall prove the result only for the first composition, the second being similar.
By Proposition \ref{appaelambda} the symbol $e^{\lambda(x,\xi)} \in \textbf{\textrm{SG}}^{0,\infty}_{\mu,s}(\R^{2n})$ whereas the operator  $^R \hskip-1pt e^{- \lambda}$ can be regarded as an operator with amplitude in $\Pi_{\mu,s}^{0, \infty}(\R^{3n})$, hence, by Proposition \ref{reduction}, Theorem \ref{composition} and Remark \ref{refcomp}, we have: 
$$e^{\lambda}(x,D) \circ  ^R \hskip-1pt e^{-\lambda} = I + r_1(x,D)+\tilde{r}_1(x,D), $$ where $\tilde{r}_1 \in \Sigma_s(\R^{2n})$ and $r_1$ has the asymptotic expansion (cf. Definition \ref{equiv}) $r_1\sim \sum\limits_{j \geq 1} r_{1,j}$ with  
\begin{eqnarray} \label{svilid+rem}
r_{1,j} (x,\xi) &=& \sum_{|\alpha|+k = j} \sum_{|\beta|=k}(\alpha!\beta!)^{-1} \partial_\xi^\alpha e^{\lambda(x,\xi)} \cdot D_x^{\alpha +\beta} \partial_\xi^\beta e^{-\lambda(x,\xi)}
\\  \nonumber &=& \sum_{|\gamma|=j} (\gamma!)^{-1}\partial_\xi^\gamma \left[ e^{\lambda(x,\xi)} D_x^\gamma e^{-\lambda(x,\xi)} \right].
\end{eqnarray}
By the results of the calculus, the symbol $r_1(x,\xi)$ turns out to be in $\textbf{\textrm{SG}}^{0,\infty}_{\mu,s}(\R^{2n}).$ To conclude the proof, we need to prove that indeed $r_1$ has finite orders, namely $r_1 \in \textbf{\textrm{SG}}_{\mu}^{(0, -1+1/s)}(\R^{2n}).$
Now, using Fa\`a di Bruno formula, we have $$
e^{\lambda(x,\xi)} D_x^\gamma e^{-\lambda(x,\xi)} = \sum_{h=1}^{|\gamma|} \frac{(-1)^{|\gamma|}}{h!} \sum_{\stackrel{\gamma_1+\ldots +\gamma_h =\gamma}{\gamma_i \neq 0}} \frac{\gamma!}{\gamma_1! \ldots \gamma_h!} \prod_{\kappa=1}^h D_x^{\gamma_\kappa} \lambda(x,\xi) .$$
Hence, Leibniz formula gives
\begin{eqnarray*} 
r_{1,j}(x,\xi) &=& \sum_{|\gamma|=j}\sum_{h=1}^{|\gamma|} \frac{(-1)^{|\gamma|}}{h!} \sum_{\stackrel{\gamma_1+\ldots +\gamma_h =\gamma}{\gamma_i \neq 0}} \sum_{\gamma'_1+\ldots +\gamma'_h =\gamma}\frac{\gamma!}{\gamma_1! \ldots \gamma_h!\gamma'_1! \ldots \gamma'_h!} \\
&& \times \partial_\xi^{\gamma'_1}D_x^{\gamma_1} \lambda(x,\xi)\cdot \ldots \cdot \partial_\xi^{\gamma'_h}D_x^{\gamma_h} \lambda(x,\xi),
\end{eqnarray*} 
from which it follows that
\begin{eqnarray*}
|r_{1,j}(x,\xi)| &\leq &  \sum_{|\gamma|=j} \sum_{h=1}^{|\gamma|} \frac{\gamma!}{h!} \sum_{\stackrel{\gamma_1+\ldots +\gamma_h =\gamma}{\gamma_i \neq 0}} \sum_{\gamma'_1+\ldots +\gamma'_h =\gamma} C^{2|\gamma|}
(\gamma_1! \ldots \gamma_h! \gamma'_1! \ldots \gamma'_h!)^{\mu-1} \px^{\frac{h}{s}-|\gamma|} \pxi^{-|\gamma|} \\ &\leq & C_1^{2j+1}(j!)^{2\mu-1}\px^{-\left(1-\frac{1}{s} \right)j} \pxi_h^{-j}.
\end{eqnarray*}
Similarly we can estimate the derivatives of $r_{1,j}$ and we obtain that 
$$|\partial_\xi^\gamma \partial_x^\delta r_{1,j}(x,\xi) | \leq C^{|\gamma|+|\delta|+2j+1} (\gamma!\delta!)^{\mu}(j!)^{2\mu-1} \pxi^{-j-|\gamma|} \px^{-\left(1-\frac{1}{s} \right)j-|\delta|}, \qquad j \geq 1,$$
for some positive constant $C$ independent of $j, \gamma, \delta.$
\\
Moreover, we observe that $r_{1,0}=0$ in the asymptotic expansion of $r_1$. Then we can repeat readily the argument of the proof of Proposition \ref{sviluppo} by replacing $Q_{Bj^{2\mu-1}}$ by the set
$$Q^s_{Bj^{2\mu-1}}= \{ (x,\xi) \in \R^{2n}: \px^{1-\frac{1}{s}} < Bj^{2\mu-1} \quad \textit{and} \quad \pxi < Bj^{2\mu-1} \}$$ and $Q_{Bj^{2\mu-1}}^e$ by 
$Q_{Bj^{2\mu-1}}^{s,e}= \R^{2n} \setminus Q^s_{Bj^{2\mu-1}}.$
We obtain in this way that $r_1$ is in fact in $\textbf{\textrm{SG}}_{\mu}^{(0,1/s-1)}(\R^{2n})$. Deatils are left to the reader. 
\qed
\\

\textit{Proof of Theorem \ref{conjugationthm}.}
First of all we observe that 
$$(e^{\lambda}(x,D))^{-1} = ^R \hskip-1pt e^{-\lambda} \sum_{j \geq 0 }(-r_1(x,D))^j= ^R \hskip-1pt e^{-\lambda} \circ (I -r_{1,1}(x,D)+s(x,D)),$$
where $s \in \textbf{\textrm{SG}}^{(0,2(1/s-1))}(\R^{2n}).$ 
By Theorem \ref{composition} we have that the symbol of $e^{\lambda}(x,D) a(x,D) ^R \hskip-1pt e^{-\lambda}$ has the asymptotic expansion $\sum_{j \geq 0} a_j(x,\xi)$ where
\begin{eqnarray*}
a_j(x,\xi) &=& \sum_{j \geq 0} \sum_{|\alpha|+h+k=j} \sum_{|\gamma|=h} \sum_{|\delta|=k} (\alpha! \gamma! \delta!)^{-1}\partial_\xi^{\alpha}\left[\partial_\xi^\gamma e^{\lambda(x,\xi)} \cdot D_x^\gamma a(x,\xi) \right]  \partial_\xi^\delta D_x^{\alpha+\delta} e^{-\lambda(x,\xi)} \\
&=& \sum_{|\beta|+h=j} (\beta!)^{-1} \sum_{|\gamma|=h}\partial_\xi^\beta \left[ (\gamma!)^{-1} \partial_\xi^\gamma e^{\lambda(x,\xi)}D_x^\gamma a(x,\xi) D_x^\beta e^{-\lambda(x,\xi)}\right]. 
\end{eqnarray*}
Observe that $a_0=a$ and $a_1(x,\xi)= q(x,\xi)-\sum_{\ell=1}^n a(x,\xi)(\partial_{\xi_\ell}D_{x_\ell} \lambda)(x,\xi)$, with $q$ as in \eqref{qexpression}. 
Using Fa\`a di Bruno formula  and Leibniz formula and arguing as in the proof of Proposition \ref{compwithrev} it is easy to prove that for $j \geq 2$:
$$|\partial_\xi^\gamma \partial_x^\delta a_j(x,\xi)| \leq C^{|\gamma|+|\delta+2j} (\gamma!\delta!)^\nu (j!)^{2\nu-1} \pxi^{m_1-|\gamma|-j} \px^{m_2-\left(1-\frac{1}{s} \right)j-|\delta|}$$
for some positive constant $C$.
Let us now consider the composition of $e^{\lambda}(x,D) a(x,D) ^R \hskip-1pt e^{-\lambda}(x,D)$ with $-r_{1,1}(x,D)$. It is easy to notice that the leading term in $-r_{1,1}(x,\xi)$ is $\sum_{\ell=1}^n \partial_{\xi_\ell}D_{x_\ell} \lambda(x,\xi).$ Then it is easy to verify that 
$$e^{\lambda}(x,D) a(x,D) ^R \hskip-1pt e^{-\lambda}(x,D) \circ r_{1,1}(x,D) = \textrm{op} \left( \sum_{\ell=1}^n a(x,\xi)(\partial_{\xi_\ell}D_{x_\ell} \lambda)(x,\xi)\right) + r(x,D)$$
with $r \in \textbf{\textrm{SG}}^{(m_1-2,m_2+2(1/s-1))}(\R^{2n}).$
Summing up, the symbol of  $e^{\lambda}(x,D) a(x,D) (e^{\lambda}(x,D))^{-1} $ is given, modulo terms in $\textbf{\textrm{SG}}^{\left(m_1 -2, m_2-2(1-1/s)\right) }(\R^{2n}),$ by $a(x,\xi)+q(x,\xi)$ with $q$ as in \eqref{qexpression}.
\qed

\vskip1cm

\noindent
\textbf{Acknowledgements.}
The authors would like to thank the referee for his remarks and suggestions thanks to which the quality of the paper has considerably improved. \\
The authors have been supported in the preparation of the paper by
 INdAM-GNAMPA grants 
2015, 2016 and 2017.



\begin{thebibliography}{AAA}
\bibitem{CT2}
A.~Abdeljawad, M.~Cappiello, J.~Toft, {\it Pseudo-differential calculus in anisotropic Gelfand-Shilov setting}, preprint (2018), http://arxiv.org/abs/1805.03497 .
\bibitem{AB2} A.~Ascanelli, C.~Boiti, 
{\em Cauchy problem for higher order $p$-evolution equations}, J. Differential Equations \textbf{255} (2013), 2672-2708.
\bibitem{AB} A.~Ascanelli, C.~Boiti, {\em Well-posedness of the Cauchy problem for $p$-evolution
systems of pseudo-differential operators}, J. Pseudo-Differ. Oper. Appl. {\bf 4} (2013), 113--143.
\bibitem{ABbis}A.~Ascanelli, C.~Boiti, {\sl Semilinear p-evolution equations in Sobolev spaces}, J. Differential Equations {\bf 260} (2016), 7563-7605.
\bibitem{ABZ} A.~Ascanelli, C.~Boiti, L.~Zanghirati, 
{\em Well-posedness of the Cauchy problem for $p$-evolution equations}, J. Differential Equations {\bf 253} (2012), 2765--2795.
\bibitem{ABZ3} A.~Ascanelli, C.~Boiti, L.~Zanghirati, 
{\em A Necessary condition for $H^\infty $Well-Posedness of p-evolution equations} Advances in Differential Equations {\bf 21} (2016) n.12, 1165--1196.
\bibitem{AsCa1} A.~Ascanelli, M.~Cappiello, {\em Log-Lipschitz regularity for SG hyperbolic systems}, J.Differential Equations {\bf 230} (2006), 556-578.
\bibitem{AsCa2} A.~Ascanelli, M.~Cappiello, {\em H\"older continuity in time for SG hyperbolic systems}, J. Differential Equations {\bf 244} (2008), 2091-2121.
\bibitem{scncpp4} A.~Ascanelli, M.~Cappiello, {\em Weighted energy estimates for p-evolution equations in SG classes}, Journal of Evolution Equations {\bf 15} (2015) n.3, 583-607.
\bibitem{ACic} A.~Ascanelli, M.~Cicognani: {\em Gevrey solutions for a vibrating beam equation}, Rend. Sem. Mat. Univ. Pol. Torino 67 (2009), n.2, 151-164. 
\bibitem{ACC} A.~Ascanelli, M.~Cicognani, F.~Colombini, 
{\em The global Cauchy problem for a vibrating beam equation}, 
J. Differential Equations {\bf 247} (2009), 1440-1451.
\bibitem{CPP1} M.~Cappiello, {\it Gelfand spaces and pseudodifferential operators of infinite order in $\R^n$}, Ann. Univ Ferrara, Sez. VII, Sc. Mat., \textbf{48} (2002),  75-97.
\bibitem{CPP2}
M.~Cappiello, {\it Pseudodifferential parametrices of infinite order for SG-hyperbolic problems}, Rend. Sem. Mat. Univ. Pol. Torino, \textbf{61} n. 4 (2003),  411-441.
\bibitem{CPP3}
M.~Cappiello, \textit{Fourier integral operators of infinite order and applications to SG-hyperbolic equations}, Tsukuba J. Math. \textbf{28} (2004) n. 2, 311--361.
\bibitem{CR}
M.~Cappiello, L.~Rodino, \textit{SG-pseudodifferential operators and Gelfand-Shilov spaces}, Rocky Mountain J. Math. \textbf{36} (2006), n. 4, 1118--1148.
\bibitem{CW}
E.~Carypis, P.~Wahlberg, \textit{Propagation of exponential phase space singularities for Schr\"odinger equations with quadratic Hamiltonians}, J. Fourier Anal. Appl. \textbf{23} (2017) n. 3, 530-571.
\bibitem{CCK}
J.~Chung, S.Y.~Chung, D.~Kim, \textit{Characterization of the Gelfand-Shilov spaces via Fourier transforms}. Proc. Am. Math. Soc. \textbf{124} (1996) n. 7, 2101--2108
\bibitem{CC} M.~Cicognani, F.~Colombini, \textit{The Cauchy problem for $p-$evolution equations}, Trans.Amer.Math.Soc. \textbf{362} (2010) n.9, 4853-4869.
\bibitem{CRe1} 
M.~Cicognani, M.~Reissig, \textit{Well-posedness for degenerate Schr\"odinger equations}, Evol. Equ. Control Theory {\bf 3} (2014), no. 1, 15-33.
\bibitem{CReisraele} 
M.~Cicognani, M.~Reissig, {\it Some remarks on Gevrey well-posedness for degenerate Schr\"odinger equations}. Complex analysis and dynamical systems VI. Part 1, 81-91, Contemp. Math., 653, Israel Math. Conf. Proc., Amer. Math. Soc., Providence, RI, 2015. 
\bibitem{CN}
E.~Cordero, F.~Nicola, {\it Exponentially sparse representation of Fourier integral operators}, Rev. Mat. Iberoamericana, {\bf 31} (2015) n. 2, 461-476.
\bibitem{Co}
H.O.~Cordes, {\em The technique of pseudo-differential operators}, Cambridge Univ. Press, 1995.
\bibitem{D} S.~Doi, {\em On the Cauchy problem for Schr\"odinger type 
equations and the regularity of solutions}, J. Math. Kyoto Univ. 
{\bf 34} (1994) n.2, 319-328. 
\bibitem{GS} I.M.~Gelfand, G.E.~Shilov, {\it Generalized functions, Vol. 2}, Academic
Press, New York-London, 1967.
\bibitem{I1} W.~Ichinose, 
{\em Some remarks on the Cauchy problem for Schr\" odinger type equations}
Osaka J. Math. {\bf 21} (1984), 565-581.
\bibitem{I2} W.~Ichinose, {\em Sufficient condition on $H^{\infty }$ 
well-posedness for Schr\"odinger type equations},
Comm. Partial Differential Equations {\bf 9} (1984) n.1, 33-48. 
\bibitem{Ivrii}
V. Ya.~Ivrii, \textit{Conditions for correctness in Gevrey classes of the Cauchy problem
for weakly hyperbolic equations}, Sib. Mat. Zh., {\bf 17} (1976), 536-547; Sib. Math. J.,
{\bf 17} (1976), 422--435.
\bibitem{K}
K.~Kajitani, \textit{Smoothing effect in Gevrey classes for Schr\"odinger equations. Investigations on the structure of solutions to partial differential equations} (Japanese) (Kyoto, 1998). S$\bar{u}$rikaisekikenky$\bar{u}$sho K$\bar{o}$ky$\bar{o}$roku No. 1056 (1998), 46-58.
\bibitem{KB} 
K.~Kajitani, A.~ Baba, {\em The Cauchy problem 
for Schr\" odinger type equations}, Bull. Sci. Math. {\bf 119} (1995), 459-473.
\bibitem{KN}
K.~Kajitani, T.~Nishitani, \textit{The hyperbolic Cauchy problem}, Lecture Notes in Mathematics \textbf{1505} Springer-Verlag, Berlin, 1991.
\bibitem{KT}
K.~Kajitani, G.~Taglialatela, {\it Microlocal smoothing effect for Schr\"odinger equations in Gevrey spaces}, J. Math. Soc. Japan {\bf 55} (2003), 855-896.
\bibitem{Mit}
B.S. Mitjagin, {\it Nuclearity and other properties of spaces of type S}, Amer. Math. Soc. Transl., Ser. 2 \textbf{93} (1970), 45-59.
\bibitem{M2} S.~Mizohata, {\em On the Cauchy problem.} Notes and 
Reports in Mathematics in Science and Engineering, {\bf 3}, 
Academic Press, Inc., Orlando, FL; Science Press, Beijing, 1985.
\bibitem{Mizu}
R.~Mizuhara, {\it Microlocal smoothing effect for the Schr\"odinger evolution equation in a Gevrey class}, J. Math. Pures Appl. \textbf{91} (2009) n. 2, 115-136.
\bibitem{Pa}
C.~Parenti, {\em Operatori pseudodifferenziali in $\mathbb{R}^n$ e applicazioni}, Ann.
Mat. Pura Appl. {\bf 93} (1972), 359--389.
\bibitem{Pi}
S.~Pilipovi\'c, \textit{Tempered ultradistributions}, Boll. Unione Mat. Ital., VII. Ser., B,  {\bf 2} (1988) n. 2, 235--251.
\bibitem{PPV}
S. Pilipovi\'c, B. Prangoski, J. Vindas \textit{On quasianalytic classes of Gelfand-Shilov type. Parametrix and convolution}, J. Math. Pures Appl. (9) \textbf{116} (2018), 174-210. 
\bibitem{Pr} 
B.~Prangoski, \textit{Pseudodifferential operators of infinite order in spaces of
tempered ultradistributions}, J. Pseudo-Differ. Oper. Appl. \textbf{4} (2013), 495-549.
\bibitem{Rodino}
L.~Rodino, \textit{Linear partial differential operators in Gevrey classes}, World Scientific Publishing Co., Inc., River Edge, NJ, 1993. 
\end{thebibliography}
\end{document}